\documentclass[11pt]{amsart}
\usepackage[margin=1in]{geometry}
\usepackage{amsmath,amssymb,amsthm,mathtools,booktabs,enumitem,array}
\usepackage{xcolor}
\usepackage[colorlinks,linkcolor=blue!50!black,citecolor=blue!50!black,urlcolor=blue!50!black]{hyperref}

\theoremstyle{plain}
\newtheorem{theorem}{Theorem}[section]
\newtheorem{proposition}[theorem]{Proposition}
\newtheorem{lemma}[theorem]{Lemma}
\newtheorem{corollary}[theorem]{Corollary}
\newtheorem{fact}[theorem]{Fact}
\newtheorem{question}[theorem]{Question}

\theoremstyle{definition}
\newtheorem{definition}[theorem]{Definition}
\newtheorem{remark}[theorem]{Remark}

\newtheorem{setup}[theorem]{Setup}

\newcommand{\F}{\mathbb{F}}
\newcommand{\Z}{\mathbb{Z}}

\newcommand{\PP}{\mathbb{P}}
\newcommand{\A}{\mathbb{A}}
\newcommand{\W}{W}
\newcommand{\Ubar}{\overline{U}}
\newcommand{\tw}[2]{{}^{#1}\!#2}
\newcommand{\gen}{\mathrm{gen}}
\newcommand{\tgen}{\tau_{\gen}}
\newcommand{\free}{\mathrm{free}}

\DeclareMathOperator{\ed}{ed}
\newcommand{\edd}[1]{\ed^{[#1]}}
\DeclareMathOperator{\rdim}{rdim}
\DeclareMathOperator{\trdeg}{trdeg}

\DeclareMathOperator{\im}{im}

\DeclareMathOperator{\Aut}{Aut}
\DeclareMathOperator{\Hom}{Hom}
\DeclareMathOperator{\Gal}{Gal}
\DeclareMathOperator{\Ind}{Ind}
\DeclareMathOperator{\Spec}{Spec}
\DeclareMathOperator{\ord}{ord}
\DeclareMathOperator{\PGL}{PGL}

\DeclareMathOperator{\Pic}{Pic}
\DeclareMathOperator{\Div}{Div}
\DeclareMathOperator{\Supp}{Supp}
\DeclareMathOperator{\divv}{div}

\DeclareMathOperator{\id}{id}

\newcommand{\Aref}[1]{[A,\,\ref{A-#1}]}
\newcommand{\Bref}[1]{[A,\,\S\ref{A-ss:ed},\,(B#1)]}

\title[The finite-degree profile of essential dimension]{The finite-degree profile of essential dimension II: cyclic $p$-groups in characteristic $p$}
\author{Abhishek Shukla}
\address{Department of Statistics \& Actuarial Science, Simon Fraser University, Burnaby, BC, Canada}
\email{shukla@sfu.ca}
\date{\today}

\begin{document}

\begin{abstract}
Let $k$ be algebraically closed of characteristic $p$. Reichstein and Vistoli showed that every $\Z/p^n$-torsor becomes at most one-dimensional over a field extension of degree prime to $p$, and Ledet conjectured $\ed_k(\Z/p^n)=n$. Using the finite-degree profile $\edd d_k(\tau)$ and its jump degrees introduced in \cite{PA}, we study the generic $\Z/p^n$-torsor $\tgen$, whose profile records how many Artin--Schreier--Witt layers an extension of degree $\le d$ can absorb. An Abel--Jacobi rigidity argument, in which the purely transcendental total space of $\tgen$ forces divisor classes on twisted curves to be constant, shows that a $G$-stable linear system on a faithful $\Z/p^n$-curve has degree $\ge p^{n-1}$. Hence $\edd d_k(\tgen)\ge2$ for $d<p^{n-1}$ and $d_1(\tgen)=p^{n-1}$ for all $p$ and $n\ge2$; in particular the truncation tower is optimal in degree $p^{n-2}$. We prove a quantitative Reichstein--Vistoli theorem, $d_1^{(p')}(\tgen)\le l_n(p)=1+(p-1)\sum_{j=1}^{n-1}p^{2j-1}$, and show that it is attained: $d_1^{(p')}(\tgen)=l_n(p)$ for all $p$ and $n$. We determine the profile at $(p,n)=(2,2)$ and every slot but the first at $(2,3)$, where the remaining slot is $\ed_k(\Z/8)\in\{2,3\}$, that is, Ledet's conjecture at level $3$. We show that the subgroup formula $\edd{p}_k(\tau_n)=n-1$ is equivalent to Ledet's conjecture at level $n-1$ together with a statement about extensions of degree $\le p$ disjoint from $L$, so that the profile relocates, but does not decide, Ledet's conjecture.
\end{abstract}
\maketitle
\tableofcontents

\section{Introduction}\label{sec:intro}

Let $k$ be an algebraically closed field of characteristic $p>0$ and $G=\Z/p^n$. Reichstein and Vistoli proved that every $G$-torsor over every field containing $k$ becomes at most one-dimensional over a field extension of degree prime to $p$, that is $\ed_{k,p}(G)\le1$ \cite{RV18}; the same holds over any large field \cite[Lemma~9.1]{RS22}. Ledet proved $\ed_k(\Z/p^n)\le n$ and conjectured equality \cite{Ledet04}. Since $\ed_k(\Z/p^n)\ge2$ for $n\ge2$ \cite{Ledet07}, the conjecture holds for $n\le2$ and is open for $n\ge3$ \cite[Rem.~4.2]{BRV18}, \cite[Rem.~9.4]{RS22}. Whatever the truth, the content of the conjecture lies in the passage from the base field $K$ of the generic torsor to a large field, which is an infinite union of finite extensions. In the companion paper \cite{PA} we introduced the \emph{finite-degree profile}
\[
\edd d_k(\tau)=\min\{\ed_k(\tau\otimes_FF'):F'/F\text{ a field extension},\ [F':F]\le d\}
\]
of a torsor $\tau$ over $F$, with jump degrees $d_j(\tau)=\min\{d:\edd d_k(\tau)\le j\}$ and prime-to-$q$ variants $d_j^{(q')}$. Here we study it for the generic $\Z/p^n$-torsor $\tgen$ over $K=k(s_0,\dots,s_{n-1})$, where it records how many Artin--Schreier--Witt layers of the class of $\tgen$ an extension of degree $\le d$ can absorb. In characteristic prime to $|G|$, Sylow divisibility constrains the profile at every prime \cite{PA}; here it says only $p\mid d_0=p^n$ and gives no information on $d_1,\dots,d_{n-1}$. Cohomological invariants and large-field arguments are likewise blind to the $p$-power part: $\mathrm{cd}_p(K)\le1$, and over every large field the whole profile collapses to $\le1$. This is the part of the profile that carries new information.

\subsection*{Results}
The first jump degree. The truncation tower gives $\edd{p^j}_k(\tgen)\le n-j$ (Proposition~\ref{prop:trunc}). Our first main result shows that it is optimal in degree $p^{n-2}$, and indeed that no extension of degree $<p^{n-1}$ makes $\tgen$ one-dimensional.

\begin{theorem}[Theorem~\ref{thm:N1}]\label{thm:introN1}
Let $n\ge2$ and let $F'/K$ be a finite extension with $\ed_k(\tgen\otimes_KF')\le1$. Then $[F':K]\ge p^{n-1}$. Hence $\edd d_k(\tgen)\ge2$ for $d<p^{n-1}$, $d_1(\tgen)=p^{n-1}$ and $d_1^{(p')}(\tgen)\ge p^{n-1}+1$, for all $p$ and $n\ge2$.
\end{theorem}

The case $d=1$, $\ed_k(\Z/p^n)\ge2$ for $n\ge2$, is not new: it follows from the classification of finite groups of essential dimension one \cite{Ledet07,CHKZ}. The content of Theorem~\ref{thm:introN1} is the range $2\le d<p^{n-1}$, where no extension of that degree simplifies the generic torsor to one parameter.

The proof is an Abel--Jacobi rigidity argument. A witness for $\ed\le1$ over $F'$ is a free point of degree $[F':K]$ on the twist of a faithful $G$-curve $C$; over $L=k(z_0,\dots,z_{n-1})$ the twist becomes constant, the point becomes a divisor class on $C$ over $k$, and since $L$ is purely transcendental over $k$ that class is constant \cite[Section~6]{PA}. Hence it is $G$-invariant. If the divisor moves in a linear system of dimension $\ell-1$, $G$ acts on it through $\PGL_\ell(k)$, in which a cyclic $p$-group has order $\le p^{\lceil\log_p\ell\rceil}$ (unipotent matrices); the kernel acts trivially on the image curve, and counting degrees gives $\deg D\ge p^{n-1}$ (Proposition~\ref{prop:rigidN}). When $L\otimes_KF'$ is not a field the argument is applied to the stabiliser of a factor. At $n=2$ the theorem gives $d_1(\tau_2)=p$ for every $p$ (Theorem~\ref{thm:n2} for $p=2$, Corollary~\ref{cor:n2odd} for $p$ odd); at $(2,3)$, $d_1=4$.

A quantitative Reichstein--Vistoli theorem. Riemann--Roch on the twisted Artin--Schreier--Witt curve gives $d_1^{(p')}(\tau)\le l_n(p)=1+(p-1)\sum_{j=1}^{n-1}p^{2j-1}$ for every $\Z/p^n$-torsor $\tau$ over every field containing $k$ (Theorem~\ref{thm:qRV}); here $l_n(p)$ is the last lower ramification break at infinity of the curve. On every smooth projective curve with a faithful $\Z/p^n$-action fixing a point $\infty$, every function regular away from $\infty$ with pole order prime to $p$ has pole order $\ge l_n(p)$ (Proposition~\ref{thm:lambda}). We show that it is sharp for the generic torsor:

\begin{theorem}[Theorem~\ref{thm:ptp}]\label{thm:introptp}
Let $n\ge2$ and let $F'/K$ be a finite extension of degree prime to $p$ with $\ed_k(\tgen\otimes_KF')\le1$. Then $[F':K]\ge l_n(p)$. Hence $d_1^{(p')}(\tgen)=l_n(p)=(p^{2n-1}+1)/(p+1)$ for all $p$ and $n\ge2$.
\end{theorem}

\begin{remark}\label{rem:dependsramif}
The proof of Theorem~\ref{thm:introptp} uses the ramification theory of cyclic covers of $k((x))$ in exactly two places, both quoted: (i) the upper breaks of the Artin--Schreier--Witt cover $\Ubar_n\to\PP^1_t$ of class $[t]$ are $1,p,\dots,p^{n-1}$, which gives the value of $l_n(p)$, the genus of $\Ubar_{n-1}$ and the upper bound (the cover is defined over $\F_p$ and lower breaks do not change under the constant field extension $\F_p((1/t))\subset k((1/t))$, so the formula for a finite residue field suffices); (ii) $u_{j+1}\ge p\,u_j$ for every cyclic $p^n$-extension of $k((x))$, which gives the lower bound for all faithful curves. For $n\le2$ and $(p,n)=(2,3)$, (i) is verified by hand (Theorem~\ref{thm:B}, Appendix~\ref{app:genus}). Nothing else in the proof of the Reichstein--Vistoli bound or of its sharpness uses ramification formulas.
\end{remark}

The proof combines the base-point-freeness of the witness system with the existence, for cyclic $G$ over an algebraically closed field, of a $G$-invariant divisor in an invariant class (Hilbert~90 and $H^2(G,k^\times)=0$). When the degree is prime to $p$ this invariant divisor has a $G$-fixed point of multiplicity prime to $p$, where the last Artin--Schreier layer produces a function with a controlled pole.

The cases $n=2$ and $(p,n)=(2,3)$. At $n=2$ the profile is $(2,1,1,0)$ for $p=2$ (Theorem~\ref{thm:n2}), and $d_1(\tau_2)=p$ and $d_1^{(p')}(\tau_2)=p^2-p+1$ for $p$ odd (Corollary~\ref{cor:n2odd}). At $(p,n)=(2,3)$ every slot of the profile except the first is determined (Theorems~\ref{thm:23main} and~\ref{thm:status}):
\[
\edd1\in\{2,3\},\qquad\edd2=\edd3=2,\qquad\edd4=\dots=\edd7=1,\qquad\edd8=0,
\]
so $d_1=4$ and the truncation tower is optimal in degree $2$; moreover $d_1^{(p')}=11$, and no faithful $\Z/8$-curve has a free point of degree $5$, $7$ or $9$ on its twist. No smooth projective curve with a faithful $\Z/8$-action has a free closed point of degree $\le3$ on its twist. The quartic tower $K(v_1,v_2)$, over which both $s_1$ and $s_2$ become $\wp$-exact, is not a witness for $\edd4=1$ (Theorem~\ref{thm:fam}).

\begin{table}[tbp]
\centering\small
\renewcommand{\arraystretch}{1.25}
\begin{tabular}{@{}l>{\raggedright\arraybackslash}p{0.33\textwidth}>{\raggedright\arraybackslash}p{0.31\textwidth}>{\raggedright\arraybackslash}p{0.17\textwidth}@{}}
\toprule
$(p,n)$ & profile $\edd1,\edd2,\dots,\edd{p^n}$ & jump degrees & $d_1^{(p')}$\\
\midrule
$(2,1)$ & $1,\ 0$ & $d_0=2$ & $1$\\
$(2,2)$ & $2,\ 1,\ 1,\ 0$ & $d_1=2$, $d_0=4$ & $3$\\
$(2,3)$ & $\{2,3\},\ 2,\ 2,\ 1_{[4,7]},\ 0$ & $d_2\in\{1,2\}$, $d_2^{(p')}\in\{1,3\}$, $d_1=4$, $d_0=8$ & $11$\\
$(3,1)$ & $1,\ 1,\ 0$ & $d_0=3$ & $1$\\
$(3,2)$ & $2,\ 2,\ 1_{[3,8]},\ 0$ & $d_1=3$, $d_0=9$ & $7$\\
$(3,3)$ & $\{2,3\},\ \{2,3\},\ 2_{[3,8]},\ 1_{[9,26]},\ 0$ & $d_2\in\{1,2,3\}$, $d_2^{(p')}\le7$, $d_1=9$, $d_0=27$ & $61$\\
$(p,n)$, $n\ge2$ & $\edd d\ge2$ for $d<p^{n-1}$; $\edd{p^j}\le n-j$ & $d_1=p^{n-1}$, $d_0=p^n$ & $l_n(p)$\\
\bottomrule
\end{tabular}
\caption{The profile $\edd d_k(\tgen)$, $d=1,\dots,p^n$, of the generic $\Z/p^n$-torsor in characteristic $p$. A subscript $c_{[a,b]}$ means $\edd d_k=c$ for all $a\le d\le b$; a set such as $\{2,3\}$ means that the value is only known to lie in that set. The jump degrees $d_j=\min\{d:\edd d_k\le j\}$ are taken relative to $k$, and $d_1^{(p')}$ is the first drop of the prime-to-$p$ profile. At $(3,3)$ the only undetermined slots are $\edd1\ge\edd2$.}\label{tab:preview}
\end{table}

Ledet's conjecture. (For the infinitesimal analogue see \cite{TV}, where $\ed\le n$ is shown for trigonalizable group schemes of length $p^n$, and \cite[Conj.~1.4]{To}, which generalises Ledet's conjecture to finite commutative unipotent group schemes: $\ed_kG\ge n_V(G)$.) The first slot $\edd1_k(\tgen)=\ed_k(\Z/8)\in\{2,3\}$ at $(2,3)$ is Ledet's conjecture at level $3$; the profile does not decide it, but it locates it. Unconditionally, faithful curves detect nothing in degrees $\le3$ at $(2,3)$ (Theorem~\ref{thm:23main}(c)). By Proposition~\ref{prop:DX}, $\ed_k(\Z/8)\le2$ iff some smooth projective surface with a faithful $\Z/8$-action has a free $K$-point on its twist, and the profile records this as the single bit $d_2\in\{1,2\}$. For general $n$ we show that the subgroup formula $\edd{p}_k(\tau_n)=n-1$ is equivalent to Ledet's conjecture at level $n-1$ together with a statement about extensions of degree $\le p$ disjoint from $L$ (Theorem~\ref{thm:Pn1}); in particular it implies $\ed_k(\Z/p^{n-1})=n-1$. Hence Theorem~\ref{thm:introN1} settles exactly the slots $\edd{p^j}=n-j$ with $n-j\le2$, and the next open slot is $\ed_k(\Z/p^3)$.

Roadmap. Section~\ref{sec:wittbg} recalls Witt vectors, shows that $\W_n$ with the Witt translation is versal for $\Z/p^n$, and describes the intermediate fields $K_j$. Section~\ref{sec:witt} proves the truncation bound and records that minima over all generically free $G$-varieties of dimension $\le n-1$ carry only the number $d_{n-1}\le p$. Section~\ref{ss:rigid} proves Theorem~\ref{thm:introN1}. Section~\ref{ss:tail} proves the quantitative Reichstein--Vistoli theorem, its optimality among curves with a fixed point, and that it is sharp for the generic torsor (Theorem~\ref{thm:introptp}). Sections~\ref{sec:n2} and~\ref{sec:p2n3} treat $n=2$ and $(2,3)$. Section~\ref{sec:ledet} places the profile relative to Ledet's conjecture: it proves a layer inequality for the profile, the equivalence of Theorem~\ref{thm:Pn1}, and which surfaces would decide the case $n=3$. Section~\ref{sec:questions} lists open questions. The appendices contain Witt-vector formulas at $p=2$, the ramification, genus and Weierstrass semigroup of $\Ubar_n$, and a supplementary script (in the addendum) that checks the explicit identities.


\section{Witt vectors and the generic torsor}\label{sec:wittbg}

Throughout this paper $k$ is algebraically closed of characteristic $p>0$ and $G=\Z/p^n=\langle g\rangle$. Definitions and results on the finite-degree profile are quoted from \cite{PA}; references of the form [A,\,x.y] point to numbered items of \cite{PA}, and the conventions on torsors, twists and large fields are those of \cite[Section~2]{PA}.

\subsection*{Results used from Part~I}
Notation as in \cite{PA}: for a torsor $\tau$ over $F\supseteq k$, $\edd d_k(\tau)=\min\{\ed_k(\tau\otimes F'):[F':F]\le d\}$ and $d_j(\tau)=\min\{d:\edd d_k(\tau)\le j\}$, with prime-to-$p$ variants $\edd d_k{}'$ and $d_j^{(p')}$ (\Aref{def:profile}). We use the following, always with $k$ algebraically closed.
\begin{enumerate}[label=\textup{(P\arabic*)}]
\item $\ed_k(\tau\otimes F')\le\ed_k(\tau)$; $\ed_k(\tau)=0$ iff $\tau$ is split; $\ed_k$ is unchanged by algebraic extension of $k$ (\Bref{1}--\Bref{3}). Purely transcendental base change over an infinite field does not change $\ed_k$, and $A(F')=A(F)$ for an abelian variety $A$ and $F'/F$ purely transcendental (\Aref{lem:constant}).
\item $\edd d_k$ is non-increasing and $\edd1=\ed_k(\tau)$ (\Aref{prop:shape}); $d_0(\tgen)=|G|$, and $\ed_k(\tgen\otimes F')=0$ forces $|G|\mid[F':K]$ (\Aref{prop:d0}(c)).
\item $X$ is $\tau$-weakly versal iff $\tw{\tau}{X}(F)\ne\emptyset$, and $\tw{\tau}{X}(F)=\Hom_G(T,X_F)$ (\Aref{fact:rs}).
\item A free $F$-point of $\tw{\tau}{Z}$ gives $\ed_k(\tau)\le\dim Z$ (\Aref{lem:freepoint}); $\edd d_k(\tau)\le m$ iff some generically free primitive $G$-variety $Z$ with $\dim Z\le m$ has a free closed point of degree $\le d$ on $\tw{\tau}{Z}$ (\Aref{thm:geom}); Section~\ref{ss:recalled} below records the form of this statement used for curves.
\item $\edd d_k$ may be computed over separable extensions (\Aref{lem:pi}); $\ed_k(\Ind_H^G\tau_H)=\ed_k(\tau_H)$ (\Aref{lem:ind}); $\ed_{k,p}(\tau)=\ed_k(\tau\otimes F^{(p)})$ and prime-to-$p$ extensions embed in $F^{(p)}$ (\Aref{lem:qcl}).
\item $\edd d_k(\tau)\le\edd d_k(\tgen)$ for every torsor over every field $F\supseteq k$ (\Aref{prop:vers}); its proof applies to $V=\W_n$ below because only the existence of rational points on twists of dense opens is used.
\item Sylow divisibility (\Aref{thm:div}), the subgroup formula of \cite[Section~\ref{A-sec:sylow}]{PA} and the Karpenko--Merkurjev plateau (\Aref{prop:KMplateau}), quoted for comparison only.
\end{enumerate}

\subsection{Witt vectors and Artin--Schreier--Witt theory}\label{ss:witt}
For an $\F_p$-algebra $A$, $\W_n(A)=A^n$ with the ring structure defined over $\Z$ by the ghost components $w_m(x)=\sum_{i\le m}p^ix_i^{p^{m-i}}$. We use the Frobenius $F(x)=(x_i^p)$, the Verschiebung $V(x)=(0,x_0,\dots,x_{n-2})$, the Teichm\"uller lift $[a]=(a,0,\dots,0)$ and $\wp=F-\id$, an additive endomorphism. Standard facts: $V$ is additive; $FV=VF=p$ on $\W(A)$ when $pA=0$, hence $\wp V=V\wp$; and $x+V(y)=(x_0,\,(x_1,\dots,x_{n-1})+y)$, whence by iteration
\begin{equation}\label{eq:Vsplit}
(x_0,\dots,x_{r-1},0,\dots,0)+V^r(y)=(x_0,\dots,x_{r-1},y_0,\dots,y_{n-r-1}),\qquad y\in\W_{n-r}(A),
\end{equation}
so that every $x=(x_0,\dots,x_{n-1})$ equals $(x_0,\dots,x_{r-1},0,\dots,0)+V^r(x_r,\dots,x_{n-1})$; more generally $x+V^r(y)=(x_0,\dots,x_{r-1},(x_r,\dots,x_{n-1})+y)$, the last sum taken in $\W_{n-r}(A)$. On $\W_1(A)=A$ the Teichm\"uller map is the identity and $\wp(a)=a^p-a$. For $p=2$, $n=3$ the addition and $\wp$ read as follows (these are the Witt polynomials for $p=2$: from the ghost components $w_0=x_0$, $w_1=x_0^2+2x_1$, $w_2=x_0^4+2x_1^2+4x_2$ one gets $(a+b)_1=a_1+b_1-a_0b_0$ over $\Z$, which is the first formula modulo $2$, and the second is obtained in the same way from $w_2$; the identities are checked by expanding, see Appendix~\ref{app:witt}):
\begin{gather}
(a+b)_1=a_1+b_1+a_0b_0,\qquad (a+b)_2=a_2+b_2+a_1b_1+a_0b_0(a_1+b_1)+a_0^3b_0+a_0b_0^3,\label{eq:add2}\\
\wp(a)=\bigl(a_0^2+a_0,\ a_1^2+a_1+a_0^3+a_0^2,\ a_2^2+a_2+C(a_0,a_1)\bigr),\label{eq:wp2}\\
C(a_0,a_1)=a_0^7+a_0^4+a_0^3a_1^2+a_0^3a_1+a_0^2a_1^2+a_0^2a_1+a_1^3+a_1^2.\notag
\end{gather}
Artin--Schreier--Witt theory \cite{Witt37}: for every field $F$ of characteristic $p$,
\[
H^1(F,\Z/p^n)\cong\W_n(F)/\wp\W_n(F),
\]
the class of $a$ being the $G$-torsor $\tau_a=\Spec F[x]/(\wp x=a)$ with $g\cdot x=x+1$, $1=(1,0,\dots,0)$. The truncation $\W_n\to\W_m$ is a ring map commuting with $\wp$ and corresponds to $\Z/p^n\to\Z/p^m$; the map $V^{n-m}\colon\W_m\to\W_n$ is additive, commutes with $\wp$ and corresponds to the inclusion $\Z/p^m\hookrightarrow\Z/p^n$ on cohomology. We shall use repeatedly:
\begin{equation}\label{eq:Vker}
V^{n-1}(c)=(0,\dots,0,c)\in\wp\W_n(F)\iff c\in\wp F .
\end{equation}
(If $(0,\dots,0,c)=\wp(x)$ then $x_0^p=x_0$, so $x=[x_0]+V(x')$ with $\wp[x_0]=0$, hence $(0,\dots,0,c)=V\wp(x')$; induct.) Also $a^{p^r}\equiv a\bmod\wp F$ for $a\in F$.

\subsection{The generic torsor of \texorpdfstring{$\Z/p^n$}{Z/p\^{}n} in characteristic \texorpdfstring{$p$}{p}}\label{ss:gen}
Let $k$ be algebraically closed of characteristic $p$ and $G=\Z/p^n=\langle g\rangle$. $G$ acts on $\W_n=\A^n_k$ (coordinates $z=(z_0,\dots,z_{n-1})$) by the Witt translation $g(z)=z+1$. For $0\le j\le n$,
\begin{equation}\label{eq:gpj}
g^{p^j}(z)=z+p^j\cdot1=z+V^j(1)=\bigl(z_0,\dots,z_{j-1},\,(z_j,\dots,z_{n-1})+1\bigr),
\end{equation}
because $p\cdot1=VF(1)=V(1)$ in $\W_n(\F_p)$ and by the last identity of Section~\ref{ss:witt}; so $g^{p^j}$ fixes $z_0,\dots,z_{j-1}$ and acts on $(z_j,\dots,z_{n-1})$ as the Witt translation of $\W_{n-j}$. The map $\wp\colon\W_n\to\W_n$ is finite \'etale Galois with group $\langle g\rangle$; writing $s=\wp(z)$, $L:=k(z)$, $K:=L^G=k(s_0,\dots,s_{n-1})$, the generic fibre $\tgen\colon\Spec L\to\Spec K$ is the $G$-torsor of class $s\in\W_n(K)$, and $L/K$ is a field extension of degree $p^n$ whose intermediate fields are $K_j:=L^{\langle g^{p^j}\rangle}=k(z_0,\dots,z_{j-1},s_j,\dots,s_{n-1})$, $[K_j:K]=p^j$. When several $n$ are in play we write $\tau_n$, $K_{(n)}$, $L_{(n)}$. The action of $g^{p^{n-1}}$, $z\mapsto(z_0,\dots,z_{n-2},z_{n-1}+1)$, has no fixed points, so $G$ acts freely on $\W_n$.

We let $G$ act on $k[\W_n]=k[z_0,\dots,z_{n-1}]$, and on $L$, by pull-back along the translation, $g(f)=f\circ(z\mapsto z+1)$ (since $G$ is abelian it does not matter whether this is regarded as a left or a right action). Regarding $z$ as an element of $\W_n(k[\W_n])$, and of $\W_n(L)$, we have $g(z)=z+1$.

\begin{lemma}\label{lem:versal}
Let $\tau\colon T=\Spec P\to\Spec F$ be a $G$-torsor over a field $F\supseteq k$.
\begin{enumerate}[label=\textup{(\alph*)}]
\item There is $x\in\W_n(P)$ with $g(x)-x=1$, and for any such $x$ the twist $\tw{\tau}{\W_n}$ is isomorphic to $\A^n_F$: its coordinate ring is $F[h_0,\dots,h_{n-1}]$, where $h=z-x\in\W_n(P\otimes_kk[\W_n])$. On points: for a field $F'\supseteq F$, an $F'$-point of $\tw{\tau}{\W_n}$ is a $G$-equivariant morphism $T_{F'}\to\W_n$, i.e.\ an element $y\in\W_n(P\otimes_FF')$ with $g(y)-y=1$, and it corresponds to $h=y-x\in\W_n(F')=F'^{\,n}$. For $\tau=\tgen$ one may take $x=z$, so that $h=y-z$.
\item $\W_n$ is $\tau$-versal for every $G$-torsor $\tau$ over every field $F\supseteq k$.
\item $\ed_k(G)=\ed_k(\W_n)=\ed_k(\tgen)\le n$.
\end{enumerate}
\end{lemma}
\begin{proof}
(a) By Artin--Schreier--Witt theory (Section~\ref{ss:witt}) $\tau\cong\tau_a$ for some $a\in\W_n(F)$, so $P$ contains $x'\in\W_n(P)$ with $g(x')=x'+1$; put $x=x'$. Since $\W_n$ is affine, $\tw{\tau}{\W_n}=\Spec R$ with $R=(P\otimes_kk[\W_n])^G=P[z_0,\dots,z_{n-1}]^G$, $G$ acting diagonally. In $\W_n(P[z])$ put $h:=z-x$. As $G$ acts on $P[z]$ by ring automorphisms, it acts on $\W_n(P[z])$ by ring automorphisms, and $g(h)=(z+1)-(x+1)=h$; so the coordinates $h_0,\dots,h_{n-1}$ lie in $R$. Since Witt addition and subtraction are given by polynomials with integer coefficients, of the form $(a\pm b)_i=a_i\pm b_i+(\text{polynomial in }a_0,\dots,a_{i-1},b_0,\dots,b_{i-1})$, we have $z_i=h_i+(\text{polynomial in }h_0,\dots,h_{i-1},x_0,\dots,x_i)$ and $h_i=z_i+(\text{polynomial in }z_0,\dots,z_{i-1},x_0,\dots,x_i)$. Hence $P[z]=P[h_0,\dots,h_{n-1}]$ with $h_0,\dots,h_{n-1}$ algebraically independent over $P$, i.e.\ $P[z]=P\otimes_FF[h]$, and $G$ acts on it through the factor $P$ only. Choosing an $F$-basis of $F[h]$ shows $(P\otimes_FF[h])^G=P^G\otimes_FF[h]=F[h]$. So $R=F[h_0,\dots,h_{n-1}]$ and $\tw{\tau}{\W_n}\cong\A^n_F$. An $F'$-point of $\Spec R$ is the specialisation $h\mapsto h(y)\in F'^{\,n}$; via $P[z]=P\otimes_FF[h]$ it corresponds to the $G$-equivariant $P$-algebra map $P[z]\to P\otimes_FF'$, i.e.\ to the equivariant morphism $T_{F'}\to\W_n$ with $y=x+h$, which is the identification of Fact~\Aref{fact:rs}(a). For $\tau=\tgen$, $P=L$ and $x=z\in\W_n(L)$ satisfies $g(x)-x=\varepsilon$.

(b) Let $U\subseteq\W_n$ be a dense $G$-invariant open subset. Then $\tw{\tau}{U}$ is a nonempty open subscheme of $\tw{\tau}{\W_n}\cong\A^n_F$ \cite[Cor.~3.4]{DR15}, and $F\supseteq k$ is infinite, so $\tw{\tau}{U}(F)\ne\emptyset$, and $U$ is $\tau$-weakly versal by Fact~\Aref{fact:rs}(a).

(c) $G$ acts freely on the irreducible variety $\W_n$, so $\W_n$ is generically free and primitive, and by (b) and Fact~\Aref{fact:rs}(c), $\ed_k(\W_n)\ge\ed_k(\tau)$ for every $G$-torsor $\tau$ over every $F\supseteq k$, i.e.\ $\ed_k(\W_n)\ge\ed_k(G)$. By the identity $\ed_k(Y)=\ed_k(\tau_Y)$ of \cite[Section~\ref{A-ss:ed}]{PA}, $\ed_k(\W_n)=\ed_k(\tau_{\W_n})=\ed_k(\tgen)$, since the torsor $\W_n\to\W_n/G$ has generic fibre $\tgen$; and $\ed_k(\tgen)\le\ed_k(G)$ by definition and $\ed_k(\tgen)\le\trdeg_kK=n$ (the bound $\ed_k(\Z/p^n)\le n$ is also in \cite{Ledet04}, and for trigonalizable group schemes in \cite[Thm.~1.4]{TV}).
\end{proof}

\begin{corollary}\label{cor:insep}
For $G=\Z/p^n$ over an algebraically closed $k$ of characteristic $p$, the profile only sees separable extensions: $\edd d_k(\tau)=\min\{\ed_k(\tau\otimes F'):[F':F]\le d,\ F'/F\text{ separable}\}$, and every cyclic extension of degree $p$ of $F$ (for $p=2$: every separable quadratic extension) is an Artin--Schreier extension $F(v_f)$, $v_f^p-v_f=f\in F\smallsetminus\wp F$.
\end{corollary}

\begin{remark}[the Witt-vector proof]\label{rem:insepwitt}
For $G=\Z/p^n$ the isomorphism $\sigma_*(\tau\otimes F^{1/p^m})\cong\tau$ can be seen on Witt classes (Section~\ref{ss:witt}): if $a\in\W_n(F)$ is the class of $\tau$, the class of $\tau\otimes F^{1/p^m}$ is $a$ viewed in $\W_n(F^{1/p^m})$, and by functoriality of the Artin--Schreier--Witt isomorphism the class of $\sigma_*(\tau\otimes F^{1/p^m})$ is $\W_n(\sigma)(a)=F^m(a)$. Since $F(a)-a=\wp(a)$, $F^m(a)\equiv a\bmod\wp\W_n(F)$.
\end{remark}

\begin{proposition}[the intermediate fields of $L/K$ for $G=\Z/p^n$]\label{prop:intermediate}
Let $k$ be algebraically closed of characteristic $p$, $G=\Z/p^n$, $K_j=L^{\langle g^{p^j}\rangle}$, $1\le j\le n-1$. Then $\tgen\otimes K_j$ is induced from the $\Z/p^{n-j}$-torsor $L/K_j$, which is the generic $\Z/p^{n-j}$-torsor over the base field $k':=k(z_0,\dots,z_{j-1})$, and
\[
\ed_k(\tgen\otimes_KK_j)=\ed_k(\Z/p^{n-j})\le n-j .
\]
In particular $\ed_k(\tgen\otimes K_1)\ge\min(n-1,2)$, so for $n=3$ the subextension $K_1=K(z_0)$ is never a witness for $d_1=p$.
\end{proposition}
\begin{proof}
$L\otimes_KK_j\cong\prod_{G/H}L$ with $H=\langle g^{p^j}\rangle$, i.e.\ $\tgen\otimes K_j=\Ind_H^G(L/K_j)$, and $\ed_k(\tgen\otimes K_j)=\ed_k(L/K_j)$ by Lemma~\Aref{lem:ind}. By \eqref{eq:gpj}, $H$ acts on $(z_j,\dots,z_{n-1})$ by the Witt translation of $\W_{n-j}$ and fixes $z_0,\dots,z_{j-1}$, so $L/K_j$ is the $\Z/p^{n-j}$-torsor of class $s':=\wp_{\W_{n-j}}(z_j,\dots,z_{n-1})\in\W_{n-j}(K_j)$. Its coordinates satisfy $s'_i=z_{j+i}^p-z_{j+i}+(\text{polynomial in }z_j,\dots,z_{j+i-1})$, while $s_{j+i}=z_{j+i}^p-z_{j+i}+(\text{polynomial in }z_0,\dots,z_{j+i-1})$; hence $q_i:=s_{j+i}-s'_i$ is a polynomial in $z_0,\dots,z_{j+i-1}$, and it is $H$-invariant. Now $H$ fixes $z_0,\dots,z_{j-1}$ and, through the truncation $\W_{n-j}\to\W_i$, acts on $(z_j,\dots,z_{j+i-1})$ by the Witt translation of $\W_i$, whose quotient map is the finite \'etale Galois map $\wp\colon\W_i\to\W_i$, $(z_j,\dots,z_{j+i-1})\mapsto(s'_0,\dots,s'_{i-1})$; as $\A^i$ is normal, $k[z_j,\dots,z_{j+i-1}]^H=k[s'_0,\dots,s'_{i-1}]$. Hence $q_i\in k[z_0,\dots,z_{j-1}][s'_0,\dots,s'_{i-1}]$. Consequently
\[
K_j=k'(s'_0,\dots,s'_{n-j-1}),\qquad k'=k(z_0,\dots,z_{j-1}),
\]
The $n$ elements $z_0,\dots,z_{j-1},s'_0,\dots,s'_{n-j-1}$ generate $K_j$, which has transcendence degree $n$ over $k$, so they are algebraically independent. Put $F'':=k(s'_0,\dots,s'_{n-j-1})$. Then $s'$ is the class of the generic $\Z/p^{n-j}$-torsor $\tau_{n-j}$ over $F''$, $K_j=F''(z_0,\dots,z_{j-1})$ is purely transcendental over $F''$, and $L/K_j=\tau_{n-j}\otimes_{F''}K_j$. By Lemma~\Aref{lem:constant}(i) and Lemma~\ref{lem:versal}(c), $\ed_k(L/K_j)=\ed_k(\tau_{n-j})=\ed_k(\Z/p^{n-j})\le n-j$.
\end{proof}

\subsection{Free points and curves: results recalled from the general theory}\label{ss:recalled}
The following statements are used below. They are proved here, in the form needed, so that no proof of this paper relies on material of \cite{PA} that is stated there only as a remark. (Theorem~\Aref{thm:geom} and the other numbered results of \cite{PA} are quoted as before.)

\begin{lemma}[image of a free point]\label{lem:imageB}
Let $Z$ be an irreducible $G$-variety with faithful action, $\tau\colon T\to\Spec F$ a torsor with $T$ connected, and $f\colon T\to Z$ equivariant. Then the closure of $f(T)$ is an irreducible $G$-stable closed subset $Y$, and the point $f(T)$ is free iff no $1\ne h\in G$ acts trivially on $Y$; for $G=\Z/p^n$ this holds iff $G$ acts faithfully on $Y$. If $\dim Z=1$, $G\ne1$ and $f(T)$ is free, then $f$ is dominant if $k$ is algebraically closed; for arbitrary $k$ it is dominant unless $\ed_k(\tau)=0$.
\end{lemma}
\begin{proof}
As $T$ is connected and $Z$ separated, $f(T)$ is a single point $\xi$ (the image of the generic point of $T$), and $h\xi=\xi$ for all $h\in G$ by equivariance, so $Y=\overline{\{\xi\}}$ is $G$-stable and irreducible. For $h\ne1$, $\xi\in Z^h$ iff $Y\subseteq Z^h$ iff $h$ acts trivially on $Y$. For $G=\Z/p^n$ all nontrivial subgroups contain $g^{p^{n-1}}$, so $\xi$ is free iff $g^{p^{n-1}}$ is nontrivial on $Y$ iff the action on $Y$ is faithful. If $k$ is algebraically closed, a $G$-stable closed point is fixed by $G$, hence not free when $G\ne1$; so for a curve a free point has $Y=Z$. For arbitrary $k$ the closed point $\xi=Y$ can be free (e.g.\ $k=\mathbb{R}$, $\Z/2$ acting on $\PP^1$ by $z\mapsto-z$, $\xi=(z^2+1)$); then $\xi=\Spec\kappa$ is a $G$-torsor over the finite extension $\kappa^G$ of $k$, the equivariant map $T\to\xi$ identifies $\tau$ with the base change of $\xi\to\Spec\kappa^G$ by (T2), and $\ed_k(\tau)=0$.
\end{proof}

\begin{definition}\label{def:DB}
For a quasi-projective generically free primitive $G$-variety $Z$ and a torsor $\tau$ over $F$, let $D(Z,\tau)\subseteq\Z_{>0}$ be the set of degrees of closed points of $\tw{\tau}{Z}$ lying in the free locus, and $\delta_\free(Z,\tau):=\min D(Z,\tau)$ ($=\infty$ if empty). Write $D(Z)$, $\delta_\free(Z)$ for $\tau=\tgen$. For an affine curve $U$ with smooth projective model $\Ubar$ such that $\Ubar^{\free}=U$ we write $D(U)=D(\Ubar)$.
\end{definition}

\begin{corollary}\label{cor:jumpsB}
For every $G$-torsor $\tau$ and $0\le j\le\ed_k(\tau)$,
\[
d_j(\tau)=\min\{\delta_\free(Z,\tau):Z\text{ quasi-projective generically free primitive},\ \dim Z\le j\}.
\]
For $k$ perfect, $j=1$ and $\tau$ with connected total space $T$ (e.g.\ $\tau=\tgen$) the minimum may be restricted to smooth projective irreducible curves with faithful $G$-action provided $d_1(\tau)$ is attained by an $F'$ with $T\times_FF'$ connected; and then the free point is dominant (Lemma~\ref{lem:imageB}, since $\ed_k(\tau\otimes F')=1\ne0$). More precisely, for every finite $F'/F$ with $T\times_FF'$ connected and $\ed_k(\tau\otimes_FF')=1$ there is a smooth projective irreducible curve with faithful $G$-action whose twist by $\tau\otimes F'$ has a free, dominant $F'$-point. For $\tau=\tgen$ with $G=\Z/p^n$ this holds whenever $F'\not\supseteq K_1=K(z_0)$, in particular whenever $p\nmid[F':K]$.
\end{corollary}
\begin{proof}
The first statement is Theorem~\Aref{thm:geom}. For the second (where $k$ is assumed perfect, so that the smooth projective model of an irreducible curve is a smooth curve over $k$), let $F'$ be as in the last sentence; by the proof of Theorem~\Aref{thm:geom} ((1)$\Rightarrow$(3) with $m=1$, where $T_{F'}$ connected makes $Z$ irreducible with faithful action), there is an irreducible $Z$ with faithful $G$-action, $\dim Z\le1$, whose twist has a free $F'$-point, and $\dim Z=1$ since $\ed_k(\tau\otimes F')=1$ (Lemma~\Aref{lem:freepoint}). So $Z$ is an irreducible faithful $G$-curve, its smooth projective model $\tilde Z$ carries the induced action, the free $F'$-point of $\tw{\tau'}{Z}$ is dominant (Lemma~\ref{lem:imageB}, as $\ed_k(\tau\otimes F')=1\ne0$) hence corresponds to a $G$-equivariant $T_{F'}\to\tilde Z$ whose image is the generic point, free because the action on $\tilde Z$ is faithful (an element acting trivially on the generic point acts trivially on $\tilde Z$); conversely a free point on $\tilde Z$ is dominant, so lands in the open $Z$. For $G=\Z/p^n$ the intermediate fields of $L/K$ are the chain $K\subset K_1\subset\dots\subset L$, so a finite $F'/K$ has $F'\cap L=K$ (i.e.\ $L\otimes_KF'$ a field, $L/K$ being Galois) iff $F'\not\supseteq K_1$.
\end{proof}

\begin{fact}\label{fact:rsd}
If $F$ is a large field and $X$ a generically free $G$-variety with a smooth $G$-fixed $k$-point, then $\ed_k(\tau)\le\ed_k(X)\le\dim X$ for every $G$-torsor $\tau$ over $F$; and if $G$ is a finite $p$-group and $\operatorname{char}k=p$, then $\ed_k(\tau)\le1$ for every $G$-torsor over a large field \cite[Lemma~9.1]{RS22}.
\end{fact}


\section{The truncation tower and a remark on global minima}\label{sec:witt}

This paper treats the cyclic group $G=\Z/p^n=\langle g\rangle$ over an algebraically closed field $k$ of characteristic $p>0$; this is fixed from now on. Recall from Section~\ref{ss:gen} that $G$ acts freely on $\W_n=\A^n_k$ by the Witt translation $z\mapsto z+1$, that $\wp\colon\W_n\to\W_n$ is the quotient map, and that the generic torsor is $\tgen\colon\Spec L\to\Spec K$ with
\[
L=k(z_0,\dots,z_{n-1}),\qquad K=L^G=k(s_0,\dots,s_{n-1}),\qquad s=\wp(z);
\]
its Artin--Schreier--Witt class is $s\in\W_n(K)/\wp\W_n(K)$, and the intermediate fields of $L/K$ are the $K_j=L^{\langle g^{p^j}\rangle}$, $[K_j:K]=p^j$. By Lemma~\ref{lem:versal}, $\ed_k(\tgen)=\ed_k(G)\le n$, and by Proposition~\Aref{prop:vers}, which applies to $V=\W_n$ because $\tw{\tau}{\W_n}\cong\A^n_F$ (Lemma~\ref{lem:versal}(a)), the profile of $\tgen$ dominates the profile of every $G$-torsor over every infinite field containing $k$. By Proposition~\Aref{prop:d0}(c), $d_0(\tgen)=p^n$ and $\edd d_k(\tgen)\ge1$ for $d<p^n$. Since $\ed_{k,q}(G)=0$ for primes $q\ne p$ and $\ed_{k,p}(G)\le1$ (Proposition~\ref{prop:tail}), Sylow divisibility (Theorem~\Aref{thm:div}) says only that $p\mid d_0(\tgen)$; it gives no information on $d_1,\dots,d_{n-1}$. The $p$-power part of the profile is governed instead by the Artin--Schreier--Witt layers of the class $s$. This section develops the truncation tower and a remark on global minima; rigidity of curve witnesses, the prime-to-$p$ tail and the quantitative Reichstein--Vistoli theorem follow in Sections~\ref{ss:rigid} and~\ref{ss:tail}. Throughout this paper, $\edd d$ and $d_j$ are taken relative to $k$, and the prime $q$ of Definition~\Aref{def:profile} is $p$; thus $\edd d_k{}'$ is the prime-to-$p$ profile and $d_j^{(p')}$ its jump degrees.

\subsection{The truncation tower}

\begin{lemma}[truncation tower]\label{lem:tower}
Let $F$ be a field of characteristic $p$, $a\in\W_n(F)$ and $0\le r\le n$. There is a tower $F=F^{(0)}\subseteq F^{(1)}\subseteq\dots\subseteq F^{(n-r)}$ of Artin--Schreier extensions (each of degree $1$ or $p$), hence $[F^{(n-r)}:F]\le p^{n-r}$, such that
\[
a\equiv(a_0,\dots,a_{r-1},0,\dots,0)\bmod\wp\W_n(F^{(n-r)}).
\]
The $i$-th layer is $F^{(i)}=F^{(i-1)}(v_i)$ with $v_i^p-v_i=b_i$, where $b_1=a_r$ and $b_i$ (for $i\ge2$) is the coordinate of index $r+i-1$ of the class obtained at the previous step; in general $b_i\ne a_{r+i-1}$ because of the Witt carries.
\end{lemma}
\begin{proof}
Induction on $n-r$; there is nothing to do if $r=n$. Let $r<n$. By \eqref{eq:Vsplit}, $a=(a_0,\dots,a_{r-1},0,\dots,0)+V^r(a')$ with $a'=(a_r,\dots,a_{n-1})\in\W_{n-r}(F)$, and $a'=[a_r]+V(a'')$. Let $v^p-v=a_r$ in $F^{(1)}=F(v)$. In $\W_{n-r}(F^{(1)})$ the vectors $[a_r]=[v^p-v]$ and $\wp[v]=[v^p]-[v]$ have the same coordinate of index $0$, so $[a_r]=\wp[v]+V(e)$ for some $e\in\W_{n-r-1}(F^{(1)})$. Since $\wp$ commutes with $V^r$, $V^r\wp[v]\in\wp\W_n(F^{(1)})$ and
\[
a\equiv(a_0,\dots,a_{r-1},0,\dots,0)+V^{r+1}(a''+e)=(a_0,\dots,a_{r-1},0,b_2,\dots,b_{n-r})\bmod\wp\W_n(F^{(1)}),
\]
with $b_2,\dots$ the coordinates of $a''+e$ (again \eqref{eq:Vsplit}). Apply the induction hypothesis (with $r+1$) to this vector over $F^{(1)}$.
\end{proof}

\begin{proposition}[truncation bound]\label{prop:trunc}
For $0\le j\le n$: $\ed^{[p^j]}_k(\tgen)\le n-j$. In particular $d_{n-1}(\tgen)\le p$, $d_1(\tgen)\le p^{n-1}$, $d_0(\tgen)\le p^n$. More generally, for every $G$-torsor $\tau_a$ over a field $F\supseteq k$ one has $\ed^{[p^j]}_k(\tau_a)\le\ed_k(\tau_{(a_0,\dots,a_{n-j-1},0,\dots,0)})\le\trdeg_kk(a_0,\dots,a_{n-j-1})$.
\end{proposition}
\begin{proof}
Lemma~\ref{lem:tower} with $r=n-j$: over $F^{(j)}$, of degree $\le p^j$, the class is $(a_0,\dots,a_{n-j-1},0,\dots,0)$, a Witt vector with coordinates in $k(a_0,\dots,a_{n-j-1})$, so the torsor descends to that field. For $\tgen$, $\trdeg_kk(s_0,\dots,s_{n-j-1})=n-j$.
\end{proof}

\begin{remark}\label{rem:towerp2}
For $n=2$ no carry occurs: $s=[s_0]+V(s_1)$ and $V(s_1)=V(\wp v)=\wp V(v)$ over $K(v)$, $v^p-v=s_1$, so $s\equiv[s_0]$ over $K(v)$. For $(p,n)=(2,3)$, with $v_i^2+v_i=s_i$, the formulas \eqref{eq:add2}--\eqref{eq:wp2} give $\wp(0,v_1,0)=(0,s_1,v_1^3+v_1^2)$ and hence
\begin{gather*}
s\equiv(s_0,0,\ c)\ \text{over }K(v_1),\ \ c=s_2+s_1^2+s_1v_1;\qquad s\equiv(s_0,0,\ s_1^2+s_1v_1)\ \text{over }K(v_1,v_2);\\
s\equiv(s_0,0,0)\ \text{over }K(v_1,v'),
\end{gather*}
where $v'^2+v'=c$ (these are the specialisations $x=(0,v_1,0)$, $(0,v_1,v_2)$, $(0,v_1,v')$ of the general translate of Appendix~\ref{app:witt}). The middle class is not $(s_0,0,0)$: $s_1^2+s_1v_1=v_1^4+v_1^3\equiv v_1^3+v_1\bmod\wp K(v_1,v_2)$, and $v_1^3+v_1\notin\wp\,k(s_0,v_2)(v_1)$, since it has a pole of odd order $3$ at $v_1=\infty$ while an element $w^2+w$ has a pole of even order or none there. So the quartic extension $K(v_1,v_2)$ does not kill the two upper layers of the class, although both $s_1$ and $s_2$ become $\wp$-exact over it; in fact $\ed_k(\tgen\otimes K(v_1,v_2))=2$ (Theorem~\ref{thm:fam}(c)), so $K(v_1,v_2)$ is not a witness for $\edd4_k(\tgen)\le1$, while the tower $K(v_1,v')$ of Lemma~\ref{lem:tower} is.
\end{remark}

The truncation tower removes the top Witt coordinates by field extensions of $p$-power degree. Conversely, a descent of the truncation over any extension lifts to a descent of the full class at the cost of one parameter, and this bounds both profiles.

\begin{lemma}[lifting one Witt layer]\label{lem:lift}
Let $F\supseteq k$ be a field, $n\ge2$, $a\in\W_n(F)$, and $\bar a\in\W_{n-1}(F)$ its truncation, the class of the $\Z/p^{n-1}$-torsor $\tau_{\bar a}$ induced by $\Z/p^n\to\Z/p^{n-1}$. For every field extension $F'/F$,
\[
\ed_k(\tau_a\otimes F')\le\ed_k(\tau_{\bar a}\otimes F')+1 .
\]
Consequently $\edd d_k(\tau_a)\le\edd d_k(\tau_{\bar a})+1$ and $\ed^{[d]'}_k(\tau_a)\le\ed^{[d]'}_k(\tau_{\bar a})+1$ for every $d\ge1$. For the generic torsors $\tau_n$ over $K=k(s_0,\dots,s_{n-1})$ and $\tau_{n-1}$ over $k(s_0,\dots,s_{n-2})$ this gives $\edd d_k(\tau_n)\le\edd d_k(\tau_{n-1})+1$ and $\ed^{[d]'}_k(\tau_n)\le\ed^{[d]'}_k(\tau_{n-1})+1$, hence
\[
d_{j+1}(\tau_n)\le d_j(\tau_{n-1}),\qquad d_{j+1}^{(p')}(\tau_n)\le d_j^{(p')}(\tau_{n-1})\qquad(0\le j\le n-1).
\]
\end{lemma}
\begin{proof}
Let $e=\ed_k(\tau_{\bar a}\otimes F')$: there are a field $k\subseteq F_0\subseteq F'$ with $\trdeg_kF_0=e$, $b\in\W_{n-1}(F_0)$ and $y\in\W_{n-1}(F')$ with $\bar a=b+\wp y$. Let $\tilde y=(y,0)\in\W_n(F')$. Truncation $\W_n\to\W_{n-1}$ is a ring map commuting with $\wp$, so $a-\wp\tilde y$ has truncation $b$, i.e.\ $a-\wp\tilde y=(b_0,\dots,b_{n-2},c)$ for some $c\in F'$. This vector lies in $\W_n(F_0(c))$ and $\trdeg_kF_0(c)\le e+1$, so $\ed_k(\tau_a\otimes F')\le e+1$. Taking the minimum over $F'/F$ of degree $\le d$ (and prime to $p$) gives the statements on profiles. For the generic torsors, the truncation of $\tau_n$ is $\tau_{n-1}\otimes K$, and an extension $F'_0$ of degree $d$ of $k(s_0,\dots,s_{n-2})$ gives the extension $F'_0\otimes K=F'_0(s_{n-1})$ of $K$ of the same degree, over which a descent of $\tau_{n-1}\otimes F'_0$ remains a descent; so $\edd d_k(\tau_{n-1}\otimes K)\le\edd d_k(\tau_{n-1})$, and likewise for the prime-to-$p$ profiles. The inequalities for the jump degrees follow.
\end{proof}

\begin{remark}\label{rem:lift}
(i) At $(2,3)$, Lemma~\ref{lem:lift} with $j=1$ and Theorem~\ref{thm:n2}(d) give $d_2^{(p')}(\tgen)\le d_1^{(p')}(\tau_2)=3$; this is the upper bound in Theorem~\ref{thm:23main}(d), made explicit there. With $j=0$ it gives $d_1(\tgen)\le d_0(\tau_2)=4$, the bound of Proposition~\ref{prop:trunc}.
(ii) At $(3,3)$ it gives $d_2^{(p')}(\tgen)\le d_1^{(p')}(\tau_2)\le p^2-p+1=7$ (Proposition~\ref{prop:n2p}); as the prime-to-$p$ profile can only change at degrees prime to $3$, $d_2^{(p')}(\tgen)\le7$.
(iii) The lemma does not interact with Ledet's conjecture: it bounds $\ed_k(\tau_n)$ by $\ed_k(\tau_{n-1})+1$, which is the bound $\ed_k(\Z/p^n)\le n$ again.
\end{remark}

\subsection{A remark on global minima}\label{ss:collapse}

Any invariant of $\tgen$ of the form ``minimum, over all generically free $G$-varieties $Z$ of dimension $\le n-1$, of the degree of a free point of $\tw{\tgen}{Z}$'' is dominated by its value on the following hypersurface.

\begin{definition}\label{def:S0}
For $n\ge2$ let $S_0:=\{z\in\W_n:(\wp z)_{n-1}=0\}=\wp^{-1}(\{s_{n-1}=0\})\subset\W_n$.
\end{definition}

\begin{proposition}\label{prop:S0}
Let $n\ge2$.
\begin{enumerate}[label=\textup{(\alph*)}]
\item $S_0$ is an irreducible $G$-stable hypersurface of dimension $n-1$ on which $G$ acts freely. Under the identification $\tw{\tgen}{\W_n}\cong\A^n_K$, $y\mapsto h=y-z$, of Lemma~\ref{lem:versal},
\[
\tw{\tgen}{S_0}=\{h\in\A^n_K:(s+\wp h)_{n-1}=0\}.
\]
\item $\tw{\tgen}{S_0}$ has a free closed point of degree $p$, namely $h=(0,\dots,0,v)$ with $v^p-v=-s_{n-1}$, and it has no $K$-point unless $\ed_k(\tgen)\le n-1$. Hence $d_{n-1}(\tgen)=\min_Z\delta_\free(Z)\le p$, the minimum taken over quasi-projective generically free primitive $G$-varieties $Z$ of dimension $\le n-1$; at $p=2$ this minimum is $1$ if $\ed_k(\Z/2^n)\le n-1$ and $2$ otherwise.
\end{enumerate}
\end{proposition}
\begin{proof}
(a) $S_0$ is $G$-stable because $\wp$ is $G$-invariant. The restriction of the $G$-torsor $\wp\colon\W_n\to\W_n$ to the hyperplane $\{s_{n-1}=0\}\cong\A^{n-1}$ is a $G$-torsor $S_0\to\A^{n-1}$ whose generic fibre is the Artin--Schreier--Witt torsor of class $a:=(s_0,\dots,s_{n-2},0)\in\W_n(K_{(n-1)})$, $K_{(n-1)}=k(s_0,\dots,s_{n-2})$. So $S_0$ has dimension $n-1$, $G$ acts freely on it, and it is irreducible iff this generic fibre is a field, i.e.\ iff $a$ has order $p^n$ in $\W_n(K_{(n-1)})/\wp$. Now $p^{n-1}a=V^{n-1}F^{n-1}a=(0,\dots,0,s_0^{p^{n-1}})$, which by \eqref{eq:Vker} is $\equiv0$ iff $s_0^{p^{n-1}}\equiv s_0$ lies in $\wp K_{(n-1)}$; this fails, since over $k(s_1,\dots,s_{n-2})$ the element $s_0$ has a simple pole at $s_0=\infty$, whereas an element $w^p-w$ has a pole of order divisible by $p$ there or none. By Lemma~\ref{lem:versal}(a), a $K'$-point of $\tw{\tgen}{\W_n}$ ($K'\supseteq K$) is a $y\in\W_n(L\otimes_KK')$ with $g(y)-y=1$, and $h:=y-z$ is then $G$-invariant, i.e.\ $h\in\W_n(K')$; the point lies in $\tw{\tgen}{S_0}$ iff $(\wp(z+h))_{n-1}=(s+\wp h)_{n-1}=0$.

(b) By \eqref{eq:Vsplit} and $\wp V^{n-1}=V^{n-1}\wp$: $s+\wp V^{n-1}(v)=s+V^{n-1}(v^p-v)=(s_0,\dots,s_{n-2},0)$. The residue field $K(v)$ has degree $p$, because $-s_{n-1}\notin\wp K$ (simple pole at $s_{n-1}=\infty$). As $G$ acts freely on $S_0$, every point of $\tw{\tgen}{S_0}$ is free. A free $K$-point on the $(n-1)$-dimensional $S_0$ gives $\ed_k(\tgen)\le n-1$ by Lemma~\Aref{lem:freepoint}. The formula $d_{n-1}=\min_Z\delta_\free(Z)$ is Corollary~\ref{cor:jumpsB}, and the point of degree $p$ gives $\le p$; at $p=2$, $d_{n-1}=1$ iff $\ed_k(\tgen)\le n-1$.
\end{proof}

\begin{remark}
By Proposition~\ref{prop:S0}(b), the minimum over all $G$-varieties of dimension $\le n-1$ of the degree of a free point on the twist equals $d_{n-1}(\tgen)\le p$: it carries at most one bit when $p=2$. Graded information on the profile therefore has to come from the degree of the base change, via curves (Section~\ref{ss:rigid}), or from varieties one at a time (the sets $D(Z)$, Proposition~\ref{prop:DX}).
\end{remark}


\section{Curve witnesses and the first jump degree}\label{ss:rigid}

By Proposition~\ref{prop:S0}(b), minima over all $G$-varieties carry at most one bit. The profile escapes it because its slots are indexed by the degree of the base change, and at $j=1$ that degree is the degree of a divisor on a curve. This subsection turns that remark into a lower bound valid for every $(p,n)$. The mechanism is short: a hypothetical witness for $\ed_k(\tgen\otimes F')\le1$ is a point of degree $[F':K]$ on the twist of a faithful $G$-curve $C$; over $L$ the twist becomes $C_L$, the point becomes a divisor whose class is an $L$-point of $\Pic_{C/k}$, and since $L$ is purely transcendental over $k$ this class is constant. The constant class is $G$-invariant, and a $G$-invariant pencil of small degree is impossible because $\PGL_2(k)$ contains no element of order $p^2$. Throughout this subsection $n\ge2$.

\begin{fact}[curves and Picard schemes]\label{fact:pic}
Let $C$ be a smooth projective curve over $k$ of genus $g(C)$, and $D$ a divisor on $C$; $\mathcal L(D)=H^0(C,\mathcal O_C(D))$, $\ell(D)=\dim_k\mathcal L(D)$.
\begin{enumerate}[label=\textup{(\alph*)}]
\item $\Pic_{C/k}$ is represented by a $k$-group scheme whose components $\Pic^d_{C/k}$ are torsors under the Jacobian $J=\Pic^0_{C/k}$, an abelian variety of dimension $g(C)$; since $C(k)\ne\emptyset$, $\Pic(C_F)=\Pic_{C/k}(F)$ for every field $F\supseteq k$, compatibly with degrees and with pullback by automorphisms of $C$ \cite[8.1, Prop.~4; 8.2, Thm.~3; 8.4, Thm.~3]{BLR90}, \cite[Thm.~9.2.5]{Kl05}; $\Pic^0_{C/k}$ is an abelian variety of dimension $g(C)$ by \cite[8.4, Thm.~1]{BLR90}.
\item For a field $F\supseteq k$: $\mathcal L(D_F)=\mathcal L(D)\otimes_kF$ (flat base change), so $\ell(D_F)=\ell(D)$.
\item (Riemann--Roch) $\ell(D)-\ell(K_C-D)=\deg D+1-g(C)$ \cite[IV.1.3]{Hart77}. If $g(C)\ge1$ and $\deg D\ge1$ then $\ell(D)\le\deg D$ \cite[IV, Ex.~1.5]{Hart77}; if $g(C)=1$ and $\deg D\ge1$ then $\ell(D)=\deg D$.
\item (Clifford) If $D$ is effective and special ($\ell(K_C-D)>0$) then $\ell(D)-1\le\frac12\deg D$ \cite[IV.5.4]{Hart77}.
\end{enumerate}
\end{fact}

\begin{lemma}[two facts on curves]\label{lem:curvesB}
\textup{(i)} $\PGL_2(\bar k)$ has no element of order $p^2$ if $\operatorname{char}k=p$; in particular $\PGL_2(k)$ has none. \textup{(ii)} Let $p=2$ and $E$ a supersingular elliptic curve over $k$ algebraically closed. Then $\Aut(E)$ (all automorphisms of the curve) has no element of order $8$. \textup{(iii)} An irreducible curve dominated by $\A^m_k$ has smooth projective model $\PP^1$. \textup{(iv)} A nonconstant morphism of smooth projective curves $E\to C$ over a perfect field satisfies $g(C)\le g(E)$.
\end{lemma}
\begin{proof}
(i) A $p$-element of $\PGL_2(\bar k)$ lifts to $A\in\mathrm{GL}_2(\bar k)$ with $A^{p^m}=1$ after scaling, so $(A-1)^{p^m}=0$, $A=1+N$ with $N^2=0$, and $A^p=1+pN=1$. (ii) $\Aut(E)=E(k)\rtimes\Aut(E,O)$, and $\Aut(E,O)$ is a group of order $24$ with Sylow $2$-subgroup $Q_8$ \cite[Thm.~III.10.1 and App.~A, Prop.~1.2]{Silverman09}, so its elements of $2$-power order have order $\le4$. Let $\phi=t_P\circ\alpha$ with $\alpha\in\Aut(E,O)$, $t_P$ translation by $P$, and $\phi^8=1$; then $\alpha^8=1$, so $m:=\ord\alpha\in\{1,2,4\}$. If $m=1$, $\ord\phi=\ord P$, which is odd because $E(k)[2]=0$ for a supersingular curve in characteristic $2$. If $m\in\{2,4\}$: $\phi^m=t_Q$ with $Q=(1+\alpha+\dots+\alpha^{m-1})P$, and $(\alpha-1)(1+\alpha+\dots+\alpha^{m-1})=\alpha^m-1=0$ in the domain $\mathrm{End}(E)$ with $\alpha-1\ne0$, so $Q=0$ and $\ord\phi\mid m\le4$. (iii) is L\"uroth's theorem, valid in every characteristic (restrict the dominant map to the curve $t\mapsto(t,t^N,t^{N^2},\dots,t^{N^{m-1}})$ with $N$ large, which is not contained in a fibre, to get a nonconstant $\A^1\to C$). (iv) Factor $E\to C$ as a purely inseparable map (a power of Frobenius composed with an isomorphism, genus preserving) followed by a separable map, to which Riemann--Hurwitz applies.
\end{proof}

\begin{lemma}[the twist is trivialised by $L$]\label{lem:twistL}
Let $C$ be a quasi-projective $G$-variety over $k$ and $X:=\tw{\tgen}{C}$. There is an $L$-isomorphism $\psi\colon X\otimes_KL\xrightarrow{\sim}C\otimes_kL$ such that
\begin{enumerate}[label=\textup{(\alph*)}]
\item for every $h\in G$, $\psi$ intertwines the automorphism $\id_X\times\Spec(h)$ of $X\otimes_KL$ with $\phi_h:=h_C\times\Spec(h)$ ($\psi$ is $L$-linear, while $\phi_h$ is semilinear over $\Spec(h)$);
\item for every field $F'\supseteq K$ and $P\in X(F')$, the composite $\mathrm{pr}_1\circ\psi\circ(P\times_K\id_{\Spec L})\colon\Spec(F'\otimes_KL)\to C$ is the $G$-equivariant morphism attached to $P$ by Fact~\Aref{fact:rs}(a).
\end{enumerate}
\end{lemma}
\begin{proof}
Write $T=\Spec L$. Since $L/K$ is Galois with group $G$, $T\times_KT\cong G\times T$ via $(t_1,t_2)\mapsto(\gamma,t_2)$ with $t_1=\gamma t_2$. Hence $X\times_KT=((T\times_kC)/G)\times_KT=(T\times_KT\times_kC)/G\cong(G\times T\times_kC)/G\cong T\times_kC$, the quotient being by the diagonal action on the first and last factors (the twist commutes with base change, \cite[\S3]{DR15}). For (a), the automorphism $\id_X\times\Spec(h)$ of $X\times_KT$ acts on the second tensor factor $T$ only. Under $T\times_KT\cong G\times T$, $(t_1,t_2)\mapsto(\gamma,t_2)$ with $t_1=\gamma t_2$, replacing $t_2$ by $ht_2$ changes $\gamma$ to $\gamma h^{-1}$ (as $t_1=\gamma h^{-1}(ht_2)$). On $(G\times T\times_kC)/G$ the class of $(\gamma,t,c)$ is therefore sent to that of $(\gamma h^{-1},ht,c)=(\gamma,ht,hc)$, using the diagonal $G$-action; this is $h_C\times\Spec(h)$ on $T\times_kC$ under the last identification, which is $\phi_h$. For (b), an $F'$-point of $X$ is by Fact~\Aref{fact:rs}(a) a $G$-equivariant $x\colon T_{F'}\to C$, and its base change to $T$ is, under the displayed isomorphism, the section $(x,\id)$ of $T_{F'}\times_kC\to T_{F'}$; its first component is $x$.
\end{proof}

\begin{setup}\label{setup:witness}
This is the setting of \Aref{prop:witness}; nothing in it, in Lemma~\ref{lem:twistL}, or in parts (b),(c) of Proposition~\ref{prop:rigid} uses that $G=\Z/p^n$ or that the characteristic is $p$: it holds for any finite group $G$ and any $G$-Galois extension $L/K$ with $L$ purely transcendental over $k$, with $\tau=\Spec L$. We write it in the notation of this paper, with $\tgen=\Spec L$. Let $F'/K$ be a finite field extension of degree $d$ such that $L\otimes_KF'$ is a field and $\ed_k(\tgen\otimes_KF')\le1$. By \Bref{3}, $\ed_k(\tgen\otimes F')=1$ (a torsor whose total space is a field is not split), so by Theorem~\Aref{thm:geom} and Corollary~\ref{cor:jumpsB} there are a smooth projective curve $C$ over $k$ with a faithful $G$-action and a free point $P\in\tw{\tgen}{C}(F')$; the attached equivariant morphism $x\colon\Spec(L\otimes_KF')\to C$ is dominant (Lemma~\ref{lem:imageB}). Let $\xi$ be the closed point of $X=\tw{\tgen}{C}$ underlying $P$, of degree $e=[\kappa(\xi):K]$, which divides $d$ as $\kappa(\xi)\subseteq F'$. Put
\[
D_P:=\tfrac de\,\xi\quad(\text{an effective $K$-rational divisor of degree }d\text{ on }X),\qquad \Delta:=\psi(D_P\otimes_KL)\ \text{on }C_L .
\]
Since $\kappa(\xi)\otimes_KL\subseteq F'\otimes_KL$ is a field, $\xi\otimes_KL$ is a single closed point $\eta$ of $X_L$, and $\Delta=\frac de\,\psi(\eta)$.
\end{setup}

\begin{proposition}[rigidity lemma]\label{prop:rigid}
In Setup~\ref{setup:witness}:
\begin{enumerate}[label=\textup{(\alph*)}]
\item $g(C)\ge1$.
\item There is an effective divisor $D_0$ on $C$ of degree $d$ with $\Delta\sim D_0\otimes_kL$ on $C_L$ and $h^*D_0\sim D_0$ for every $h\in G$.
\item $\ell(D_0)\ge2$.
\end{enumerate}
\end{proposition}
\begin{proof}
(a) If $C\cong\PP^1$ then $G\cong\Z/p^n$ embeds in $\Aut(\PP^1)=\PGL_2(k)$, which has no element of order $p^2$ (Lemma~\ref{lem:curvesB}(i)); as $n\ge2$ this is impossible.

(b) By Fact~\ref{fact:pic}(a) the class of $\Delta$ is an $L$-point of $\Pic^d_{C/k}$, a torsor under $J$; as $k$ is algebraically closed it has a $k$-point and is isomorphic to $J$ as a $k$-variety. Since $L=k(z_0,\dots,z_{n-1})$ is purely transcendental over $k$, Lemma~\Aref{lem:constant}(ii), applied to the abelian variety $J$, gives $J(L)=J(k)$; so $[\Delta]$ is the image of a class $c_0\in\Pic^d(C)$. As $\Delta$ is effective, $\ell(c_0\otimes L)\ge1$, hence $\ell(c_0)\ge1$ by Fact~\ref{fact:pic}(b), and we may choose $D_0\in c_0$ effective. The divisor $D_P\otimes_KL$ is invariant under $\id\times\Spec(g)$, so by Lemma~\ref{lem:twistL}(a) $\phi_g(\Delta)=\Delta$. Since $D_0$ is defined over $k$, $\phi_g^*(D_0\otimes L)=(g^*D_0)\otimes L$; hence $(g^*D_0)\otimes L\sim\Delta\sim D_0\otimes L$, and as $\Pic(C)=\Pic_{C/k}(k)\to\Pic_{C/k}(L)$ is injective, $g^*D_0\sim D_0$. Iterate for the powers of $g$.

(c) Suppose $\ell(D_0)=1$. Then $\ell(D_0\otimes L)=1$ by Fact~\ref{fact:pic}(b), so $D_0\otimes L$ is the only effective divisor in its class, and $\Delta=D_0\otimes L$. Thus $\psi(\eta)$ lies over a point of the finite set $\Supp(D_0)\subseteq C(k)$. By Lemma~\ref{lem:twistL}(b), $x$ factors through $\Spec(F'\otimes_KL)\to\eta\to X_L\xrightarrow{\psi}C_L\to C$, so its image is a closed point of $C$, contradicting the dominance of $x$.
\end{proof}

\begin{remark}\label{rem:rigid-d1}
For $d=1$ the proposition reproves $\ed_k(\Z/p^n)\ge2$ for $n\ge2$: by (a) and Fact~\ref{fact:pic}(c), $\ell(D_0)\le\deg D_0=1$, against (c). The same argument at $n=2$ is Proposition~\ref{prop:G1}. The only property of $\tgen$ used is that $L$ is purely transcendental over $k$; the only property of $G$ used is that it is cyclic of order $p^n$ with $n\ge2$ (through Lemma~\ref{lem:curvesB}(i)). In particular nothing about the Witt coordinates of the class $s$ enters.
\end{remark}

\begin{remark}
Parts (b) and (c) are the general case of \Aref{lem:AJ} and \Aref{prop:witness}, where the same constancy statement is proved by an embedding-theoretic argument; we use the twist formulation here because Step~2 of Theorem~\ref{thm:ptp} needs that $\Delta$ is supported on a single closed point of $C_L$ of degree $e\ge2$ over $L$. Part (a) uses $G=\Z/p^n$, $n\ge2$ (through Lemma~\ref{lem:curvesB}(i)).
\end{remark}

\begin{proposition}[rigidity lemma, $\PGL_\ell$ form]\label{prop:rigidN}
Let $m\ge2$, let $H\cong\Z/p^m$ act faithfully on a smooth projective curve $C$ over $k$, and let $D_0$ be an effective divisor on $C$ with $\ell:=\ell(D_0)\ge2$ and $h^*D_0\sim D_0$ for all $h\in H$. Then $\deg D_0\ge p^{m-1}$.
\end{proposition}
\begin{proof}
Let $B$ be the base locus of $|D_0|$ and $D_1=D_0-B$, $e'=\deg D_1\le\deg D_0$, $\mathcal L=\mathcal L(D_1)=\mathcal L(D_0)$, so $\dim\mathcal L=\ell$. For $h\in H$, $h^*$ preserves $|D_0|$, hence $h^*B=B$ and $h^*D_1=D_1+\divv(\varphi_h)$ with $\varphi_h\in k(C)^\times$. Then $A_h\colon f\mapsto\varphi_hh^*f$ is a linear automorphism of $\mathcal L$, well defined up to a scalar; this gives a homomorphism $\rho\colon H\to\PGL(\mathcal L)=\PGL_\ell(k)$ ($H$ is abelian, so the contravariance of $h\mapsto h^*$ is harmless).

The image. Let $a=\lceil\log_p\ell\rceil$, so $p^{a-1}<\ell\le p^a$. If $\rho(h)$ has order $p^r$, lift it to $A\in\operatorname{GL}(\mathcal L)$ with $A^{p^r}=\lambda\cdot\id$ and choose $\mu\in k$ with $\mu^{p^r}=\lambda$ ($k$ algebraically closed). Then $B_1=\mu^{-1}A$ satisfies $B_1^{p^r}=\id$, so $(B_1-1)^{p^r}=0$ in characteristic $p$ and $B_1$ is unipotent; hence $(B_1-1)^\ell=0$ and $B_1^{p^a}=1+(B_1-1)^{p^a}=1$. So the cyclic group $\rho(H)$ has order $\le p^a$, and $H_0:=\ker\rho$ has $|H_0|\ge\max(1,p^{m-a})$.

The kernel. Let $f_0,\dots,f_{\ell-1}$ be a basis of $\mathcal L$. Elements of $H_0$ act on $\mathcal L$ as scalars up to the factor $\varphi_h$, so they fix every ratio $f_i/f_0$; thus $H_0$ fixes pointwise the subfield $k(\Gamma)\subseteq k(C)$ generated by these ratios, where $\varphi\colon C\to\Gamma\subseteq\PP^{\ell-1}$ is the morphism defined by the base-point-free system $|D_1|$. Let $\nu\colon\tilde\Gamma\to\Gamma$ be the normalisation and $E=\nu^*\mathcal O(1)$ (so $\deg E=\deg\Gamma$); $\varphi$ factors as $C\xrightarrow{\varphi'}\tilde\Gamma$ of degree $\delta=[k(C):k(\Gamma)]$ and $D_1\sim\varphi'^*E$, whence $e'=\delta\deg E$. The $\ell$ coordinate sections are linearly independent in $H^0(\tilde\Gamma,\mathcal O(E))$ since $\Gamma$ is non-degenerate, so $\ell\le\ell(E)\le\deg E+1$ (Fact~\ref{fact:pic}(c) and its genus-$0$ analogue). Since $\ell\ge2$ we have $\deg E\ge1$ (as $\ell(E)\ge\ell\ge2$). As $H_0$ acts faithfully on $k(C)$ and fixes $k(\Gamma)$, Artin's theorem gives $|H_0|=[k(C):k(C)^{H_0}]$, which divides $\delta=[k(C):k(\Gamma)]$. Therefore
\[
e'=\delta\deg E\ \ge\ |H_0|(\ell-1)\ \ge\ \max(1,p^{m-a})\,p^{a-1}\ \ge\ p^{m-1}
\]
(for $m\ge a$ this is $p^{m-a}p^{a-1}$; for $m<a$ it is $\ell-1\ge p^{a-1}\ge p^m$). Hence $\deg D_0\ge e'\ge p^{m-1}$.
\end{proof}

\begin{theorem}[the first jump degree of $\Z/p^n$]\label{thm:N1}
Let $n\ge2$ and let $F'/K$ be a finite field extension with $\ed_k(\tgen\otimes_KF')\le1$. Then $[F':K]\ge p^{n-1}$. Consequently:
\begin{enumerate}[label=\textup{(\alph*)}]
\item $\edd d_k(\tgen)\ge2$ for every $d<p^{n-1}$; with Proposition~\ref{prop:trunc}, $d_1(\tgen)=p^{n-1}$ for all $p$ and all $n\ge2$.
\item $d_1^{(p')}(\tgen)\ge p^{n-1}+1$.
\item $\edd{p^j}_k(\tgen)=n-j$ for $j=n-2$ \textup{(}all $n\ge2$\textup{)}; in particular for $j=1$ and $n=3$, for every $p$.
\item $\edd d_k(\tau_3)=1$ for $p^2\le d<p^3$ and $=2$ for $p\le d<p^2$, and $\edd d_k(\tau_3)\in\{2,3\}$ for $d<p$.
\end{enumerate}
\end{theorem}
\begin{proof}
By Lemma~\Aref{lem:pi} we may replace $F'$ by its separable closure of $K$, which only decreases the degree. So let $F'/K$ be separable. Choose an embedding of $F'$ into an algebraic closure containing $L$ and let $N_0=LF'$, $H=\Gal(N_0/F')\subseteq G$. The intermediate fields of $L/K$ are the $K_i=L^{\langle g^{p^i}\rangle}$, so $H=\langle g^{p^i}\rangle$ for some $0\le i\le n$, $L\cap F'=K_i$, $F'\supseteq K_i$ and $[F':K]=p^i[F':K_i]$. If $i\ge n-1$ then $[F':K]\ge p^{n-1}$ and we are done. Let $i\le n-2$, $m=n-i\ge2$, so $H\cong\Z/p^m$.

Now $L\otimes_KF'=\prod_{G/H}N_0$ with $N_0=L\otimes_{K_i}F'$ a field, and $\tgen\otimes F'=\Ind_H^G(N_0/F')$, where $N_0/F'$ is the base change of the $H$-torsor $L/K_i$. By Lemma~\Aref{lem:ind}, $\ed_k(N_0/F')=\ed_k(\tgen\otimes F')\le1$, and it equals $1$ because $N_0$ is a field. The statements of Setup~\ref{setup:witness}, Lemma~\ref{lem:twistL} and Proposition~\ref{prop:rigid}(b),(c) use only that the torsor is a Galois extension $L/K_i$ with cyclic group $H\cong\Z/p^m$ acting faithfully and that $L$ is purely transcendental over $k$ \textup{(}the genus statement~\textup{(a)} is not needed\textup{)}. So they apply with $(H,K_i,F')$ in place of $(G,K,F')$ and give a smooth projective curve $C$ with a faithful $H$-action and an effective divisor $D_0$ of degree $d'':=[F':K_i]$ with $\ell(D_0)\ge2$ and $h^*D_0\sim D_0$ for all $h\in H$. Proposition~\ref{prop:rigidN} gives $d''\ge p^{m-1}$, hence $[F':K]=p^id''\ge p^{n-1}$.

(a) follows from this and Proposition~\ref{prop:trunc} ($\edd{p^{n-1}}\le1$). (b): a degree prime to $p$ is not $p^{n-1}$ for $n\ge2$. (c): $\edd{p^{n-2}}\le2$ is Proposition~\ref{prop:trunc} and $\ge2$ is~(a). (d): the slots with $d\ge p^2$ and $p\le d<p^2$ are Proposition~\ref{prop:trunc}, Proposition~\Aref{prop:d0} and (a); for $d<p$, $\edd d\le\ed_k(\Z/p^3)\le3$ and (a) gives $\ge2$.
\end{proof}

\begin{corollary}[free points on faithful curves]\label{cor:curvepts}
Let $n\ge2$ and let $C$ be a smooth projective curve with a faithful $\Z/p^n$-action. Every closed point of the free locus of $\tw{\tgen}{C}$ has degree $\ge p^{n-1}$; that is, $D(C)\cap[1,p^{n-1}-1]=\emptyset$ (Definition~\ref{def:DB}).
\end{corollary}
\begin{proof}
A free closed point $\xi$ of degree $e$ has residue field $F'=\kappa(\xi)$ with $[F':K]=e$, and $\ed_k(\tgen\otimes F')\le1$ by Lemma~\Aref{lem:freepoint}. Apply Theorem~\ref{thm:N1}.
\end{proof}

\begin{remark}[sharpness of Theorem~\ref{thm:N1}]\label{rem:N1sharp}
The bound is sharp: $\edd{p^{n-1}}_k(\tgen)\le1$ by Proposition~\ref{prop:trunc}. For example at $(2,2)$, $\edd2_k(\tau_2)=1$ (Theorem~\ref{thm:n2}); at $(3,2)$, $\edd3_k(\tau_2)=1$ (Proposition~\ref{prop:n2p}); and at $(2,3)$, $d_1(\tgen)=4$ (Section~\ref{sec:p2n3}). Theorem~\ref{thm:N1} concerns extensions of degree $<p^{n-1}$; the prime-to-$p$ degrees above $p^{n-1}$ are treated in Theorem~\ref{thm:ptp}, which gives $d_1^{(p')}(\tgen)=l_n(p)$. Which slots remain open, and how they relate to Ledet's conjecture, is discussed in Remark~\ref{rem:N1ledet}.
\end{remark}


\section{The prime-to-\texorpdfstring{$p$}{p} tail and a quantitative Reichstein--Vistoli theorem}\label{ss:tail}

\begin{proposition}[the prime-to-$p$ tail]\label{prop:tail}
Let $G$ be a finite $p$-group and $\tau$ a $G$-torsor over a field $F\supseteq k$. Recall the prime-to-$p$ profile $\ed^{[d]'}_k(\tau)=\min\{\ed_k(\tau\otimes F'):[F':F]\le d,\ p\nmid[F':F]\}$ (Definition~\Aref{def:profile}), so that $\edd d_k\le\ed^{[d]'}_k$.
\begin{enumerate}[label=\textup{(\alph*)}]
\item $\ed_{k,p}(\tau)=\min_d\ed^{[d]'}_k(\tau)$, and $\ed_k(\tau\otimes F')\ge\ed_{k,p}(\tau)$ for every finite extension $F'/F$ of degree prime to $p$. Thus $\ed_{k,p}(\tau)$ is the terminal value of the prime-to-$p$ profile, and $\min_{p\nmid d}\edd d_k(\tau)\le\ed_{k,p}(\tau)$.
\item $\ed_{k,p}(\tau)\le1$; hence $\ed^{[d]'}_k(\tau)\le1$ for all large $d$, i.e.\ $d_1^{(p')}(\tau)<\infty$. If $G\ne1$ and $\tau=\tgen$, the prime-to-$p$ profile never reaches $0$, so $\ed^{[d]'}_k(\tgen)=1$ for every $d\ge d_1^{(p')}(\tgen)$.
\item For $G=\Z/p^n$ the witness in \textup{(b)} can be taken on the Artin--Schreier--Witt curve: there are a finite extension $F'/F$ of degree prime to $p$ and $t\in F'$ with $\tau\otimes F'\cong\tau_{[t]}$ (Teichm\"uller form).
\end{enumerate}
\end{proposition}
\begin{proof}
(a) Put $e:=\ed_{k,p}(\tau)=\ed_k(\tau\otimes F^{(p)})$ (Lemma~\Aref{lem:qcl}). A descent of $\tau\otimes F^{(p)}$ to a subfield of transcendence degree $e$ and the isomorphism realising it involve finitely many elements of $F^{(p)}$, all lying in a finite subextension $F_1\subset F^{(p)}$, which has degree prime to $p$ (Lemma~\Aref{lem:qcl}); so $\ed_k(\tau\otimes F_1)\le e$ and $\ed^{[d]'}_k(\tau)\le e$ for $d=[F_1:F]$. Conversely, let $F'/F$ be finite with $p\nmid[F':F]$. Its inseparable degree is a power of $p$ dividing $[F':F]$, so $F'/F$ is separable, and Lemma~\Aref{lem:qcl} gives an $F$-embedding $F'\hookrightarrow F^{(p)}$ and hence $\ed_k(\tau\otimes F')\ge\ed_k(\tau\otimes F^{(p)})=e$ by \Bref{1}.

(b) The field $F^{(p)}$ is large (Fact~\Aref{fact:rs}), so $\ed_k(\tau\otimes F^{(p)})\le1$ by Fact~\ref{fact:rsd}; the first claim follows from (a). For $\tgen$ with $G\ne1$: $\ed_k(\tgen\otimes F')=0$ iff $\tgen\otimes F'$ is split (by \Bref{3}, $k$ being algebraically closed), iff $L$ embeds in $F'$ over $K$ (Proposition~\Aref{prop:d0}(a)), which forces $|G|\mid[F':K]$; so $\ed_k(\tgen\otimes F')\ge1$ whenever $p\nmid[F':K]$.

(c) Let $\Ubar_n$ be the smooth projective model of the Artin--Schreier--Witt curve $U_n=\{y\in\W_n:\wp(y)\in[\A^1]\}$. The map $U_n\to\A^1$, $y\mapsto t=y_0^p-y_0$, is a connected $G$-Galois cover: its generic fibre is the torsor of class $[t]\in\W_n(k(t))$, which has order $p^n$, since $p^{n-1}[t]=V^{n-1}F^{n-1}[t]=(0,\dots,0,t^{p^{n-1}})$ is $\not\equiv0$ by \eqref{eq:Vker} ($t^{p^{n-1}}\equiv t\notin\wp k(t)$). It is \'etale over $\A^1$ and totally ramified over $\infty$ (as $G$ is abelian, the inertia groups of the points over $\infty$ coincide, and the quotient of $\Ubar_n$ by this common inertia group is a connected cover of $\PP^1$ unramified at $\infty$ and \'etale over $\A^1$, hence trivial since $\PP^1$ is simply connected), so $\Ubar_n\smallsetminus U_n$ is a single smooth $G$-fixed $k$-point $\infty$, and $G$ acts freely on $U_n$; thus $\Ubar_n^{\free}=U_n$. The constant map to $\infty$ gives a smooth $F$-point $\infty'\in\tw{\tau}{\Ubar_n}(F)$. Over the large field $F^{(p)}$ the rational points of $\tw{\tau}{\Ubar_n}$ are dense, so one of them lies in the dense open $\tw{\tau}{U_n}$; it is defined over a finite subextension $F'\subset F^{(p)}$, of degree prime to $p$. Write $T=\Spec P$. By Lemma~\ref{lem:versal}(a), a free $F'$-point is an equivariant morphism $T_{F'}\to U_n$, i.e.\ an element $y=x+h$ of $\W_n(P\otimes_FF')$ with $h\in\W_n(F')$ and $\wp(y)=[t]$, $t\in F'$, where $x\in\W_n(P)$ is fixed with $g(x)-x=1$; so $\wp(x)+\wp(h)=[t]$, and $\wp(x)\in\W_n(F)$ is a class of $\tau$ (Section~\ref{ss:gen}).
\end{proof}

\begin{remark}\label{rem:tailcaveat}
For $p\nmid d$, $\edd d_k$ minimises over all extensions of degree $\le d$, including some of degree divisible by $p$; this is why (a) is stated with the prime-to-$p$ profile, and why $d_1^{(p')}$ is defined through it (Definition~\Aref{def:profile}). For $\tgen$ the number $d^{(p')}_1(\tgen)$, the least degree prime to $p$ of an extension over which $\tgen$ becomes one-dimensional, is the quantitative form of the theorem $\ed_{k,p}(\Z/p^n)\le1$ of \cite{RV18}, which says only that $d_1^{(p')}(\tgen)<\infty$.
\end{remark}

\subsection{How finite: a Riemann--Roch bound and the quantitative Reichstein--Vistoli theorem}\label{ss:qRV}

Proposition~\ref{prop:tail}(b) says that $d_1^{(p')}(\tgen)<\infty$; the proof, through the large field $K^{(p)}$, gives no bound. Riemann--Roch on the twisted Artin--Schreier--Witt curve does, and the optimal bound obtainable in this way is a ramification invariant. We first fix the ramification notation.

Let $C$ be a smooth projective curve over $k$ with a faithful action of $G=\Z/p^n$ fixing a point $\infty$. As $k$ is algebraically closed, the inertia group at $\infty$ is $G$, so $\infty$ is totally ramified in $C\to C/G$. Let $G=G_0\supseteq G_1\supseteq\dots$ be the lower ramification groups at $\infty$ \cite[IV, \S1]{Se2}. Since $G$ is a $p$-group, $G_1=G$; since $G_i/G_{i+1}$ embeds in the additive group of $k$ for $i\ge1$ \cite[IV, \S2, Prop.~7, Cor.~1]{Se2}
and $G$ is cyclic, each jump of the filtration has index $p$. So there are exactly $n$ lower breaks $1\le l_1<\dots<l_n$, with $G_i$ of order $p^{n-j+1}$ for $l_{j-1}<i\le l_j$ ($l_0:=-1$). The upper breaks are $u_j:=\varphi(l_j)$, $\varphi$ the Herbrand function \cite[IV, \S3]{Se2}; explicitly
\begin{equation}\label{eq:herbrand}
u_1=l_1,\qquad l_{j+1}-l_j=p^j(u_{j+1}-u_j)\quad(1\le j<n).
\end{equation}
We write $l(C,\infty):=l_n$ for the last lower break. It is the lower break of $\sigma:=g^{p^{n-1}}$: $v_\infty(\sigma\pi-\pi)=l_n+1$ for a uniformiser $\pi$ at $\infty$; equivalently, as lower numbering is compatible with subgroups \cite[IV, \S1, Prop.~2]{Se2}, it is the break at $\infty$ of the degree-$p$ cover $C\to C/\langle\sigma\rangle$. For the curve $\Ubar_n$ of Proposition~\ref{prop:tail}(c) we put $l_n=l_n(p):=l(\Ubar_n,\infty)$; here $\Ubar_n\to\Ubar_{n-1}$ (truncation of Witt vectors) is the quotient by $\langle\sigma\rangle$, \'etale over $U_{n-1}$, and $\Ubar_0:=\PP^1_t$.

\begin{lemma}[Artin--Schreier layers]\label{lem:ASlayer}
Let $C\to C'$ be a Galois cover of smooth projective curves over $k$ with group $\langle\sigma\rangle\cong\Z/p$, totally ramified at $\infty\in C$ over $\infty'\in C'$, with lower break $l$ at $\infty$ (so $l\ge1$), and let $\hat F$ be the completion of $k(C)$ at $\infty$.
\begin{enumerate}[label=\textup{(\alph*)}]
\item For every $x\in\hat F^\times$ with $p\nmid v_\infty(x)$: $v_\infty(\sigma x-x)=v_\infty(x)+l$.
\item Let $k(C)=k(C')(y)$ with $y^p-y=c\in k(C')$ and $v_{\infty'}(c)=-M<0$. Then $l\le M$, with equality iff $p\nmid M$. In particular $l$ is prime to $p$.
\item If $C\to C'$ is \'etale away from $\infty$, then $2g(C)-2=p\,(2g(C')-2)+(p-1)(l+1)$.
\end{enumerate}
\end{lemma}
\begin{proof}
(a) As $k$ is algebraically closed, $\hat F=k((\pi))$ for a uniformiser $\pi$, and $\sigma$ acts on $\hat F$ continuously and $k$-linearly. Write $\sigma\pi=\pi(1+\varepsilon)$ with $v(\varepsilon)=l$, and $x=u\pi^r$ with $r=v(x)$ and $u=\sum_{j\ge0}a_j\pi^j$, $a_0\ne0$. Then $\sigma u-u=\sum_{j\ge1}a_j\pi^j((1+\varepsilon)^j-1)$ has valuation $\ge l+1$, so $\sigma u/u=1+\delta$ with $v(\delta)\ge l+1$; and $(1+\varepsilon)^r=1+r\varepsilon+\varepsilon^2(\dots)$ with $v(r\varepsilon)=l$ because $p\nmid r$, and $v(\varepsilon^2)=2l>l$. Hence $\sigma x/x=(1+\delta)(1+\varepsilon)^r=1+r\varepsilon+(\text{terms of valuation}>l)$, so $v(\sigma x-x)=v(x)+l$.

(b) The roots of $Y^p-Y-c$ in $k(C)$ are $y+\F_p$, and $\sigma y\ne y$, so $\sigma y-y=j\in\F_p^\times$. Suppose first $p\nmid M$. As $v_\infty(c)=-pM$ (ramification index $p$), $v_\infty(y)=-M$, and (a) gives $0=v(j)=v(\sigma y-y)=-M+l$, so $l=M$. Now let $p\mid M$. In the completion $\hat F'=k((\pi'))$ of $k(C')$ at $\infty'$ reduce $c$: while $c=\gamma\pi'^{-pa}+(\text{higher terms})$ with $a>0$, replace $c$ by $c-\wp(\gamma^{1/p}\pi'^{-a})$, which has larger valuation. After finitely many steps $c=\wp(b)+c'$ with $b\in\hat F'$ and either $v(c')\ge0$ or $v(c')=-m$ with $p\nmid m$ and $m<M$. Put $y'=y-b\in\hat F$; then $y'^p-y'=c'$ and $\sigma y'-y'=j$, as $\sigma$ fixes $\hat F'$ (the completion of the fixed field). If $v(c')\ge0$, Hensel's lemma gives a root of $Y^p-Y-c'$ in $\hat F'$ (its reduction is separable with a root in $k$), so $y'\in\hat F'$ and $\sigma y'=y'$, a contradiction. Hence $v(c')=-m$, $v_\infty(y')=-m$, and (a) applied to $y'$ gives $l=m<M$.

(c) The different exponent at $\infty$ is $\sum_{i\ge0}(|G_i|-1)=(l+1)(p-1)$ by Hilbert's formula \cite[IV, \S1, Prop.~4]{Se2}, and $C\to C'$ is unramified elsewhere; apply the Riemann--Hurwitz formula \cite[IV.2.4]{Hart77}.
\end{proof}

\begin{lemma}\label{lem:lnformula}
Recall $\Ubar_0=\PP^1_t$, so $g(\Ubar_0)=0$.
\begin{enumerate}[label=\textup{(\alph*)}]
\item The upper ramification breaks of $\Ubar_n\to\PP^1_t$ at $\infty$ are $1,p,\dots,p^{n-1}$; consequently the lower breaks are the numbers $1+(p-1)\sum_{i=1}^{j-1}p^{2i-1}$, $1\le j\le n$, and
\begin{equation}\label{eq:ln}
\begin{gathered}
l_n(p)=1+(p-1)\sum_{j=1}^{n-1}p^{2j-1}=\frac{p^{2n-1}+1}{p+1}:\\
l_2(2)=3,\quad l_2(p)=p^2-p+1,\quad l_3(2)=11,\quad l_4(2)=43,\quad l_3(3)=61 .
\end{gathered}
\end{equation}
In particular $l_{m+1}(p)=p^2\,l_m(p)-p+1$ for $m\ge1$.
\item $2g(\Ubar_m)-2=p\,(2g(\Ubar_{m-1})-2)+(p-1)(l_m(p)+1)$ for $m\ge1$.
\item $l_m(p)=p\,(2g(\Ubar_{m-1})-1)+p^m+1$ for $m\ge1$.
\end{enumerate}
\end{lemma}
\begin{proof}
(a) The completion at $\infty$ of $\Ubar_n\to\PP^1_t$ is the Artin--Schreier--Witt extension of $k((1/t))$ of class $[t]$, a reduced Witt vector whose only nonzero coordinate has a pole of order $1$. By the formula for the upper breaks of Artin--Schreier--Witt extensions of a reduced Witt vector of Schmid, Kanesaka--Sekiguchi and Thomas for finite residue fields \cite{Sch37,KS79,Th}. Since $\Ubar_n\to\PP^1_t$ and the class $[t]$ are defined over $\F_p$, and lower breaks do not change under the constant field extension $\F_p((1/t))\subset k((1/t))$, this finite-residue-field case applies directly (see also \cite[Thm.~2.3]{EK25} for arbitrary perfect residue fields, \cite[Thm.~1]{Ga} for the top break of cyclic covers of curves over a perfect field, and \cite{Bry83})
(if the coordinates $a_i$ of a reduced vector that are not regular at $\infty$ have pole orders $m_i$, then $u_j=\max\{p^{\,j-1-i}m_i:i<j,\ a_i\text{ not regular}\}$), the upper breaks are $u_j=p^{j-1}$. Then \eqref{eq:herbrand} gives $l_1=1$ and $l_{j+1}=l_j+p^j(p^j-p^{j-1})=l_j+(p-1)p^{2j-1}$, which is \eqref{eq:ln}; the closed form and the recursion are elementary.
(b) $\Ubar_m\to\Ubar_{m-1}$ is a $\Z/p$-cover, \'etale over $U_{m-1}$ (as $U_m\to\A^1$ is \'etale), totally ramified at $\infty$ with break $l_m(p)$; apply Lemma~\ref{lem:ASlayer}(c).
(c) Put $M_m:=l_m(p)-p(2g(\Ubar_{m-1})-1)$. Then $M_1=1+p$. By (b), $p(2g(\Ubar_m)-1)=p^2(2g(\Ubar_{m-1})-1)-p^2+p(p-1)(l_m+1)+p$, and $p^2(2g(\Ubar_{m-1})-1)=p(l_m-M_m)$; together these give $p(2g(\Ubar_m)-1)=p^2l_m-pM_m$, i.e.\ $M_{m+1}=l_{m+1}-p^2l_m+pM_m=pM_m+1-p$ by (a). By induction $M_m=p^m+1$.
\end{proof}

\begin{definition}\label{def:lambda}
Let $G$ be a finite group. For a smooth projective curve $C$ over $k$ with a faithful $G$-action and a $G$-fixed $k$-point $\infty$, let $\lambda(C,\infty)$ be the smallest element prime to $p$ of the Weierstrass semigroup $H(\infty)$ (the set of pole orders at $\infty$ of functions regular on $C\smallsetminus\{\infty\}$), and let $\lambda_k(G)$ be the minimum of $\lambda(C,\infty)$ over all such pairs ($\lambda_k(G)=\infty$ if none exists).
\end{definition}

Such pairs exist for every finite $p$-group and, more generally, for every $G\cong P\rtimes\Z/m$ with $P$ a $p$-group and $p\nmid m$: by the theorem of Katz and Gabber \cite[Main Thm.~1.4.1, Cor.~1.4.12]{Ka} (the equivalence between special covers of $\mathbb G_m$ and covers of $\Spec k((1/T))$, and the local monodromy at $0$; every $P\rtimes\Z/m$ is a quotient of the local Galois group; for $p$-groups already by Harbater \cite[Section~2]{Ha}, cf.\ \cite[Thm.~2.1.4]{Ka}; see also \cite[Lemma~3]{RV18}, where the same pairs are obtained over $\F_p$)
there is a connected $G$-cover of $\PP^1$ \'etale over $\mathbb G_m$ (over $\A^1$ if $G$ is a $p$-group) with inertia group $G$ at $\infty$, and the point above $\infty$ is then $G$-fixed. Conversely a $G$-fixed point forces $G$ to be an inertia group in characteristic $p$, hence of this form. For $G=\Z/p^n$ the pair $(\Ubar_n,\infty)$ is explicit.

\begin{theorem}\label{thm:B}
Let $n\ge1$.
\begin{enumerate}[label=\textup{(\alph*)}]
\item The top layer $\Ubar_n\to\Ubar_{n-1}$ has an Artin--Schreier generator $y\in\mathcal O(U_n)$ with $y^p-y=c\in\mathcal O(U_{n-1})$ and $v_\infty(c)=-l_n$ on $\Ubar_{n-1}$; then $v_\infty(y)=-l_n$ on $\Ubar_n$.
\item $\lambda(\Ubar_n,\infty)\le l_n(p)$.
\item $D(U_n)$ contains an integer prime to $p$ and $\le l_n(p)$; hence
\[
d_1^{(p')}(\tgen)\le l_n(p).
\]
\end{enumerate}
\end{theorem}
\begin{proof}
(a) For $n=1$ take $y=y_0$, $c=t$. Let $n\ge2$. By the definition of $U_n$, $U_n\subset U_{n-1}\times\A^1$ is the curve $y_{n-1}^p-y_{n-1}=c_0$, where $c_0\in k[y_0,\dots,y_{n-2}]\subseteq\mathcal O(U_{n-1})$ is minus the part of $(\wp y)_{n-1}$ not involving $y_{n-1}$; so $y_{n-1}$ is an Artin--Schreier generator of $k(\Ubar_n)/k(\Ubar_{n-1})$, regular on $U_n$. Every Artin--Schreier generator $c'\in\mathcal O(U_{n-1})$ of this layer has $v_\infty(c')<0$: otherwise $c'$ is regular on the projective curve $\Ubar_{n-1}$, hence constant, and the layer would split. We modify $c_0$ by elements $\wp(b)$, $b\in\mathcal O(U_{n-1})$, until its pole order is prime to $p$; by Lemma~\ref{lem:ASlayer}(b) the final pole order is then $l_n$. Suppose the current generator $c'$ has pole order $pa$. By Lemma~\ref{lem:ASlayer}(b), $l_n<pa$. It suffices to know that $a\in H(\infty)$ on $\Ubar_{n-1}$: then there is $b\in\mathcal O(U_{n-1})$ with $v_\infty(b)=-a$, which after scaling by a constant ($k$ is algebraically closed) has $b^p$ with the same leading term as $c'$ with respect to a fixed uniformiser, and $c'-\wp(b)$ has smaller pole order. Since the pole orders strictly decrease, the process stops, and $y:=y_{n-1}-\sum b$ is the required generator; $v_\infty(y)=-l_n$ because $v_\infty(c)=-pl_n$ on $\Ubar_n$.

It remains to see $a\in H(\infty)$. For $n=2$, $\Ubar_1=\PP^1_{y_0}$ and $H(\infty)=\Z_{\ge0}$. For $n\ge3$, Lemma~\ref{lem:lnformula}(c), which rests on \cite{Sch37,EK25}, gives $pa>l_n=p(2g(\Ubar_{n-1})-1)+p^n+1$, so $a\ge2g(\Ubar_{n-1})$, and every integer $\ge2g(\Ubar_{n-1})$ lies in $H(\infty)$ by Riemann--Roch.

For $n=2$ the reduction is not even needed: from the ghost component $w_1(x)=x_0^p+px_1$ one gets $(\wp y)_1=y_1^p-y_1+Q(y_0)$ with
\[
Q(y_0)=\tfrac1p\bigl(y_0^{p^2}-y_0^p-(y_0^p-y_0)^p\bigr)=y_0^{p^2-p+1}+(\text{terms of lower degree}),
\]
since the term $\binom p{p-1}y_0^{p(p-1)}(-y_0)$ of $(y_0^p-y_0)^p$ contributes $+y_0^{p^2-p+1}$ to $Q$, the term $y_0^{p^2}$ cancels, and the remaining terms have degree $(p-1)i+p<p^2-p+1$ ($0\le i\le p-2$). As $v_\infty(y_0)=-1$ on $\Ubar_1=\PP^1_{y_0}$, the generator $c_0=-Q(y_0)$ already has pole order $p^2-p+1$, prime to $p$; so $y_1$ is a regular reduced generator and $l_2(p)=p^2-p+1$, independently of \cite{Th}. At $(p,n)=(2,3)$ a single reduction step, with $b=y_0^2y_1$, gives the reduced generator $y_2+y_0^2y_1$ of pole order $11$ (Appendix~\ref{app:genus}); so $l_3(2)=11$, again independently of \cite{Th}.

(b) By (a), $y$ is regular on $U_n$ with pole order $l_n$, prime to $p$ (Lemma~\ref{lem:ASlayer}(b)).

(c) $X:=\tw{\tgen}{\Ubar_n}$ is a smooth projective geometrically integral $K$-curve with the $K$-point $\infty'$ (the twist of $\infty$), and its free locus is $X\smallsetminus\{\infty'\}$, as $\Ubar_n^{\free}=U_n$. Riemann--Roch dimensions $\ell(d\infty')$ are invariant under the base change to $L$, where $(X,\infty')$ becomes $(\Ubar_n,\infty)$ (Lemma~\ref{lem:twistL}, Fact~\ref{fact:pic}(b)); so by (a) there is $f\in K(X)$ with pole divisor exactly $l_n\infty'$. For $c\in K$ the divisor $\divv(f-c)+l_n\infty'$ is effective, $K$-rational, of degree $l_n$, and supported on the free locus. The degrees of its closed points, counted with multiplicity, sum to $l_n$, which is prime to $p$; so one of them has degree $e$ prime to $p$, $e\le l_n$, and it is free. Hence $e\in D(U_n)$, and Lemma~\Aref{lem:freepoint} gives $\ed_k(\tgen\otimes F')\le1$ for its residue field $F'$, of degree $e$; so $d_1^{(p')}(\tgen)\le e$.
\end{proof}
The specialisation to $(2,3)$ is Corollary~\ref{cor:B23}.

\begin{proposition}[curves with a fixed point]\label{thm:lambda}
Let $C$ be a smooth projective curve over $k$ with a faithful action of $G=\Z/p^n$ fixing a point $\infty$. If $f\in k(C)$ is regular on $C\smallsetminus\{\infty\}$ and $p\nmid v_\infty(f)$, then
\[
-v_\infty(f)\ \ge\ l(C,\infty)\ \ge\ l_n(p).
\]
Consequently $\lambda(C,\infty)\ge l_n(p)$ for every such pair, $\lambda(\Ubar_n,\infty)=l_n(p)$, and
\[
\lambda_k(\Z/p^n)=l_n(p)=1+(p-1)\sum_{j=1}^{n-1}p^{2j-1}.
\]
\end{proposition}
\begin{proof}
Let $\sigma=g^{p^{n-1}}$; then $C\to C/\langle\sigma\rangle$ is a $\Z/p$-cover totally ramified at $\infty$ with lower break $l(C,\infty)$. We have $\sigma f\ne f$: otherwise $f\in k(C/\langle\sigma\rangle)$ and $v_\infty(f)\in p\Z$, as $\infty$ has ramification index $p$ over its image. Since $\sigma$ preserves $C\smallsetminus\{\infty\}$, the nonzero function $\sigma f-f$ is regular on $C\smallsetminus\{\infty\}$, hence is a constant or has a pole at $\infty$: $v_\infty(\sigma f-f)\le0$. By Lemma~\ref{lem:ASlayer}(a), $v_\infty(\sigma f-f)=v_\infty(f)+l(C,\infty)$. Hence $-v_\infty(f)\ge l(C,\infty)$.

For the second inequality, the completion of $k(C)$ at $\infty$ is a totally ramified cyclic extension of degree $p^n$ of a field $k((x))$, so its upper breaks satisfy $u_{j+1}\ge p\,u_j$ (Schmid \cite{Sch37}; this inequality is the only consequence of the break formula needed here). It follows from the formula $u_{j+1}=\max\{pu_j,m_j\}$ for a reduced Witt vector of the extension (every class in $\W_n(k((x)))/\wp$ has a reduced representative, each coordinate being regular or having a pole of order prime to $p$, since the residue field $k$ is perfect; total ramification forces $u_1=m_0\ge1$).
Since $u_1=l_1\ge1$, we get $u_j\ge p^{j-1}$ and $u_{j+1}-u_j\ge(p-1)u_j\ge(p-1)p^{j-1}$, and \eqref{eq:herbrand} gives
\[
l(C,\infty)=u_1+\sum_{j=1}^{n-1}p^j(u_{j+1}-u_j)\ \ge\ 1+(p-1)\sum_{j=1}^{n-1}p^{2j-1}=l_n(p),
\]
the last equality being \eqref{eq:ln}. The consequences follow, using Theorem~\ref{thm:B}(b) for $\lambda(\Ubar_n,\infty)\le l_n(p)$.
\end{proof}

\begin{remark}
(i) The first inequality of Proposition~\ref{thm:lambda} uses none of the ramification formulas; with Theorem~\ref{thm:B}(a) it gives $\lambda(\Ubar_n,\infty)=l(\Ubar_n,\infty)$ directly. The second uses only $u_{j+1}\ge p\,u_j$. The closed formula \eqref{eq:ln} enters for the value of $l_n(p)$, for the sharpness (that $\Ubar_n$ attains the bound) and, for $n\ge3$, for $a\in H(\infty)$ in Theorem~\ref{thm:B}(a); it is proved here directly for $n\le2$ and for $(p,n)=(2,3)$.
(ii) By Proposition~\ref{thm:lambda}, $l_n(p)$ is the smallest $d$ prime to $p$ for which a twisted curve $\tw{\tgen}{C}$ with a $G$-fixed point carries a function with pole divisor exactly $d\infty'$, i.e.\ the smallest degree of a pencil $|d\infty'|$ to which the degree-sum argument in the proof of Theorem~\ref{thm:B}(c) applies. That this is nevertheless the true value of $d_1^{(p')}(\tgen)$ is Theorem~\ref{thm:ptp}, which also controls fibres of degree divisible by $p$.
(iii) The naive argument ``$\ell(d\infty')\ge2$ for $d\ge g(\Ubar_n)+1$, take $d$ prime to $p$'' does not work: the fibres of $f\in\mathcal L(d\infty')$ have degree equal to the actual pole order of $f$, which lies in the Weierstrass semigroup and may be divisible by $p$. At $(2,3)$, $g(\Ubar_3)=7$ and the semigroup of $\infty$ is $\langle4,6,11\rangle$, whose gaps are $\{1,2,3,5,7,9,13\}$ (Appendix~\ref{app:genus}): every function regular on $U_3$ with pole order $\le10$ has even pole order.
\end{remark}

\medskip
Reichstein and Vistoli \cite{RV18} prove $\ed_{k,p}(G)\le1$ for every finite $p$-group $G$ in characteristic $p$, i.e.\ every $G$-torsor becomes one-dimensional over some extension of degree prime to $p$: in the language of the profile, $d_1^{(p')}(\tau)<\infty$ for every $\tau$. The proof passes through the $p$-closure, a large field, and gives neither a bound on the degree nor uniformity in $\tau$. The Riemann--Roch argument of Theorem~\ref{thm:B} gives both. In the next theorem $G$ is an arbitrary finite group and $V$ a faithful representation of $G$ over $k$ with generic torsor $\tgen$ over $K=k(V)^G$ (for $G=\Z/p^n$ we may and do take $V=\W_n$, Proposition~\Aref{prop:vers}).

\begin{theorem}[quantitative Reichstein--Vistoli]\label{thm:qRV}
\begin{enumerate}[label=\textup{(\alph*)}]
\item Let $(C,\infty)$ be as in Definition~\ref{def:lambda}. For every $G$-torsor $\tau$ over every field $F\supseteq k$ there is a separable extension $F'/F$ with $p\nmid[F':F]\le\lambda(C,\infty)$ and $\ed_k(\tau\otimes F')\le1$. Hence
\[
d_1^{(p')}(\tau)\le d_1^{(p')}(\tgen)\le\lambda_k(G)\qquad\text{for every }\tau .
\]
\item For $G=\Z/p^n$: $\lambda_k(\Z/p^n)=l_n(p)=1+(p-1)\sum_{j=1}^{n-1}p^{2j-1}$, realised by $(\Ubar_n,\infty)$, and no pair $(C,\infty)$ gives a smaller bound in \textup{(a)} (Proposition~\ref{thm:lambda}). So every $\Z/p^n$-torsor over every $F\supseteq k$ becomes one-dimensional over a separable extension of degree prime to $p$ and at most $l_n(p)$; this bound is $1$ for $n=1$, $p^2-p+1$ for $n=2$, $11$ at $(2,3)$, $43$ at $(2,4)$, and it is attained by $\tgen$ (Theorem~\ref{thm:ptp}). Together with Proposition~\ref{prop:trunc}, $d_1(\tau)\le p^{n-1}$ for every $\Z/p^n$-torsor $\tau$ (note $p^{n-1}\le l_n(p)$).
\item For an arbitrary finite $G$ with Sylow $p$-subgroup $G_p$: $d_1^{(p')}(\tau)\le[G:G_p]\cdot\lambda_k(G_p)$ for every $G$-torsor $\tau$ over every $F\supseteq k$.
\end{enumerate}
\end{theorem}
\begin{proof}
(a) The field $F\supseteq k$ is infinite. Let $X=\tw{\tau}{C}$ and $\infty'\in X(F)$ the twist of $\infty$. The non-free locus $N\subset C$ is a finite $G$-stable set (fixed points of the nontrivial elements of a group acting faithfully on a curve), so $\tw{\tau}{N}$ is a finite set of closed points of $X$. The Weierstrass semigroup of $\infty'$ on $X$ equals $H(\infty)$ (Riemann--Roch dimensions are invariant under base change, and $(X,\infty')_{\bar F}\cong(C,\infty)_{\bar F}$), so there is $f\in F(X)$ with pole divisor exactly $\lambda\infty'$, $\lambda=\lambda(C,\infty)$. The fibres $f=c$, $c\in F$, are pairwise disjoint effective $F$-rational divisors of degree $\lambda$ not containing $\infty'$; all but finitely many avoid $\tw{\tau}{N}$. For such a fibre the degrees of its closed points, counted with multiplicity, sum to $\lambda$, which is prime to $p$; so one closed point has degree $e$ prime to $p$, $e\le\lambda$, and lies in the free locus. Its residue field $F'$ is separable over $F$ (its degree is prime to $p$), and Theorem~\Aref{thm:geom} with $Z=C^{\free}$ gives $\ed_k(\tau\otimes F')\le1$. The chain of inequalities is Proposition~\Aref{prop:vers}.

(b) The equality is Proposition~\ref{thm:lambda}, and the bound for every torsor is (a). Sharpness for $\tgen$ is Theorem~\ref{thm:ptp}. The last claim is Proposition~\ref{prop:trunc} applied to an arbitrary torsor: $\edd{p^{n-1}}_k(\tau_a)\le\trdeg_kk(a_0)\le1$.

(c) $T/G_p\to\Spec F$ is \'etale of degree $[G:G_p]$, so it has a field factor $F_1$ with $p\nmid[F_1:F]\le[G:G_p]$; over $F_1$ the torsor has a point of $T/G_p$, hence $\tau\otimes F_1\cong\Ind_{G_p}^G\tau_p$ for a $G_p$-torsor $\tau_p$ over $F_1$. By (a) there is $F_2/F_1$ with $p\nmid[F_2:F_1]\le\lambda_k(G_p)$ and $\ed_k(\tau_p\otimes F_2)\le1$, and $\ed_k(\tau\otimes F_2)=\ed_k(\tau_p\otimes F_2)$ by Lemma~\Aref{lem:ind}.
\end{proof}

\begin{remark}\label{rem:qRV}
(i) What is new relative to \cite{RV18} and to Proposition~\ref{prop:tail}(b),(c): the explicit bound, its uniformity in $\tau$, and the identification of the optimal bound obtainable from curves with a fixed point with the ramification invariant $l_n(p)$ (Proposition~\ref{thm:lambda}). The mechanism, a free point on a curve with a fixed point, is that of \cite[Lemma~9.1]{RS22} and of Proposition~\ref{prop:tail}(c); what replaces largeness is Riemann--Roch on the twisted curve.
(ii) A crude bound valid for every $(C,\infty)$ is $\lambda(C,\infty)\le$ the least integer prime to $p$ that is $\ge2g(C)$, since $H(\infty)\supseteq[2g(C),\infty)$; at $(2,3)$ this gives $15$ against $11$.
(iii) By Theorem~\ref{thm:ptp}, $d_1^{(p')}(\tgen)=l_n(p)$: the bound of Theorem~\ref{thm:qRV}(a) is sharp for $\tgen$, and by Proposition~\ref{thm:lambda} it cannot be improved from any pair $(C,\infty)$.
(iv) The uniform statement is a statement about $\sup_\tau$: by Proposition~\Aref{prop:vers} it is equivalent to the bound for $\tgen$ alone.
\end{remark}

\begin{theorem}[the first prime-to-$p$ jump degree]\label{thm:ptp}
Let $n\ge2$ and let $F'/K$ be a finite extension of degree prime to $p$ with $\ed_k(\tgen\otimes_KF')\le1$. Then
\[
[F':K]\ \ge\ l_n(p)=1+(p-1)\sum_{j=1}^{n-1}p^{2j-1}=\frac{p^{2n-1}+1}{p+1}.
\]
Consequently $d_1^{(p')}(\tgen)=l_n(p)$, and no smooth projective curve with a faithful $\Z/p^n$-action has a free closed point of degree $d<l_n(p)$, $p\nmid d$, on its twist.
\end{theorem}
\begin{proof}
Since $p\nmid d:=[F':K]$, the extension $F'/K$ is separable and $F'\cap L=K$ (an intermediate field $K_i\subseteq F'$ has degree $p^i\mid d$), so $L\otimes_KF'$ is a field and Setup~\ref{setup:witness} applies. Let $C$, $P$, $\xi$, $e=[\kappa(\xi):K]$, $\Delta=\tfrac de\,\psi(\eta)$ and $D_0$ be as there; by Proposition~\ref{prop:rigid}(b),(c), $D_0$ is an effective divisor of degree $d$ on the faithful $G$-curve $C$, $\Delta\sim D_0\otimes_kL$, $\ell(D_0)\ge2$, and $h^*D_0\sim D_0$ for all $h\in G$.

Step 1: $e\ge2$. If $e=1$ then $\xi$ is a free $K$-point of $X=\tw{\tgen}{C}$, and $\ed_k(\tgen)\le1$ by Lemma~\Aref{lem:freepoint}, contradicting $\ed_k(\tgen)\ge2$ (Remark~\ref{rem:rigid-d1}).

Motivation for Step 2. Purity of $L$ over $k$ has so far only made the class of $D_0$ constant. Step 2 uses the other half of the hypothesis, that the witness is a single closed point of degree $e\ge2$ over $L$: a fixed point of $|D_0|$ would be an $L$-rational point of $\Supp\Delta$, which is impossible. Base-point-freeness is what lets Step 5 find $f$ with prescribed order at the fixed point $\infty$.

Step 2: $|D_0|$ has no base points. Let $B$ be the fixed part of $|D_0|$, the largest effective divisor with $\mathcal L(D_0-B)=\mathcal L(D_0)$. By Fact~\ref{fact:pic}(b), $\mathcal L(D_0\otimes L)=\mathcal L(D_0)\otimes L$, so the fixed part of $|D_0\otimes L|$ is $B\otimes L$. As $\Delta\in|D_0\otimes L|$ is effective, $\Delta\ge B\otimes L$. If $B\ne0$, the support of $\Delta$ would contain a point of $\Supp(B\otimes L)$, which is $L$-rational. But $\Supp\Delta=\{\psi(\eta)\}$ is a single closed point with residue field $\kappa(\xi)\otimes_KL$, of degree $e\ge2$ over $L$. So $B=0$.

Step 3: a $G$-invariant divisor in the class of $D_0$. Put $\mathcal P=k(C)^\times/k^\times$. The sequence $0\to\mathcal P\to\Div(C)\to\Pic(C)\to0$ is $G$-equivariant. From $0\to k^\times\to k(C)^\times\to\mathcal P\to0$ and $H^1(G,k(C)^\times)=0$ (Hilbert~90, $G$ acting faithfully on $k(C)$) we get an injection $H^1(G,\mathcal P)\hookrightarrow H^2(G,k^\times)=k^\times/(k^\times)^{p^n}=0$ ($G$ is cyclic and $k$ is algebraically closed). Hence $\Div(C)^G\to\Pic(C)^G$ is onto, and since $[D_0]\in\Pic(C)^G$ there is a $G$-invariant divisor $E\sim D_0$. It need not be effective.

Step 4: a fixed point. Write $E=\sum_xn_xx$. The orbits of $G$ have $p$-power size, so $d=\deg E\equiv\sum_{x\ G\text{-fixed}}n_x\pmod p$. As $p\nmid d$ there is a $G$-fixed point $\infty$ with $p\nmid m:=n_\infty$. In particular the faithful action fixes $\infty$, and $\infty$ is totally ramified in $C\to C/G$.

Step 5: the function $f$. By Step~2 and $E\sim D_0$, $\infty$ is not a base point of $|E|$, so there is $f\in\mathcal L(E)=\{f:\divv f+E\ge0\}$ whose divisor $\divv f+E$ has coefficient $0$ at $\infty$, i.e.\ $v_\infty(f)=-m$ with $p\nmid m$.

Step 6: the bound. Let $\sigma=g^{p^{n-1}}$ generate the last non-trivial ramification group at $\infty$ and let $l=l(C,\infty)$ be the last lower break, so that $C\to C/\langle\sigma\rangle$ satisfies the hypotheses of Lemma~\ref{lem:ASlayer} with lower break $l$. By Lemma~\ref{lem:ASlayer}(a), applied to $f$ (valid for either sign of $v_\infty(f)$ as $p\nmid v_\infty(f)$), $h:=\sigma f-f\ne0$ and $v_\infty(h)=-m+l$. Since $E$ is $\sigma$-invariant, $h\in\mathcal L(E)$. So $\divv h+E\ge0$ is an effective divisor of degree $d$ with coefficient $v_\infty(h)+m=l$ at $\infty$, whence $d\ge l(C,\infty)\ge l_n(p)$ by Proposition~\ref{thm:lambda}.

The consequence $d_1^{(p')}(\tgen)=l_n(p)$ follows from Theorem~\ref{thm:qRV}(b). A free closed point $\xi$ of degree $d$ on the twist of a faithful curve has residue field $F'$ with $\ed_k(\tgen\otimes F')\le1$ (Lemma~\Aref{lem:freepoint}) and $[F':K]=d$, so $d\ge l_n(p)$ if $p\nmid d$.
\end{proof}

\begin{corollary}\label{cor:ptpvalues}
$d_1^{(p')}(\tgen)=p^2-p+1$ for $n=2$ and every $p$ \textup{(}$=3$ for $p=2$, $7$ for $p=3$, $21$ for $p=5$\textup{)}; $d_1^{(p')}(\tgen)=11$ at $(2,3)$ and $61$ at $(3,3)$; $d_1^{(p')}(\tgen)=43$ at $(2,4)$. Thus the bound of Theorem~\ref{thm:qRV} is sharp for $\tgen$ in every case.
\end{corollary}

\begin{remark}\label{rem:ptpmech}
The mechanism is different from the pencil argument of Section~\ref{ss:rigid}: for a degree $d$ prime to $p$ the base-point-freeness of the witness system and the invariance of its class (a purely cohomological fact for cyclic $G$ over an algebraically closed field) force a fixed point of the invariant divisor, where the last Artin--Schreier layer $\sigma$ produces a function in $\mathcal L(E)$ with a controlled pole. This is also why the argument needs $p\nmid d$: for $p\mid d$ the divisor $E$ need not have a $G$-fixed point of multiplicity prime to $p$, and Theorem~\ref{thm:N1} (the PGL$_\ell$ argument) takes over.
\end{remark}

Over a large field the same curve does more: it has dense rational points, and the Reichstein--Scavia lemma follows on it.

\begin{corollary}[large fields]\label{cor:large}
Let $(C,\infty)$ be as in Definition~\ref{def:lambda} (e.g.\ $G=\Z/p^n$ and $(C,\infty)=(\Ubar_n,\infty)$), let $F\supseteq k$ be a large field (Fact~\Aref{fact:rs}; e.g.\ $F_0((t))$, a henselian nontrivially valued field, a $p$-closed field), $\tau$ a $G$-torsor over $F$, $X=\tw{\tau}{C}$, $\infty'\in X(F)$ the twist of $\infty$, and $f\in F(X)$ with pole divisor $\lambda\infty'$, $\lambda=\lambda(C,\infty)$, as in the proof of Theorem~\ref{thm:qRV}(a).
\begin{enumerate}[label=\textup{(\alph*)}]
\item $X(F)$ is dense in $X$. All but finitely many fibres $f=c$, $c\in F$, are free $F$-rational divisors of degree $\lambda$ of the pencil $|\lambda\infty'|$, and the fibres through the (infinitely many) free $F$-points of $X$ contain an $F$-point.
\item $\ed_k(\tau)\le1$, i.e.\ $d_1(\tau)=1$: the profile of every $G$-torsor over a large field is $(\le1,\dots,\le1,0)$, with the terminal $0$ at $d_0(\tau)=|\im\varphi_\tau|$.
\item For $G=\Z/p^n$: $\tau\cong\tau_{[t]}$ for some $t\in F$, i.e.\ every class in $\W_n(F)/\wp$ is Teichm\"uller, with no finite extension of $F$ needed.
\end{enumerate}
Part \textup{(b)} is \cite[Lemma~9.1]{RS22} (Fact~\ref{fact:rsd}) for the groups admitting a pair $(C,\infty)$; its proof below is theirs, specialised to the curve $C$, and uses only the density of $X(F)$, not the pencil.
\end{corollary}
\begin{proof}
(a) $X$ is a smooth projective geometrically integral $F$-curve with the smooth $F$-point $\infty'$, so $X(F)$ is dense in $X$ by the definition of a large field. The non-free locus $\tw{\tau}{N}$ is a finite set of closed points, so only finitely many fibres of $f$ meet it; every other fibre $f=c$ is an effective $F$-rational divisor of degree $\lambda$ supported in the free locus. A free $F$-point $P$ of $X$ (there are infinitely many, by density and finiteness of $\tw{\tau}{N}\cup\{\infty'\}$) has $f(P)=c\in F$ since $P\ne\infty'$, and lies on the fibre $f=c$. (b) A free $F$-point is a $G$-equivariant morphism $T\to C^{\free}$ (Fact~\Aref{fact:rs}(a)), and Lemma~\Aref{lem:freepoint} with $Z=C^{\free}$ gives $\ed_k(\tau)\le1$; hence $d_1(\tau)=1$. The terminal value and its position are Proposition~\Aref{prop:shape}(a),(e) and Proposition~\Aref{prop:d0}(a) ($k$ is algebraically closed). (c) As in the proof of Proposition~\ref{prop:tail}(c): a free $F$-point of $\tw{\tau}{\Ubar_n}$ lies in $\tw{\tau}{U_n}$ and is $y=x+h$ with $\wp(x)=a$ the class of $\tau$, $h\in\W_n(F)$ and $\wp(y)=[t]$, $t\in F$; so $a\equiv[t]$.
\end{proof}

\begin{remark}\label{rem:large}
(i) The proof of \cite[Lemma~9.1]{RS22} is the same mechanism: over a large field, weak versality (the fixed point $\infty$, i.e.\ the $F$-point $\infty'$) upgrades to versality (Fact~\Aref{fact:rs}(b)), which is the density statement in (a). Theorem~\ref{thm:qRV} is what survives over a non-large field: the pencil $|\lambda\infty'|$ still consists of free $F$-rational divisors, but they need not contain $F$-points, and the profile records how they split, into closed points of degree $\le\lambda$ prime to $p$ at worst.
(ii) The instance $(p,n)=(2,2)$: over $K_2=k(s_0,s_1)$ the twisted curve $E'\colon w^2+w=u(u+1)(u+s_0)+s_1$ has $E'(K_2)=\{\infty'\}$, and the pencil $|3\infty'|=\langle1,u,w\rangle$ ($\lambda=l_2(2)=3$) has as free members the line sections, each a single cubic point; so $d_1^{(p')}(\tau_2)=3=l_2(2)$ and $d_1(\tau_2)=2$ (Theorem~\ref{thm:n2}). Over the large field $F=K_2((t))$ the same curve acquires the points near $\infty'$: with the local parameter $u/w$ at $\infty'$ put $u/w=t$, i.e.\ $w=u/t$, and $u=t^{-2}v$; the equation becomes
\[
v^3+v^2+(1+s_0)t^2v^2+t^3v+s_0t^4v+s_1t^6=0,
\]
which has the simple root $v\equiv1\bmod t$ (the derivative is $v^2+t^3+s_0t^4\equiv1$), so Hensel's lemma gives a unique $v\in K_2[[t]]$ with $v(0)=1$, in fact $v\in\F_2[s_0,s_1][[t]]$:
\begin{align*}
u&=t^{-2}+(1+s_0)+t+s_0t^2+(1+s_0)t^3+(1+s_0+s_0^2+s_1)t^4+(1+s_0^2)t^5+O(t^6),\\
w&=u/t=t^{-3}+(1+s_0)t^{-1}+1+s_0t+(1+s_0)t^2+O(t^3)
\end{align*}
(the coefficients are obtained by the Hensel recursion, and one checks directly that substituting $v=t^2u$ into the cubic gives $0$ modulo $t^8$; the identities hold in $\F_2[s_0,s_1]$); note that $u$ has no $t^{-1}$ term ($v\equiv1\bmod t^2$). This is a free $F$-point of $E'^\circ=\tw{\tau_2}{U_2}$, and as in Theorem~\ref{thm:n2}(b) (the identity $s+\wp h=[(s+\wp h)_0]$ for $h=(u,w)$ with $(s+\wp h)_1=0$ holds over every extension of $K_2$) the generic $\Z/4$-torsor is Teichm\"uller over $K_2((t))$:
\[
(s_0,s_1)\equiv[x],\qquad x=s_0+u^2+u=t^{-4}+t^{-2}+s_0^2+t+(1+s_0)t^2+(1+s_0)t^3+O(t^4).
\]
So $d_1(\tau_2\otimes K_2((t)))=1$, against $d_1(\tau_2)=2$ and $d_1^{(p')}(\tau_2)=3$ over $K_2$; and $d_0(\tau_2\otimes K_2((t)))=4$ still, by Proposition~\Aref{prop:d0}(b), since $K_2$ is algebraically closed in $K_2((t))$.
(iii) In general, over a large field every class is Teichm\"uller (part (c)), which is the ``local form'' of Proposition~\ref{prop:tail}(c) without the finite prime-to-$p$ extension; in profile language, $d_1=1$. The profile is the invariant of what happens before the field is large.
\end{remark}

The degree-sum argument of Theorem~\ref{thm:B}(c) produces a free point of degree prime to $p$ but does not control its exact degree. When the exact degree matters we use Hilbert irreducibility over $K$.

\begin{lemma}[Hilbert irreducibility on a twisted curve]\label{lem:hilbertfree}
Let $C$ be a smooth projective curve over $k$ with a faithful action of $G=\Z/p^n$ and a $G$-fixed point $\infty$, and let $d\ge1$. Suppose some $f_0\in\mathcal L(d\infty)$ has pole order exactly $d$ and $df_0\ne0$. Then $\tw{\tgen}{C}$ has a free closed point of degree $d$; that is, $d\in D(C)$.
\end{lemma}
\begin{proof}
Let $X=\tw{\tgen}{C}$ and $\infty'\in X(K)$ the twist of $\infty$. Let $V=\mathcal L_X(d\infty')$, $V_1=\mathcal L_X((d-1)\infty')$ and $V_2=\{f\in V:df=0\}$, the kernel of the $K$-linear map $d\colon V\to\Omega_{K(X)/K}$. Under the base change to $L$ and the isomorphism $\psi$ of Lemma~\ref{lem:twistL}, these become $\mathcal L(d\infty)\otimes_kL$, $\mathcal L((d-1)\infty)\otimes_kL$ and the kernel of $d$ on $\mathcal L(d\infty)\otimes_kL$ (Fact~\ref{fact:pic}(b); kernels commute with flat base change). As $f_0\otimes1$ lies in neither of the last two, $V_1$ and $V_2$ are proper subspaces of $V$, and since $K$ is infinite there is $f\in V\smallsetminus(V_1\cup V_2)$. Then $f$ has pole divisor exactly $d\infty'$, so $[K(X):K(f)]=d$ \cite[Thm.~1.4.11]{St}. As $X$ is geometrically integral, $K(X)/K$ is separably generated and $\Omega_{K(X)/K}$ is one-dimensional, spanned by $df\ne0$; hence $\Omega_{K(X)/K(f)}=0$ and $K(X)/K(f)$ is separable \cite[II.8.6A]{Hart77}. Choose $\theta\in K(X)$ integral over $K[f]$ with $K(X)=K(f)(\theta)$, with minimal polynomial $P(f,Y)\in K[f][Y]$. For all $c\in K$ outside a finite set, the fibre of $f\colon X\to\PP^1_K$ over $c$ is $\Spec K[Y]/(P(c,Y))$ (the order $K[f][\theta]$ agrees with the integral closure of $K[f]$ in $K(X)$ away from finitely many points) and avoids the finite non-free locus of $X$ (the fibres are pairwise disjoint). The field $K=k(s_0,\dots,s_{n-1})$ is Hilbertian \cite[Thm.~13.4.2]{FJ}, so $P(c,Y)$ is irreducible over $K$ for all $c$ in a separable Hilbert subset of $K$, which is nonempty and remains so after removing finitely many elements \cite[\S12.1]{FJ}. For such $c$ the fibre is a single free closed point of degree $d$.
\end{proof}

A function $f_0$ whose pole order is prime to $p$ automatically has $df_0\ne0$: in the completion $k((\pi))$ at $\infty$ its formal derivative has valuation $v(f_0)-1$. So Lemma~\ref{lem:hilbertfree} gives $\{d\in H(\infty):p\nmid d\}\smallsetminus\{0\}\subseteq D(C)$ for every pair $(C,\infty)$; for $d$ divisible by $p$ the hypothesis $df_0\ne0$ has to be checked.


\section{The case \texorpdfstring{$n=2$}{n=2}: the complete profile}\label{sec:n2}

Let $n=2$, $G_2=\Z/p^2$, $K_2=k(s_0,s_1)$, $L_2=k(z_0,z_1)$, and $\tau_2$ the generic torsor.

\begin{proposition}[rational points of twisted $\Z/p^2$-curves]\label{prop:G1}
Let $C$ be a smooth projective curve with faithful $G_2$-action. Then $\tw{\tau_2}{C}$ has no free $K_2$-point; in fact $\tw{\tau_2}{C}(K_2)$ is in bijection with $C^{G_2}(k)$. Consequently $\ed_k(\Z/p^2)=2$ and $d_1(\tau_2)\ge2$.
\end{proposition}
\begin{proof}
A $K_2$-point is an equivariant $f\colon\Spec L_2\to C$; by Lemma~\ref{lem:imageB} the closure $Y$ of its image is a $G_2$-stable point (then a fixed $k$-point, and $f$ is constant) or all of $C$. In the latter case $k(C)\subseteq L_2=k(z_0,z_1)$, so $C\cong\PP^1$ by Lemma~\ref{lem:curvesB}(iii) and $G_2\hookrightarrow\PGL_2(k)$, contradicting Lemma~\ref{lem:curvesB}(i). So all $K_2$-points are the constant maps to fixed points, none free. Then $d_1(\tau_2)\ge2$ by Corollary~\ref{cor:jumpsB} (for $d=1$, $L_2\otimes K_2=L_2$ is a field), and $\ed_k(\Z/p^2)=\ed_k(\tau_2)\ge2$ by Lemma~\ref{lem:versal}; $\le2$ is Proposition~\ref{prop:trunc} with $j=0$.
\end{proof}

\begin{theorem}[$n=2$, $p=2$: the complete profile]\label{thm:n2}
Let $p=2$.
\begin{enumerate}[label=\textup{(\alph*)}]
\item The curve $\Ubar_2$ is the supersingular elliptic curve $E_0\colon y_1^2+y_1=y_0^3+y_0^2$ with $G_2$ acting by $(y_0,y_1)\mapsto(y_0+1,y_1+y_0)$, fixing only $\infty$. Its twist is the elliptic curve
\[
E':=\tw{\tau_2}{E_0}\colon\quad w^2+w=u(u+1)(u+s_0)+s_1\qquad\text{over }K_2=k(s_0,s_1),
\]
whose origin is the point $\infty'$ (the twist of $\infty$; we write $\infty'_2$ when the curve $\tw{\tgen}{\Ubar_3}$ is also present). Moreover $E'(K_2)=\{\infty'\}$: the Mordell--Weil group of the quartic twist of $E_0$ by the generic $\Z/4$-torsor over $k(s_0,s_1)$ is trivial.
\item Every fibre of $u\colon E'\to\PP^1_{K_2}$ over a point $c\in K_2$ is a free closed point of degree $2$, with residue field $K_2(w)$, $w^2+w=c(c+1)(c+s_0)+s_1$, and over it $(s_0,s_1)\equiv[s_0+c^2+c]$.
\item Let $u$ be a root of $P(u):=u^3+(1+s_0)u^2+s_0u+s_1=u(u+1)(u+s_0)+s_1$. Then $F'=K_2(u)$ is a separable cubic extension of $K_2$ and $(s_0,s_1)\equiv[s_0+u^2+u]\bmod\wp\W_2(F')$, so $\ed_k(\tau_2\otimes F')=1$. More precisely, the cubic closed points of $E'$ are exactly the sections by the lines $w=\alpha u+\beta$, $\alpha,\beta\in K_2$; each is a single free point of degree $3$, and over its residue field $(s_0,s_1)$ is Teichm\"uller.
\item $D(U_2)=\Z_{\ge2}$, and the profile of $\tau_2$ is
\[
\edd1=2,\quad\edd2=\edd3=1,\quad\edd d=0\ (d\ge4);\qquad d_2=1,\ d_1=2,\ d_1^{(p')}=3,\ d_0=4 .
\]
\end{enumerate}
\end{theorem}
\begin{proof}
(a) $U_2=\{y_0^2+y_0=t,\ y_1^2+y_1+y_0^3+y_0^2=0\}$ by \eqref{eq:wp2} (the $\W_2$ part), and $t$ is a function of $y_0$, so $U_2\cong E_0^\circ:=\{y_1^2+y_1=y_0^3+y_0^2\}\subset\A^2$, a smooth affine curve whose projective closure $E_0$ is an elliptic curve with one point $\infty$ at infinity ($a_1=0$, $a_3=1$, $\Delta=1$; supersingular since $a_1=0$ in characteristic $2$, $j=a_1^{12}/\Delta=0$). By \eqref{eq:add2}, $y+1=(y_0+1,y_1+y_0)$, an automorphism of $E_0$ without fixed points on $E_0^\circ$, so $\infty$ is the unique fixed point and $E_0^{\free}=E_0^\circ=U_2$. By Proposition~\ref{prop:S0}(a) (with $n=2$), $\tw{\tau_2}{U_2}=\{h=(u,w)\in\A^2_{K_2}:(s+\wp h)_1=0\}$, and by \eqref{eq:add2}--\eqref{eq:wp2}, $(s+\wp h)_1=s_1+w^2+w+u^3+u^2+s_0(u^2+u)$; so $\tw{\tau_2}{U_2}=E'^\circ:=\{w^2+w=u(u+1)(u+s_0)+s_1\}$, and $\tw{\tau_2}{E_0}$ is its smooth projective closure $E'$, with the single point $\infty'$ at infinity. By Proposition~\ref{prop:G1}, $E'(K_2)=\tw{\tau_2}{E_0}(K_2)$ is in bijection with $E_0^{G_2}(k)=\{\infty\}$.

(b) The fibre of $u$ over $c\in K_2$ is $\Spec K_2[w]/(w^2+w-c(c+1)(c+s_0)-s_1)$; it has no $K_2$-point, so it is a field of degree $2$, and its point is free (it lies in $E'^\circ$). Over $K'=K_2(w)$ the vector $h=(c,w)$ satisfies $(s+\wp h)_1=0$, so $s+\wp h=[(s+\wp h)_0]=[s_0+c^2+c]$.

(c) $P\in k(s_0)[u][s_1]$ has degree $1$ in $s_1$ with coprime coefficients $1$ and $u(u+1)(u+s_0)$, hence is irreducible in $k(s_0)[u,s_1]$, hence (Gauss, $P$ monic in $u$) irreducible in $K_2[u]$; an irreducible cubic in characteristic $2$ is separable. By \eqref{eq:add2}--\eqref{eq:wp2}, for $h=(u,0)\in\W_2(F')$: $(s+\wp h)_0=s_0+u^2+u$ and $(s+\wp h)_1=s_1+u^3+u^2+s_0u^2+s_0u=P(u)=0$, so $s+\wp h=[s_0+u^2+u]$; equivalently $(u,0)$ is a free $F'$-point of $E'^\circ=\tw{\tau_2}{U_2}$, and Lemma~\Aref{lem:freepoint} gives $\ed_k(\tau_2\otimes F')\le1$; $\ge1$ since $3<4=d_0$ (Proposition~\Aref{prop:d0}). Cubic points: $E'(K_2)=\{\infty'\}$, so $\Pic^0(E')(K_2)=0$ and every effective $K_2$-divisor of degree $3$ is linearly equivalent to $3\infty'$, i.e.\ is the divisor of zeros of an element of $\mathcal L(3\infty')=\langle1,u,w\rangle$; the divisor of zeros of $a+bu+cw$ contains $\infty'$ iff $c=0$, and then it is $\infty'$ plus a fibre of $u$ (of degree $2$, by (b)); a closed point of degree $3$ is a divisor not containing $\infty'$, so $c\ne0$ and it is a section by a line $w=\alpha u+\beta$ with $\alpha,\beta\in K_2$. Such a section is a single closed point of degree $3$ because a $K_2$-rational component would be a $K_2$-point of $E'^\circ$. It lies in $E'^\circ$, hence is free.

(d) Profile: $\edd1=\ed_k(\Z/4)=2$; $\edd2\le1$ by the degree-$2$ free points of (b) (Lemma~\Aref{lem:freepoint}), equivalently by Remark~\ref{rem:towerp2}; $\edd d\ge1$ for $d<4$ and $\edd4=0$ by Proposition~\Aref{prop:d0}(c). As $d_1^{(p')}$ is odd and $\ge d_1(\tau_2)=2$ (Proposition~\ref{prop:G1}), (c) gives $d_1^{(p')}(\tau_2)=3$. Finally $D(U_2)\cap\{1\}=\emptyset$ by Proposition~\ref{prop:G1}, $2,3\in D(U_2)$ by (b), (c), and for $d\ge4$ we apply Lemma~\ref{lem:hilbertfree}: the semigroup of $\infty$ on the elliptic curve $E_0$ is $\Z_{\ge0}\smallsetminus\{1\}$; for $d$ odd every function of pole order $d$ has nonzero differential, and for $d\ge4$ even, if a function $f$ of pole order $d$ has $df=0$, then $f+y_1$ has pole order $d$ and $d(f+y_1)=dy_1\ne0$ ($y_1$ has odd pole order $3$). So $D(U_2)=\Z_{\ge2}$.
\end{proof}

\begin{remark}
(i) The obstruction $E'(K_2)=\{\infty'\}$ kills degree-$1$ points and produces the cubic ones. (ii) For odd $p$ the same mechanism gives $d_1^{(p')}(\tau_2)\le p^2-p+1$ (Theorem~\ref{thm:qRV}(b)); the first jump degree $d_1(\tau_2)=p$ is Corollary~\ref{cor:n2odd}.
\end{remark}

\begin{proposition}[$n=2$, all $p$]\label{prop:n2p}
For every $p$: $2\le d_1(\tau_2)\le p$, $d_0(\tau_2)=p^2$, $\edd d_k(\tau_2)=1$ for $p\le d<p^2$, and $d_1^{(p')}(\tau_2)\le p^2-p+1$. There is a free point of degree $p$ on $\tw{\tau_2}{\Ubar_2}$ lying over $t=s_0$: the fibre of $\tw{\tau_2}{U_2}\to\A^1_{K_2}$ over $t=c\in K_2$ is the $G_2$-torsor of class $[c]-s$, which for $c=s_0$ equals $V(-s_1)$, so this fibre is $\Ind_{\Z/p}^{\Z/p^2}$ of the Artin--Schreier torsor $v^p-v=-s_1$ and consists of $p$ closed points of degree $p$.
\end{proposition}
\begin{proof}
$U_2\to\A^1$, $y\mapsto t$, is $G_2$-equivariant for the trivial action on $\A^1$, and twists to $\tw{\tau_2}{U_2}\to\A^1_{K_2}$. Its fibre over $c$ is the twist of the fibre $\wp^{-1}([c])$, and its $K'$-points are the $y=z+h$ with $\wp(z+h)=s+\wp h=[c]$, i.e.\ $h\in\W_2(K')$ with $\wp h=[c]-s$: the fibre is the torsor of class $[c]-s$. For $c=s_0$, $[s_0]-s=-V(s_1)$ by \eqref{eq:Vsplit}, the image under $V$ of the nonzero $\Z/p$-class $-s_1\in K_2/\wp K_2$; the corresponding torsor is induced from $\Z/p$, i.e.\ $\Spec$ of $p$ copies of $K_2(v)$. These points are free (they lie in $\tw{\tau_2}{U_2}$). So $d_1\le p$; $d_1\ge2$ is Proposition~\ref{prop:G1}, $d_0=p^2$ is Proposition~\Aref{prop:d0}(c). The bound on $d_1^{(p')}$ is Theorem~\ref{thm:qRV}(b) with $l_2(p)=p^2-p+1$.
\end{proof}

\begin{corollary}[$n=2$, odd $p$]\label{cor:n2odd}
Let $p$ be odd. Then $\edd d_k(\tau_2)=2$ for $d<p$ and $\edd p_k(\tau_2)=1$, so $d_1(\tau_2)=p$. For every $p$, $d_1^{(p')}(\tau_2)=p^2-p+1$.
\end{corollary}
\begin{proof}
The lower bounds $\edd d\ge2$ for $d<p$ are Theorem~\ref{thm:N1}; the upper bounds are Proposition~\ref{prop:n2p}, and $\ed_k(\Z/p^2)=2$ is Proposition~\ref{prop:G1}. For $d_1^{(p')}$ see Theorem~\ref{thm:ptp}.
\end{proof}


\section{The case \texorpdfstring{$(p,n)=(2,3)$}{(p,n)=(2,3)}}\label{sec:p2n3}

Let $p=2$, $n=3$, $G=\Z/8$, $K=k(s_0,s_1,s_2)$, $L=k(z_0,z_1,z_2)$, $v_i^2+v_i=s_i$, $z_0^2+z_0=s_0$, $K_2=k(s_0,s_1)$, $E'$ the curve of Theorem~\ref{thm:n2} over $K_2$ (base-changed to $K$ when needed), with origin $\infty'_2$, and $\infty'=\infty'_3$ the $K$-point of $\tw{\tgen}{\Ubar_3}$; $\Phi_2$ is the polynomial \eqref{eq:Phi2}, and $(s_0,s_1)_K$ denotes the $\Z/4$-class $\tau_2\otimes_{K_2}K$, which has $\ed_k((s_0,s_1)_K)=\ed_k(\tau_2)=2$ by Lemma~\Aref{lem:constant}(i). Every separable quadratic extension of $K$ is $K(v_f)$, $v_f^2+v_f=f$, $f\in K\smallsetminus\wp K$ (Corollary~\ref{cor:insep}).

\subsection{What the constructions give}\label{ss:known23}

\begin{proposition}[the slots at $(2,3)$ and their witnesses]\label{prop:known23}
\begin{enumerate}[label=\textup{(\alph*)}]
\item $\edd1(\tgen)=\ed_k(\Z/8)\in\{2,3\}$ (lower bound: Remark~\ref{rem:rigid-d1} and Lemma~\ref{lem:versal}; upper bound: Proposition~\ref{prop:trunc} with $j=0$).
\item $\edd2(\tgen)\le2$, witnessed by $K(v_2)$ (class $(s_0,s_1,0)$), $K(v_1)$ (class $(s_0,0,s_2+s_1^2+s_1v_1)$, Remark~\ref{rem:towerp2}) and $K(z_0)$ (Proposition~\ref{prop:intermediate}); hence $\edd3\le2$ by monotonicity.
\item $\edd3{}'(\tgen)\le2$: the cubic $K(u)$, $u^3+(1+s_0)u^2+s_0u+s_1=0$ (Theorem~\ref{thm:n2}(c) lifted by Lemma~\ref{lem:lift}; explicit class in the proof of Theorem~\ref{thm:23main}(d)).
\item $\edd4(\tgen)=\dots=\edd7(\tgen)=1$: the tower $K(v_1,v')$, $v'^2+v'=s_2+s_1^2+s_1v_1$ (Lemma~\ref{lem:tower}) gives $\le1$, and $\ge1$ for $d<8$ is Proposition~\Aref{prop:d0}(c). Finally $\edd8(\tgen)=0$ with $F'=L$.
\end{enumerate}
Hence $d_2\in\{1,2\}$, with $d_2=1$ iff $\ed_k(\Z/8)=2$; $d_2^{(p')}\in\{1,3\}$, likewise with $d_2^{(p')}=1$ iff $\ed_k(\Z/8)=2$; $d_1\in\{2,3,4\}$; $d_0=8$; and $d_1^{(p')}=11$ (Theorem~\ref{thm:ptp}). The explicit constructions and Proposition~\Aref{prop:d0} leave $\edd2(\tgen),\edd3(\tgen)\in\{1,2\}$; the lower bounds come from the rigidity lemma (Theorem~\ref{thm:23main}).
\end{proposition}
\begin{proof}
Over $K(v_2)$, $s=(s_0,s_1,0)+V^2(\wp v_2)\equiv(s_0,s_1,0)$ by \eqref{eq:Vsplit}, a class defined over $k(s_0,s_1)$; the class over $K(v_1)$ is defined over the transcendence-degree-$2$ field $k(s_0,s_2+s_1^2+s_1v_1)$; the remaining items are as indicated.
\end{proof}

The rest of this section determines the second and third slots and records what remains: Section~\ref{ss:23main} proves $\edd2=\edd3=2$ from Theorem~\ref{thm:N1}; Section~\ref{ss:ed2} treats explicit extensions, in particular the quartic tower; Section~\ref{ss:tail23} bounds the prime-to-$2$ tail; Section~\ref{ss:status} collects the results and the remaining questions.

\subsection{The quadratic and cubic slots}\label{ss:23main}

At $(2,3)$ one has $p^{n-1}=4$, so Theorem~\ref{thm:N1} excludes every extension of degree $\le3$, and together with the truncation tower it determines the profile in every slot except the first.

\begin{theorem}[the profile at $(2,3)$ away from the Ledet slot]\label{thm:23main}
Let $k$ be algebraically closed of characteristic $2$ and $\tgen$ the generic $\Z/8$-torsor over $K=k(s_0,s_1,s_2)$.
\begin{enumerate}[label=\textup{(\alph*)}]
\item $\ed_k(\tgen\otimes_KF')\ge2$ for every field extension $F'/K$ of degree $\le3$.
\item $\edd2_k(\tgen)=\edd3_k(\tgen)=2$, hence $d_1(\tgen)=4$; the profile is
\[
\edd1\in\{2,3\},\qquad\edd2=\edd3=2,\qquad\edd4=\dots=\edd7=1,\qquad\edd8=0 .
\]
\item For every smooth projective curve $C$ with a faithful $\Z/8$-action, $D(C)\cap\{1,2,3\}=\emptyset$. In particular $D(U_3)\cap\{1,2,3\}=\emptyset$, $4,6\in D(U_3)$, and $\delta_\free(U_3)=4$.
\item No extension of degree $\le3$ realises the first drop of the prime-to-$2$ profile: $\edd3_k{}'(\tgen)=2$ and $d_1^{(p')}(\tgen)\ge5$. The value $2$ is attained by the cubic extension $K(u)$, $u^3+(1+s_0)u^2+s_0u+s_1=0$; hence $d_2^{(p')}(\tgen)\in\{1,3\}$, with $d_2^{(p')}(\tgen)=1$ iff $\ed_k(\Z/8)=2$.
\end{enumerate}
\end{theorem}
\begin{proof}
(a) is Theorem~\ref{thm:N1} with $p^{n-1}=4$.
(b) $\edd2,\edd3\ge2$ by (a), and $\le2$ by Proposition~\ref{prop:known23}; $\edd4=1$ by Proposition~\ref{prop:known23}, so $d_1=4$.
(c) The first statement is Corollary~\ref{cor:curvepts}. Let $X=\tw{\tgen}{\Ubar_3}$ and $d\in\{4,6\}$. Since $d$ lies in the Weierstrass semigroup $\langle4,6,11\rangle$ of $\infty$ on $\Ubar_3$ (Appendix~\ref{app:genus}) and $\ell(d\infty')-\ell((d-1)\infty')$ is invariant under base change, there is $f\in K(X)$ with pole divisor exactly $d\infty'$. Every fibre $f=c$, $c\in K$, is an effective $K$-rational divisor of degree $d$ in the free locus $X\smallsetminus\{\infty'\}$. Every free closed point has degree $\ge4$ by the first statement, so such a fibre is a single free point of degree $d$ (a point of degree $\ge4$ of multiplicity $\ge2$ would contribute $\ge8$). Hence $4,6\in D(U_3)$ and $\delta_\free(U_3)=4$. No Hilbert irreducibility is needed here.

(d) Lower bounds. By Definition~\Aref{def:profile}, $\edd3{}'$ minimises over extensions of degree $1$ or $3$, and $\edd4{}'=\edd3{}'$; so $\edd3{}'(\tgen)\ge2$ and $d_1^{(p')}(\tgen)\ge5$ by (a).

Upper bound. Let $u$ be a root of $P(u)=u^3+(1+s_0)u^2+s_0u+s_1=u(u+1)(u+s_0)+s_1$. By Theorem~\ref{thm:n2}(c), $P$ is irreducible and separable over $K_2$. Since $K=K_2(s_2)$ is purely transcendental over $K_2$, $K_2$ is algebraically closed in $K$ and $P$ stays irreducible over $K$ (Gauss's lemma in $K_2[s_2][u]$). So $F':=K(u)=k(s_0,u,s_2)$ is a separable cubic extension of $K$. Put $h=(u,0,0)\in\W_3(F')$. By the general translate of Appendix~\ref{app:witt},
\[
s+\wp h=\bigl(s_0+u^2+u,\ P(u),\ s_2+\Phi_2(s_0,s_1,u,0)\bigr)=\bigl(s_0+u^2+u,\ 0,\ s_2+\varphi\bigr),
\]
\[
\varphi=u^7+u^6+s_0(u^6+u^4)+s_0^2(u^4+u^2)+s_0^3(u^2+u)\in k(s_0,u),
\]
where $s_1=u(u+1)(u+s_0)$ has been substituted in $\Phi_2(s_0,s_1,u,0)$. The class $s+\wp h$ lies in $\W_3(F_0)$ with $F_0=k(s_0+u^2+u,\ s_2+\varphi)$, $\trdeg_kF_0=2$, so $\ed_k(\tgen\otimes F')\le2$. This is Lemma~\ref{lem:lift} applied to the cubic point of Theorem~\ref{thm:n2}(c), made explicit. Hence $\edd3{}'(\tgen)=2$.

Jump degree. $\edd1{}'=\edd2{}'=\edd1=\ed_k(\Z/8)$ and $\edd3{}'=2$, so $d_2^{(p')}(\tgen)=1$ if $\ed_k(\Z/8)=2$ and $=3$ otherwise.
\end{proof}

\begin{remark}\label{rem:23main}
What is proved for an individual extension of degree $\le3$ is the lower bound $\ed_k(\tgen\otimes F')\in\{2,3\}$: the value is $2$ for the extensions of Proposition~\ref{prop:known23}, of Theorem~\ref{thm:fam} and for the cubic $K(u)$ of the proof of (d), and for all of them as soon as $\ed_k(\Z/8)=2$. The values of the profile in (b) are exact. A faithful curve $C$ with the full group, with Proposition~\ref{prop:rigidN} for $\ell=2$, already gives the degree bound $\deg D_0\ge4$ in the pencil case (its proof shows $|\ker\rho|\ge4$), which excludes every quadratic witness curve of every genus; Theorem~\ref{thm:N1} extends this to linear systems of every dimension.
\end{remark}

\subsection{Explicit quadratic extensions and the quartic tower}\label{ss:ed2}

By Theorem~\ref{thm:N1}, $\ed_k(\tgen\otimes K')\ge2$ for every extension $K'/K$ of degree $\le3$; so for an explicit quadratic or cubic extension the question is the upper bound. The following extensions have $\ed=2$ exactly. The quartic tower is the first degree at which the profile drops to $1$, and the theorem below says that the naive tower $K(v_1,v_2)$ is not a witness, whereas $K(v_1,v')$ is (Lemma~\ref{lem:tower}).

\begin{lemma}[curves dominated by a fibration]\label{lem:fibration}
Let $\pi\colon V\to B$ be a dominant morphism of integral $k$-varieties whose generic fibre $V_\eta$ is geometrically integral, and let $\phi\colon V\dashrightarrow C$ be a dominant rational map to a smooth projective curve. Then either
\begin{enumerate}[label=\textup{(\roman*)}]
\item there is a dense open subset $B^\circ\subseteq B$ such that for every closed point $b\in B^\circ$ the fibre $V_b$ is integral, $\phi$ restricts to a rational map $V_b\dashrightarrow C$, and this map is dominant; or
\item $k(C)\subseteq k(B)$, so that $C$ is dominated by $B$.
\end{enumerate}
\end{lemma}
\begin{proof}
Consider $\phi_\eta\colon V_\eta\dashrightarrow C_{k(B)}$ over $k(B)$. Suppose first that $\phi_\eta$ is dominant. Since $V_\eta$ is geometrically integral, the fibres $V_b$ are integral for $b$ in a dense open subset of $B$ \cite[IV$_3$, \S9.7]{EGA};
and the domain of definition $U$ of $\phi$, a dense open subset of $V$, meets $V_b$ in a dense open subset for $b$ in a dense open subset of $B$. The map $\Phi=(\pi,\phi)\colon U\to B\times C$ is dominant (its generic fibre over $B$ is $\phi_\eta$), so by Chevalley's theorem $\Phi(U)$ contains a dense open $W\subseteq B\times C$; for $b$ in the dense open image of $W$ in $B$, $\phi(U\cap V_b)\supseteq W_b$ is infinite, hence $\phi|_{V_b}$ is dominant. This is (i). Otherwise $\phi_\eta$ has image a closed point $x$ of $C_{k(B)}$, and $k(C)\subseteq\kappa(x)\subseteq k(V_\eta)=k(V)$ with $\kappa(x)$ finite over $k(B)$. As $V_\eta$ is geometrically integral, $k(V)/k(B)$ is a regular extension, so $k(B)$ is algebraically closed in $k(V)$ and $\kappa(x)=k(B)$; hence $k(C)\subseteq k(B)$.
\end{proof}

\begin{theorem}[explicit extensions with $\ed$ exactly $2$]\label{thm:fam}
$\ed_k(\tgen\otimes K')=2$ for each of the following fields $K'$:
\begin{enumerate}[label=\textup{(\alph*)}]
\item $K'=K(v_f)$ with $f\in s_2+K_2+\wp K$ \textup{(}this contains $K(v_2)$\textup{)};
\item $K'=K(v_1)$ and $K'=K(z_0)=K(v_0)$;
\item the quartic tower $K'=K(v_1,v_2)$.
\end{enumerate}
In particular the quartic tower $K(v_1,v_2)$ is not a witness for $\edd4_k(\tgen)=1$.
\end{theorem}
\begin{proof}
The lower bound is Theorem~\ref{thm:N1} for (a), (b) (degree $\le2$), and for (c) it is proved below.

(a) Write $f=s_2+e+\wp b$, $e\in K_2$, $b\in K$. Over $K'=K(v_f)$, $s_2=\wp(v_f+b)+e$, so $s=(s_0,s_1,e)+\wp V^2(v_f+b)\equiv(s_0,s_1,e)$ by \eqref{eq:Vsplit}; this class is defined over $K_2$, so $\ed_k(\tgen\otimes K')\le2$.

(b) $\edd2\le2$ for these fields by Proposition~\ref{prop:known23} (the class over $K(v_1)$ is defined over the transcendence-degree-$2$ field $k(s_0,s_2+s_1^2+s_1v_1)$, and $K(z_0)$ is Proposition~\ref{prop:intermediate}).

(c) The upper bound is \Bref{1} applied to $K(v_1)\subset K(v_1,v_2)$ and (b). For the lower bound: $L\otimes_KK'$ is a field: $K'/K$ is a $(\Z/2)^2$-extension whose quadratic subfields are $K(v_1)$, $K(v_2)$, $K(v_1+v_2)$, of classes $s_1$, $s_2$, $s_1+s_2$, none $\equiv s_0$ (each has a simple pole at $s_1=\infty$ or $s_2=\infty$, where $s_0$ is regular), so $K'\cap L=K$. Put $y:=v_1+z_1$ and $y':=v_2+z_2$; then $y^2+y=s_1+z_1^2+z_1=z_0^3+z_0^2$ and $y'^2+y'=s_2+z_2^2+z_2=C(z_0,z_1)$ by \eqref{eq:wp2}, so $L(v_1,v_2)=k(z_0,y,z_1,y',z_2)=k(\Sigma\times\A^1_{z_2})$, where
\[
\Sigma:=\{(z_0,y,z_1,y'):\ y^2+y=z_0^3+z_0^2,\ \ y'^2+y'=C(z_0,z_1)\},
\]\[
C(z_0,z_1)=z_1^3+(z_0^3+z_0^2+1)z_1^2+(z_0^3+z_0^2)z_1+z_0^7+z_0^4
\]
($y$ is algebraic over $k(z_0)$, and $[L(v_1):L]=2$ since $z_0^3+z_0^2\notin\wp L$, having odd pole order in $z_0$). $\Sigma$ is an integral surface, fibred over $E_0\colon y^2+y=z_0^3+z_0^2$ via $(z_0,y)$, with generic fibre the affine curve $y'^2+y'=(\text{monic cubic in }z_1)$, which is smooth and geometrically integral of genus $1$ with $j=0$ (Weierstrass form with $a_1=0$, $a_3=1$), hence supersingular. Suppose $\ed_k(\tgen\otimes K')\le1$. By Theorem~\Aref{thm:geom} and Corollary~\ref{cor:jumpsB} there is a smooth projective curve $C$ with faithful $G$-action and a dominant equivariant rational map $\Sigma\times\A^1\dashrightarrow C$ (the descent field $F_0$ has $\trdeg_kF_0=1$; $\trdeg0$ is impossible as the torsor over $K'$ is a field, hence non-split). Lemma~\ref{lem:fibration} for $\mathrm{pr}_2\colon\Sigma\times\A^1\to\A^1$ gives: $C$ is dominated by $\A^1$ (rational), or by $\Sigma$. In the latter case apply Lemma~\ref{lem:fibration} to $\Sigma\to E_0$: either the general fibre maps dominantly to $C$, or $E_0$ does. (Every closed fibre of $\Sigma\to E_0$ is $y'^2+y'=C(b,z_1)$, a Weierstrass equation with $a_1=0$, $a_3=1$ and monic cubic right-hand side, hence smooth of genus $1$ and supersingular, not only the generic fibre.) Either way $C$ is $\PP^1$ (Lemma~\ref{lem:curvesB}(iii)) or, by Lemma~\ref{lem:curvesB}(iv), an elliptic curve isogenous to a supersingular one (a nonconstant morphism of elliptic curves is an isogeny up to translation), hence supersingular; neither carries an automorphism of order $8$ (Lemma~\ref{lem:curvesB}(i),(ii)), contradicting faithfulness.
\end{proof}

\noindent Comment. Over $K(v_1,v_2)$ both $s_1$ and $s_2$ have become $\wp$-exact and still $\ed=2$, the class being $(s_0,0,s_1^2+s_1v_1)$ (Remark~\ref{rem:towerp2}): killing the coordinates $s_1,s_2$ is not the same as killing the upper layers of the class, because of the Witt carry; the layers must be killed in order (Lemma~\ref{lem:tower}).

\subsection{The prime-to-\texorpdfstring{$2$}{2} tail}\label{ss:tail23}

Cubic extensions have degree prime to $p$, so $\edd3{}'(\tgen)$ is the first nontrivial slot of the prime-to-$p$ profile, and Theorem~\ref{thm:23main}(d) says that it equals $2$. The Riemann--Roch bound of Theorem~\ref{thm:B} bounds the first drop from above, and Theorem~\ref{thm:ptp} shows that it is attained.

\begin{corollary}\label{cor:B23}
At $(p,n)=(2,3)$, $d_1^{(p')}(\tgen)=l_3(2)=11$. No faithful $\Z/8$-curve has a free closed point of odd degree $<11$ on its twist, so in particular $5,7,9\notin D(C)$ for every such $C$; and $11\in D(U_3)$, so $\min\bigl(D(U_3)\smallsetminus2\Z\bigr)=11$.
\end{corollary}
\begin{proof}
The first two statements are Theorem~\ref{thm:ptp}. By Theorem~\ref{thm:B}(a) and $l_3(2)=11$ there is a function on $\tw{\tgen}{\Ubar_3}$ with pole divisor exactly $11\infty'$ (proof of Theorem~\ref{thm:B}(c)); its fibres lie in the free locus and have degree $11$. Moreover $11\in D(U_3)$ by Lemma~\ref{lem:hilbertfree}, applied to $y_2+y_0^2y_1$, which has odd pole order $11$.
\end{proof}

Degrees $5$, $7$, $9$ are thus excluded for every faithful curve; what remains open in $D(U_3)$ is whether $13\in D(U_3)$ (Question~\ref{Q4}).

\subsection{Summary at \texorpdfstring{$(2,3)$}{(2,3)} and the remaining questions}\label{ss:status}

\begin{theorem}[the profile at $(2,3)$: summary]\label{thm:status}
\[
\edd1\in\{2,3\},\quad\edd2=\edd3=2,\quad
\edd4=\dots=\edd7=1,\quad\edd8=0;
\]
$d_2\in\{1,2\}$, $d_1=4$, $d_0=8$; for the prime-to-$2$ profile, $\edd1{}'=\edd2{}'=\edd1$, $\edd3{}'=\edd4{}'=2$, $d_2^{(p')}\in\{1,3\}$, $d_1^{(p')}=11$, and $d_0^{(p')}=\infty$. Every faithful $\Z/8$-curve $C$ has $D(C)\cap\{1,2,3\}=\emptyset$; $D(U_3)\supseteq\langle4,6,11\rangle\smallsetminus\{0\}$ and $\delta_\free(U_3)=4$, $5,7,9\notin D(C)$ for every faithful $\Z/8$-curve $C$, while it is not known whether $13\in D(U_3)$. The quartic tower $K(v_1,v_2)$ is not a witness for $\edd4=1$ (Theorem~\ref{thm:fam}(c)), while $K(v_1,v')$ is (Lemma~\ref{lem:tower}).
\end{theorem}
\begin{proof}
Proposition~\ref{prop:known23}, Theorem~\ref{thm:23main}, Corollary~\ref{cor:B23}, Remark~\ref{rem:towerp2} and Lemma~\ref{lem:lift}; $d_0^{(p')}=\infty$ because $\tgen\otimes F'$ is split only if $L$ embeds in $F'$ (Proposition~\Aref{prop:d0}(a)), which forces $8\mid[F':K]$. For $D(U_3)$: $4,6\in D(U_3)$ by Theorem~\ref{thm:23main}(c); every other $d\in\langle4,6,11\rangle\smallsetminus\{0\}$ lies in $D(U_3)$ by Lemma~\ref{lem:hilbertfree}, applied to any function of pole order $d$ if $d$ is odd, and, if $d\ge8$ is even, to a function $f$ of pole order $d$, replaced by $f+y_0$ if $df=0$ (here $y_0$ has pole order $4<d$, and $dy_0=dt\ne0$ as $t=y_0^2+y_0$ is a separating element of $k(\Ubar_3)$).
\end{proof}


\section{Ledet's conjecture and the profile}\label{sec:ledet}

\subsection{What the profile sees}\label{ss:intermediate}
Over every large field $F\supseteq K$ one has $d_1(\tgen\otimes F)=1$ (Fact~\ref{fact:rsd}, Corollary~\ref{cor:large}): the profile of $\tgen\otimes F$ carries nothing beyond its last jump. Over $K$ itself, $\edd1(\tgen)$ is the Ledet indicator, but the slots $\edd d$ for $1<d<p^{n-1}$ are neither the indicator nor determined by any large-field or cohomological invariant: $\mathrm{cd}_p(K)\le1$, and the index and the elementary obstruction of $\tw{\tgen}{X}$ are trivial for every $X$ with a $G$-fixed $k$-point, since then $\tw{\tgen}{X}(K)\ne\emptyset$ (Proposition~\ref{prop:DX}(b)). Theorem~\ref{thm:N1} shows that these slots nevertheless carry information, and the input that makes it work is exactly what a large field destroys: the total space $L$ of $\tgen$ is purely transcendental over $k$, so divisor classes on $C_L$ are constant (Proposition~\ref{prop:rigid}(b)); after base change to a large $F$ the total space $L\otimes_KF$ is no longer of this form. At $(2,3)$ the large-field invariants, $\ed_{k,p}(\Z/8)=1$ and its quantitative form $d_1^{(p')}\le11$, are compatible with both $d_1=2$ and $d_1=4$; Theorem~\ref{thm:23main} shows that $d_1=4$ (see the remark below). In characteristic prime to $|G|$ \cite{PA} the situation is reversed: Sylow divisibility fixes the prime-to-$q$ tail, and the interesting part is the $q$-power part of the profile.

\begin{remark}[what this says about Ledet's conjecture]\label{rem:23ledet}
Theorem~\ref{thm:23main} does not decide $\ed_k(\Z/8)\in\{2,3\}$, and cannot: by monotonicity the slots $\edd d$ with $d>1$ bound $\edd1$ only from below, and $\edd1\ge\edd2=2$ was known. Both $\edd1=2$ and $\edd1=3$ are compatible with $d_1=4$. What the theorem does is fix the part of the profile that does not depend on the conjecture. The truncation tower of Proposition~\ref{prop:trunc} is optimal in degree $2$ (the case $j=1$ of Question~\ref{Q4} at this $(p,n)$); no cubic extension beats it; and, unconditionally, no faithful $\Z/8$-curve has a free point of degree $\le3$ on its twist (Theorem~\ref{thm:23main}(c)). So, whatever the value of $\ed_k(\Z/8)$, faithful curves detect nothing in degrees $\le3$, and the remaining uncertainty sits entirely in dimension $2$: $\edd1\le2$ iff some faithful $\Z/8$-surface has a free $K$-point on its twist, and $d_2\in\{1,2\}$ is the same bit (Proposition~\ref{prop:DX}). A proof of the conjecture must therefore be a statement about surfaces at $d=1$: every slot of the profile other than $\edd1$ is determined, and none of them carries information about $\edd1$.
\end{remark}

\subsection{How the slots are related}\label{ss:logic}

Write $f(n,d):=\edd d_k(\tau_n)$, let $\mathrm{Led}(m)$ be the statement $\ed_k(\Z/p^m)=m$, and let $P(n,1)$ be the statement $f(n,p)=n-1$, the subgroup formula at $j=1$ (the inequality $f(n,p)\le n-1$ is Proposition~\ref{prop:trunc}). Throughout $n\ge2$, $F:=k(s_0,\dots,s_{n-2})$ is the base field of $\tau_{n-1}$, $t:=s_{n-1}$ and $K=F(t)$; $\bar\tau$ denotes the truncation $\tau_{n-1}\otimes_FK$ of $\tau_n$.

\begin{lemma}[purely transcendental base change of the profile]\label{lem:profileconst}
Let $m\ge1$, $F=k(s_0,\dots,s_{m-1})$, $a\in\W_m(F)$ and $t$ transcendental over $F$. Then $\edd d_k(\tau_a\otimes_FF(t))=\edd d_k(\tau_a)$ for all $d\ge1$.
\end{lemma}
\begin{proof}
``$\le$'': an extension $F''/F$ of degree $d'\le d$ gives $F''(t)/F(t)$ of degree $d'$, and $\ed_k(\tau_a\otimes F''(t))=\ed_k(\tau_a\otimes F'')$ by Lemma~\Aref{lem:constant}(i).

``$\ge$'': let $F'/F(t)$ have degree $d'\le d$ and $\ed_k(\tau_a\otimes F')\le e$. By Lemma~\Aref{lem:pi} we may assume $F'/F(t)$ separable, so $F'=F(t)(\theta)$ with $\theta$ integral over $F[t]$ and minimal polynomial $P(t,Y)\in F[t][Y]$ monic in $Y$, irreducible in $F[t,Y]$. Put $R=F[t][\theta]\cong F[t][Y]/(P)$, free of rank $d'$ over $F[t]$. By hypothesis $a=b+\wp y$ in $\W_m(F')$ with $b\in\W_m(F_0)$, $F_0=k(f_1,\dots,f_r)\subseteq F'$, $\trdeg_kF_0=e$, and $y\in\W_m(F')$. Let $A\subseteq F_0$ be the $k$-subalgebra generated by the coordinates of $b$; it is a domain of dimension $\le\trdeg_kF_0=e$. Choose $0\ne h\in R$ such that $A\subseteq R[1/h]$ and the coordinates of $y$ lie in $R[1/h]$. The identity $a=b+\wp y$ holds in $\W_m(R[1/h])$. For $c\in F$ with $N_{R/F[t]}(h)(c)\ne0$, evaluation $t\mapsto c$ is an $F$-algebra map $\mathrm{ev}_c\colon R[1/h]\to R/(t-c)=F[Y]/(P(c,Y))$. The field $F$ is Hilbertian \cite[Thm.~13.4.2]{FJ}, so $P(c,Y)$ is irreducible for $c$ in a Hilbert subset of $F$, which stays non-empty after removing the finitely many roots of $N_{R/F[t]}(h)$; for such $c$, $F'_c:=R/(t-c)$ is a field of degree $d'$ over $F$. Applying $\mathrm{ev}_c$ gives $a=\mathrm{ev}_c(b)+\wp\,\mathrm{ev}_c(y)$ in $\W_m(F'_c)$, because $a$ has coordinates in $F$. The coordinates of $\mathrm{ev}_c(b)$ lie in the image $\mathrm{ev}_c(A)$, a quotient of $A$, so the subfield of $F'_c$ they generate over $k$ has transcendence degree $\le\dim A\le e$. Hence $\ed_k(\tau_a\otimes F'_c)\le e$ and $\edd d_k(\tau_a)\le\edd d_k(\tau_a\otimes F(t))$.
\end{proof}

\begin{proposition}[the Witt layers of the profile]\label{prop:layers}
For $n\ge2$ and $d\ge1$:
\begin{enumerate}[label=\textup{(\alph*)}]
\item $f(n-1,d)\le f(n,d)\le f(n-1,d)+1$;
\item $f(n,pd)\le f(n-1,d)$.
\end{enumerate}
Consequently, with $\delta_m(d):=f(m,d)-f(m-1,d)$ \textup{(}$f(0,d):=0$; $\delta_1(d)=f(1,d)\in\{0,1\}$ since $\ed_k(\Z/p)=1$, and $\delta_m(d)\in\{0,1\}$ for $m\ge2$ by \textup{(a)}\textup{)}, we have $f(n,p^j)=\sum_{m=j+1}^n\delta_m(p^j)$, and $P(n,j)$, the statement $f(n,p^j)=n-j$, holds iff $\delta_m(p^j)=1$ for all $j<m\le n$: every Witt layer is strict.
\end{proposition}
\begin{proof}
(a) By Lemma~\Aref{lem:pi} it suffices to consider separable $F'/K$ of degree $\le d$. Truncation is a ring map commuting with $\wp$, so a descent of $\tau_n\otimes F'$ truncates to a descent of $\bar\tau\otimes F'$ of no larger transcendence degree, and $\ed_k(\tau_n\otimes F')\ge\ed_k(\bar\tau\otimes F')\ge\edd d_k(\bar\tau)=f(n-1,d)$ by Lemma~\ref{lem:profileconst}. Conversely $f(n,d)\le\edd d_k(\bar\tau)+1=f(n-1,d)+1$ by Lemma~\ref{lem:lift} and Lemma~\ref{lem:profileconst}.

(b) By Proposition~\Aref{prop:shape}(b) with $F'=K_1$, $f(n,pd)\le\edd d_k(\tau_n\otimes K_1)$. By Lemma~\Aref{lem:ind} applied to each extension of $K_1$, $\edd d_k(\tau_n\otimes K_1)=\edd d_k(L/K_1)$, and $L/K_1=\tau_{n-1}\otimes_{F''}F''(z_0)$ with $F''=k(s'_0,\dots,s'_{n-2})$ (proof of Proposition~\ref{prop:intermediate}), whose profile is $f(n-1,d)$ by Lemma~\ref{lem:profileconst}.

The last statement: $f(j,p^j)=0$ since $\tau_j$ is split by an extension of degree $p^j$, so $f(n,p^j)=\sum_{m=j+1}^n\delta_m(p^j)\le n-j$, with equality iff all $\delta_m(p^j)=1$.
\end{proof}

\begin{remark}\label{rem:layersN1}
$\delta_{j+1}(p^j)=1$ trivially ($f(j+1,p^j)\ge1$ by Proposition~\Aref{prop:d0}(c) and $\le1$ by Proposition~\ref{prop:trunc}), and $\delta_{j+2}(p^j)=1$ is exactly Theorem~\ref{thm:N1} ($f(j+2,p^j)=2$). The first layer not covered is $\delta_{j+3}(p^j)$.
\end{remark}

\begin{theorem}[$P(n,1)$ and Ledet's conjecture]\label{thm:Pn1}
Let $n\ge2$, and let $D'(n)$ be the statement: $\ed_k(\tau_n\otimes_KF')\ge n-1$ for every separable extension $F'/K$ with $2\le[F':K]\le p$ and $L\otimes_KF'$ a field. Then
\[
P(n,1)\iff\mathrm{Led}(n-1)\ \text{and}\ D'(n).
\]
\end{theorem}
\begin{proof}
Classification. Let $F'/K$ be separable with $[F':K]\le p$. Fix an embedding of $F'$ in an algebraic closure containing $L$. Since $L/K$ is Galois with cyclic group $G$, $L\cap F'=K_i$ for some $0\le i\le n$; $K_i\subseteq F'$ gives $p^i\le[F':K]\le p$. If $i=1$ then $F'=K_1$; if $i=0$ then $L\otimes_KF'$ is a field, since $[LF':F']=[L:L\cap F']$. So $F'=K_1$, or $L\otimes_KF'$ is a field (and $F'=K$ is the case of degree $1$).

The field $K_1$. By Lemma~\Aref{lem:ind} and the proof of Proposition~\ref{prop:intermediate}, $\ed_k(\tau_n\otimes K_1)=\ed_k(L/K_1)$, and $L/K_1=\tau_{n-1}\otimes_{F''}F''(z_0)$ with $z_0$ transcendental over $F''$; so $\ed_k(\tau_n\otimes K_1)=\ed_k(\tau_{n-1})=\ed_k(\Z/p^{n-1})$ by Lemma~\Aref{lem:constant}(i) and Lemma~\ref{lem:versal}(c).

($\Rightarrow$) $D'(n)$ holds because $f(n,p)\ge n-1$. Moreover $n-1=f(n,p)\le\ed_k(\tau_n\otimes K_1)=\ed_k(\Z/p^{n-1})\le n-1$, so $\mathrm{Led}(n-1)$.

($\Leftarrow$) Let $F'/K$ have degree $\le p$; by Lemma~\Aref{lem:pi} we may assume it separable. If $F'=K_1$, then $\ed_k(\tau_n\otimes K_1)=\ed_k(\Z/p^{n-1})=n-1$ by $\mathrm{Led}(n-1)$. If $F'=K$, then $\ed_k(\tau_n)\ge\ed_k(\bar\tau)=\ed_k(\tau_{n-1})=n-1$ by truncation, Lemma~\Aref{lem:constant}(i) and $\mathrm{Led}(n-1)$. In the remaining cases $D'(n)$ applies. Hence $f(n,p)\ge n-1$, and $f(n,p)\le n-1$ is Proposition~\ref{prop:trunc}.
\end{proof}

\begin{corollary}\label{cor:Pn1}
\begin{enumerate}[label=\textup{(\alph*)}]
\item $P(n,1)$ implies $\mathrm{Led}(n-1)$ and $P(n-1,1)$. In particular $P(n,1)$ for all $n$ implies Ledet's conjecture for all $n$, and the subgroup formula at $j=1$ for $\Z/p^4$ contains $\ed_k(\Z/p^3)=3$.
\item $D'(2)$ is trivial and $D'(3)$ holds for every $p$ \textup{(}Theorem~\ref{thm:N1}\textup{)}; hence $P(3,1)$ holds for every $p$. For $n\ge4$, $D'(n)$ is open.
\end{enumerate}
\end{corollary}
\begin{proof}
(a) The first claim and $\mathrm{Led}(n-1)$ are Theorem~\ref{thm:Pn1}. For $P(n-1,1)$: $n-1=f(n,p)\le f(n-1,p)+1$ by Proposition~\ref{prop:layers}(a), and $f(n-1,p)\le n-2$ by Proposition~\ref{prop:trunc}. (b) For $n=2$, $L\otimes_KF'$ is a field, so $\tau_2\otimes F'$ is a non-split torsor and $\ed\ge1$. For $n=3$, Theorem~\ref{thm:N1} gives $\ed_k(\tau_3\otimes F')\ge2$ for every $F'$ of degree $<p^2$; and $\mathrm{Led}(2)$ holds ($\ed_k(\Z/p^2)\le2$ by Lemma~\ref{lem:versal}(c) and $\ge2$ by Remark~\ref{rem:rigid-d1}).
\end{proof}

\begin{remark}\label{rem:Dn}
$D'(n)$ is not known to be easier than $\mathrm{Led}(n-1)$ for $n\ge4$, nor to imply it. Partial and conditional results: if $F'\supseteq K$ is of the form $F_1(w)$, purely transcendental over a finite extension $F_1/F$, with $\ed_k(\tau_{n-1}\otimes_FF_1)=n-1$, then $\ed_k(\tau_n\otimes F')\ge\ed_k(\tau_{n-1}\otimes_FF_1(w))=\ed_k(\tau_{n-1}\otimes_FF_1)=n-1$ by truncation and Lemma~\Aref{lem:constant}(i); and for $2\le[F':K]<p$, $D'(n)$ would follow from the plateau $f(n-1,p-1)=n-1$ by Proposition~\ref{prop:layers}(a); but that plateau implies $\mathrm{Led}(n-1)$, which is open for $n-1\ge3$, so this is conditional. What is missing is the strictness of the top layer, $\delta_n(p)=1$, for extensions of degree $\le p$ disjoint from $L$.
\end{remark}

\begin{proposition}[Ledet's conjecture as a lifting statement]\label{prop:Led3lift}
Let $n\ge3$, $F=k(s_0,\dots,s_{n-2})$, and let $K_1'=F(z_0)\subset L$ with $z_0^p-z_0=s_0$, so that $K_1=K(z_0)=K_1'(t)$. Then
\[
\ed_k(\tau_n\otimes K_1)=\ed_k(\Z/p^{n-1}),\qquad\ed_k(\bar\tau\otimes K_1)=\ed_k(\tau_{n-1}\otimes_FK_1')=\ed_k(\Z/p^{n-2}).
\]
Consequently $\ed_k(\Z/p^{n-1})=\ed_k(\Z/p^{n-2})+1$ holds if and only if passing from the truncation $\bar\tau$ to $\tau_n$ over $K_1$ raises the essential dimension by exactly one: $\ed_k(\tau_n\otimes K_1)=\ed_k(\bar\tau\otimes K_1)+1$. In particular $\mathrm{Led}(3)$, i.e.\ $\ed_k(\Z/8)=3$, is equivalent to $\ed_k(\tau_4\otimes K_1)=\ed_k(\bar\tau\otimes K_1)+1$ at $(p,n)=(2,4)$, and, given $\mathrm{Led}(n-2)$, $\mathrm{Led}(n-1)$ is equivalent to the same statement for $n$.
\end{proposition}
\begin{proof}
The first identity is proved in the proof of Theorem~\ref{thm:Pn1}. For the second, $\bar\tau\otimes_KK_1=(\tau_{n-1}\otimes_FK_1')\otimes_{K_1'}K_1'(t)$, so $\ed_k(\bar\tau\otimes K_1)=\ed_k(\tau_{n-1}\otimes_FK_1')$ by Lemma~\Aref{lem:constant}(i), and $K_1'$ is the degree-$p$ subfield $F(z_0)$ of the total space of $\tau_{n-1}$, so $\ed_k(\tau_{n-1}\otimes_FK_1')=\ed_k(\Z/p^{n-2})$ by the same argument applied to $\tau_{n-1}$ (Proposition~\ref{prop:intermediate} and Lemma~\Aref{lem:constant}(i)); this needs $n-1\ge2$.
\end{proof}

\begin{remark}[what Theorem~\ref{thm:N1} does not decide]\label{rem:N1ledet}
(i) By Proposition~\ref{prop:intermediate} and Lemma~\Aref{lem:constant}(i), $\ed_k(\tgen\otimes K_1)=\ed_k(\tau_{n-1})$ ($K_1$ is purely transcendental over the base field of the generic $\Z/p^{n-1}$-torsor). Hence $\edd p_k(\tau_n)=n-1$ would force $\ed_k(\Z/p^{n-1})=n-1$: the subgroup formula at $j=1$ for $\Z/p^4$ contains Ledet's conjecture for $\Z/p^3$, and the formula at $j=1$ for all $n$ contains Ledet's conjecture for all $n$ (it gives $\ed_k(\Z/p^n)\ge n-1$ only). Theorem~\ref{thm:N1} therefore settles exactly the slots with $n-j\le2$; the slots with $n-j\ge3$ remain open, the first being $\ed_k(\Z/8)$ at $(2,3)$. (ii) The proof is a statement about curve witnesses, i.e.\ $\ed\le1$. For witnesses of dimension $\ge2$ the first two steps survive in the form of constancy of the class of the point cycle in $\mathrm{CH}_0$, but there is no complete linear system on which $G$ acts, which is what Proposition~\ref{prop:rigidN} uses. In dimension two the class of the zero-cycle $\Delta$ on $X_L$ is still constant in $\mathrm{CH}_0$ (the map $\mathrm{CH}_0(X)\to\mathrm{CH}_0(X_L)$ is surjective for $L/k$ purely transcendental, by homotopy invariance and localisation), hence $G$-invariant. But for $X$ with $A_0(X)=0$, for instance a rational surface, this statement is empty, and for $A_0(X)\ne0$ the class moves in an infinite-dimensional family. We therefore see no analogue of Proposition~\ref{prop:rigidN}, and a lower bound for $\edd d$ in the range $\ed\le2$ would need new input.
\end{remark}

\subsection{Which surfaces decide Ledet's conjecture at $n=3$}\label{ss:surfaces}

\begin{proposition}\label{prop:DX}
\begin{enumerate}[label=\textup{(\alph*)}]
\item Let $X$ be an irreducible quasi-projective surface with a faithful $G$-action. The following are equivalent: \textup{(i)} $1\in D(X)$; \textup{(ii)} $X$ is $\tgen$-versal; \textup{(iii)} there is a $G$-compression $\W_3\dashrightarrow X$. Moreover, $\delta_\free(X)\ge2$ for every smooth projective surface $X$ with a faithful $G$-action iff $\ed_k(\Z/p^3)=3$.
\item Let $X$ be a smooth projective surface with a faithful $G$-action and a $G$-fixed $k$-point. Then $\tw{\tgen}{X}(K)\ne\emptyset$ (the constant map to the fixed point, Fact~\Aref{fact:rs}(a)); in particular the index and elementary obstruction of $\tw{\tgen}{X}$ are trivial.
\item $d_2(\tgen)\le p$, and $d_2(\tgen)=\min_Z\delta_\free(Z)$ over quasi-projective generically free primitive $G$-varieties $Z$ of dimension $\le2$; the minimum over smooth projective surfaces with faithful $G$-action equals $1$ iff $\ed_k(\Z/p^3)\le2$.
\end{enumerate}
\end{proposition}
\begin{proof}
(a) (iii)$\Rightarrow$(ii): let $\phi\colon\W_3\dashrightarrow X$ be a $G$-compression and $U\subseteq X$ a dense $G$-stable open subset. The generic point of $\W_3$ maps to the generic point of $X$, which lies in $U$; so $\phi$ restricted to the generic point is a $G$-equivariant morphism $\Spec L\to U$, i.e.\ a $K$-point of $\tw{\tgen}{U}$ (Fact~\Aref{fact:rs}(a)), and $U$ is $\tgen$-weakly versal. (ii)$\Rightarrow$(i): as $G$ acts faithfully on the irreducible $X$, $X^{\free}$ is a dense $G$-stable open subset; its weak versality gives a $K$-point of $\tw{\tgen}{(X^{\free})}$, i.e.\ a free $K$-point. (i)$\Rightarrow$(iii): a free $K$-point is a $G$-equivariant $f\colon\Spec L\to X^{\free}$, i.e.\ (as $L=k(\W_3)$) a $G$-equivariant rational map $\W_3\dashrightarrow X$; the closure $Y$ of its image is irreducible and $G$-stable with faithful $G$-action (Lemma~\ref{lem:imageB}). If $\dim Y\le1$, then $\tw{\tgen}{Y}$ has a free $K$-point and $\ed_k(\tgen)\le1$ by Lemma~\Aref{lem:freepoint}, contradicting Remark~\ref{rem:rigid-d1}; so $Y=X$ and $f$ is a compression.

For the last statement: if $1\in D(X)$ for some $X$ then $\ed_k(\Z/p^3)=\ed_k(\tgen)\le2$ by Lemma~\Aref{lem:freepoint} and Lemma~\ref{lem:versal}. Conversely let $\ed_k(\Z/p^3)\le2$. By Lemma~\ref{lem:versal} and the identity $\ed_k(Y)=\ed_k(\tau_Y)$ of \cite[Section~\ref{A-ss:ed}]{PA} there is a $G$-compression $\W_3\dashrightarrow Z$ onto an irreducible generically free $G$-variety of dimension $\le2$, and $\dim Z=2$ since $\ed_k(\tgen)\ge2$. Replace $Z$ by a dense $G$-stable affine open subset, embed it $G$-equivariantly in a finite-dimensional representation (a $G$-stable finite-dimensional subspace of $k[Z]$ generating it) and take the closure in the projective completion: a projective $G$-surface $\bar Z$ birational to $Z$. The resolution of $\bar Z$ obtained by alternately normalising and blowing up the singular locus \cite{Lip78}
terminates in a regular surface by \cite{Lip78}; each step (normalisation, blow-up of the reduced singular locus) is canonical, i.e.\ functorial for isomorphisms of surfaces, so every automorphism of $\bar Z$ lifts to a unique automorphism of the resolution and the resolution carries a $G$-action; it is a smooth projective surface $X$ with faithful $G$-action, $G$-equivariantly birational to $Z$. The compression $\W_3\dashrightarrow X$ gives $1\in D(X)$ by (iii)$\Rightarrow$(i).

(b) The fixed point gives the constant map, a $K$-point of $\tw{\tgen}{X}$ (Fact~\Aref{fact:rs}(a)).

(c) is Corollary~\ref{cor:jumpsB} and Proposition~\ref{prop:S0}(b) (with $n=3$, $\dim S_0=2$), together with (a) for the value $1$.
\end{proof}

\subsection{The characteristic-$p$ mirror of the subgroup formula}

\begin{remark}[a characteristic-$p$ analogue of the subgroup formula]\label{rem:mirror}
In characteristic $\ne p$ the subgroup formula of \cite[Section~\ref{A-sec:sylow}]{PA} (Florence \cite{Fl08}, Karpenko--Merkurjev) bounds the profile by the minimal faithful representations of subgroups. In characteristic $p$ representations are the wrong envelope: $\rdim_k(\Z/p^n)=p^{n-1}+1$ (the generator acts by a unipotent matrix, and a unipotent matrix of order $p^n$ needs a Jordan block of size $>p^{n-1}$, which is also enough), and $\rdim_k((\Z/p)^n)=2$ (a one-dimensional representation of a $p$-group is trivial, and an $\F_p$-subspace of $(k,+)$ embeds in the unipotent group $\bigl(\begin{smallmatrix}1&*\\0&1\end{smallmatrix}\bigr)$), whereas $\ed_k((\Z/p)^n)=1$ (for a finite subgroup $\Gamma\subset(k,+)$ the additive polynomial $P_\Gamma(x)=\prod_{\gamma\in\Gamma}(x-\gamma)$ is an isogeny with kernel $\Gamma$, so every $\Gamma$-torsor over $F$ has a class $a\in F$ and descends to $k(a)$) and $\ed_k(\Z/p^n)\le n$ (Lemma~\ref{lem:versal}). The natural envelope is the embedding $\Z/p^n\hookrightarrow\W_n$ by Witt translation (Section~\ref{ss:gen}), which restricts on $\langle g^{p^j}\rangle$ to the Witt translation of $\W_{n-j}$ (Proposition~\ref{prop:intermediate}). With $\dim\W_{n-j}$ in place of $\rdim$ the subgroup formula reads
\[
\edd{p^j}_k(\tgen)=\dim\W_{n-j}=n-j\qquad(0\le j\le n),
\]
which is Ledet's conjecture at every level (Questions~\ref{Q2} and~\ref{Q3}; the case $j=n-2$ is Theorem~\ref{thm:N1}): Proposition~\ref{prop:trunc} gives ``$\le$'', and by Proposition~\ref{prop:intermediate} the fixed field $K_j$ realises exactly $n-j$ iff Ledet's conjecture holds for $\Z/p^{n-j}$. One template thus covers both regimes: the profile at degree $d$ is the minimal envelope dimension over subgroups of index $\le d$. In the Karpenko--Merkurjev setting it is proved on the prime-to-$p$ plateau (\cite[Proposition~\ref{A-prop:KMplateau}]{PA}) and, for abelian $G$ of exponent $p^a$ with $\mu_{p^a}\subset k$, in all degrees (\Aref{cor:subgroupabelian}); in general it is open in $p$-power degrees in both settings.
\end{remark}


\section{Open questions}\label{sec:questions}

This section is written to be read on its own. Throughout, $k$ is algebraically closed of characteristic $p$, $G=\Z/p^n=\langle g\rangle$, $L=k(z_0,\dots,z_{n-1})$ is the function field of $\W_n$ with $g$ acting by the Witt translation, $K=L^G=k(s_0,\dots,s_{n-1})$, and $\tau_n=\tgen$ is the generic $G$-torsor $\Spec L\to\Spec K$, whose class in $\W_n(K)/\wp\W_n(K)$ is $s=(s_0,\dots,s_{n-1})$. For $d\ge1$ put
\[
f(n,d):=\edd d_k(\tau_n)=\min\{\ed_k(\tau_n\otimes_KF'):F'/K\text{ a field extension},\ [F':K]\le d\}.
\]
This is a non-increasing function of $d$, and $f(n,1)=\ed_k(\Z/p^n)$. For a faithful $G$-curve $C$, $D(C)$ denotes the set of degrees of closed points of the free locus of the twist $\tw{\tau_n}{C}$. \emph{Ledet's conjecture} is $\ed_k(\Z/p^n)=n$; we write $\mathrm{Led}(n)$ for it.

\subsection{What is known}
\begin{enumerate}[label=\textup{(K\arabic*)},leftmargin=2.4em]
\item $f(n,d)=0$ iff $d\ge p^n$, and $f(n,d)=1$ for $p^{n-1}\le d<p^n$ (the truncation tower, Proposition~\ref{prop:trunc}, and the last jump degree, Proposition~\Aref{prop:d0}).
\item $f(n,d)\ge2$ for every $d<p^{n-1}$ (Theorem~\ref{thm:N1}). Together with (K1): the first jump degree is $d_1(\tau_n)=p^{n-1}$.
\item $f(n,p^j)\le n-j$ (Proposition~\ref{prop:trunc}), with equality for $j=n-1$ and $j=n-2$. Hence $f(n,d)=2$ for $p^{n-2}\le d<p^{n-1}$.
\item At $d=1$: $2\le\ed_k(\Z/p^n)\le n$ for $n\ge2$ (Remark~\ref{rem:rigid-d1} and Lemma~\ref{lem:versal}); $\mathrm{Led}(n)$ is known for $n\le2$ and open for $n\ge3$.
\item For degrees prime to $p$ the first jump is $d_1^{(p')}(\tau_n)=l_n(p)=1+(p-1)\sum_{j=1}^{n-1}p^{2j-1}$ (Theorem~\ref{thm:ptp}).
\item The profile is completely known for $(p,n)=(2,1),(2,2),(3,1),(3,2)$; at $(2,3)$ and $(3,3)$ it is known in every slot except the first one (at $(2,3)$) or the first two (at $(3,3)$), see Table~\ref{tab:preview}.
\end{enumerate}
So what is unknown is the part of the profile with $d<p^{n-2}$, for $n\ge3$, and in particular the value at $d=1$, which is Ledet's conjecture. The questions below are ordered from the most specific to the most general.

\subsection{Questions}

\begin{question}[the smallest open case of Ledet's conjecture]\label{Q1}
Is $\ed_k(\Z/8)=2$ or $3$, for $k$ algebraically closed of characteristic $2$?
\end{question}
At $(p,n)=(2,3)$ every slot of the profile is determined except $f(3,1)=\ed_k(\Z/8)\in\{2,3\}$ (Theorem~\ref{thm:status}); so the profile reduces Ledet's conjecture in this case to a single bit. Equivalent forms: $d_2(\tau_3)=1$ iff $\ed_k(\Z/8)=2$; $\ed_k(\Z/8)\le2$ iff some smooth projective surface with a faithful $\Z/8$-action has a free $K$-point on its twist (Proposition~\ref{prop:DX}); and $\ed_k(\Z/8)=3$ iff, over the field $K_1=K(z_0)$ of $\tau_4$ at $(2,4)$, passing from the truncation to $\tau_4$ raises the essential dimension by exactly one (Proposition~\ref{prop:Led3lift}). Faithful curves cannot decide it: no faithful $\Z/8$-curve has a free point of degree $\le3$ on its twist (Corollary~\ref{cor:curvepts}), so a proof must handle surfaces, where the Abel--Jacobi rigidity of Section~\ref{ss:rigid} is not available.

\begin{question}[the subgroup formula at $j=1$]\label{Q2}
Is $f(n,p)=n-1$ for every $n$, i.e.\ is the truncation tower optimal among extensions of degree $p$?
\end{question}
The inequality $f(n,p)\le n-1$ is Proposition~\ref{prop:trunc}; equality is known for $n\le3$ (for $n=3$ by Theorem~\ref{thm:N1}). By Theorem~\ref{thm:Pn1}, for $n\ge2$ the equality $f(n,p)=n-1$ is equivalent to $\mathrm{Led}(n-1)$ together with the statement $D'(n)$: for every separable extension $F'/K$ with $2\le[F':K]\le p$ and $L\otimes_KF'$ a field, $\ed_k(\tau_n\otimes F')\ge n-1$. Thus a positive answer for $n=4$ contains $\mathrm{Led}(3)$ (Question~\ref{Q1}), and a positive answer for all $n$ contains Ledet's conjecture. The first case not covered by anything proved here is $D'(4)$ for $p=2$: for every quadratic extension $K(v)$ of $K=k(s_0,\dots,s_3)$, $v^2+v=f$, disjoint from $L$ (that is $f\notin\wp K+\F_2s_0$), is $\ed_k(\tau_4\otimes K(v))\ge3$? Partial results are in Remark~\ref{rem:Dn}.

\begin{question}[the whole profile]\label{Q3}
Is $f(n,d)=n-\lfloor\log_pd\rfloor$ for all $1\le d<p^n$?
\end{question}
This says that each Witt layer costs exactly one parameter, and it is the conjecture suggested by the cases computed so far: it holds for $d\ge p^{n-2}$ by (K1)--(K3), at $(2,2)$, $(3,2)$, at $(2,3)$ except at $d=1$, and at $(3,3)$ except at $d=1,2$. By Proposition~\ref{prop:layers}, $f(n,p^j)=n-j$ for all $j$ iff every Witt layer is strict, and the case $d=1$ is Ledet's conjecture while $d=p$ is Question~\ref{Q2}. Its first open instance is again $f(3,1)$, i.e.\ Question~\ref{Q1}.

\begin{question}[degrees of free points on twisted curves]\label{Q4}
At $(p,n)=(2,3)$, is $13\in D(U_3)$, where $U_3$ is the Artin--Schreier--Witt curve with its $\Z/8$-action? More generally, determine the set $D(C)$ for each faithful $\Z/8$-curve $C$.
\end{question}
The Weierstrass semigroup of the fixed point on $\Ubar_3$ is $\langle4,6,11\rangle$, with gaps $\{1,2,3,5,7,9,13\}$, and $D(U_3)\supseteq\langle4,6,11\rangle\smallsetminus\{0\}$ (Appendix~\ref{app:genus}, Theorem~\ref{thm:status}). On the other side, no $D(C)$ contains $1$, $2$, $3$, $5$, $7$ or $9$ (Corollary~\ref{cor:curvepts}, Theorem~\ref{thm:ptp}). Of the gaps, only $13$ is undecided. This is a finite question, but the argument that excludes $5,7,9$ needs $p\nmid d$ and an invariant divisor of degree $<l_3(2)=11$, and it says nothing about degrees above $11$.

\subsection{Directions without results here}
Two natural extensions are not treated in this paper. Other groups: for a finite $p$-group other than $\Z/p^n$ in characteristic $p$ (for instance $\Z/p^a\times\Z/p^b$, or non-abelian $p$-groups), the profile of the generic torsor is unknown beyond $\ed_{k,p}\le1$; the elementary abelian case is trivial since $\ed_k((\Z/p)^n)=1$. Other base fields: we assume $k$ algebraically closed throughout; the profile of $\tau_n$ over a non-closed field $k$ (for example a finite field), and how its jump degrees depend on $k$, is not studied.


\appendix

\section{Witt-vector formulas at \texorpdfstring{$p=2$}{p=2}}\label{app:witt}

The addition and $\wp$ on $\W_3$ for $p=2$ are \eqref{eq:add2}--\eqref{eq:wp2} of Section~\ref{ss:witt}. All formulas of this appendix are polynomial identities over $\F_2$, derived from the ghost components $w_0=x_0$, $w_1=x_0^2+2x_1$, $w_2=x_0^4+2x_1^2+4x_2$ over $\Z$ and reduced modulo $2$; each can be checked by expanding both sides. As an independent check, a short sympy script verifying all of them, together with the reductions of Appendix~\ref{app:genus} and the series of Remark~\ref{rem:large}, is provided as ancillary material (\texttt{addendum/check\_paperB.py}); it is not needed for any proof.

Negation. $-(a_0,a_1,a_2)=(a_0,\ a_1+a_0^2,\ a_2+a_1^2+a_0^2a_1+a_0^4)$; in particular $-x\ne x$ in general, but $\{s-\wp x\}=\{s+\wp x\}$ as $x$ ranges over $\W_3(K)$, since $\wp(-x)=-\wp(x)$.

The general translate. For $s=(s_0,s_1,s_2)$ and $x=(x_0,x_1,x_2)\in\W_3(K)$, $s+\wp x=(t,r_1,r_2)$ with
\begin{align*}
t&=s_0+x_0^2+x_0,\\
r_1&=s_1+x_0(x_0+1)(x_0+s_0)+\wp(x_1),\\
r_2&=s_2+\Phi_2(s_0,s_1,x_0,x_1)+\wp(x_2),
\end{align*}
where (substituting $b=\wp x$, $b_0=x_0^2+x_0$, $b_1=x_1^2+x_1+x_0^3+x_0^2$, into \eqref{eq:add2} gives $\Phi_2=C(x_0,x_1)+s_1b_1+s_0b_0(s_1+b_1)+s_0^3b_0+s_0b_0^3$, and expanding gives the displayed formula)
\begin{equation}\label{eq:Phi2}
\begin{aligned}
\Phi_2={}&s_0^3(x_0^2+x_0)+s_0s_1(x_0^2+x_0)+s_0(x_0^6+x_0^4)+s_0(x_0^2+x_0)(x_1^2+x_1)\\
&+s_1(x_0^3+x_0^2)+s_1(x_1^2+x_1)+x_0^7+x_0^4+(x_0^3+x_0^2)(x_1^2+x_1)+x_1^3+x_1^2 .
\end{aligned}
\end{equation}
Thus the family of fibres of $u$ in Theorem~\ref{thm:n2}(b) is the $x_1$-independent part of $r_1$:
$(s+\wp(c,x_1))_1-\bigl(c(c+1)(c+s_0)+s_1\bigr)=\wp(x_1)$; and
$(s-\wp x)_1-(s+\wp x)_1=x_0^4+x_0^2\in\wp K$. The classes of Remark~\ref{rem:towerp2} (the class $s$ over $K(v_1)$, $K(v_1,v_2)$ and $K(v_1,v')$) are the specialisations $x=(0,v_1,0)$, $x=(0,v_1,v_2)$, $x=(0,v_1,v')$ of the general translate. For instance $\wp(0,v_1,0)=(0,s_1,v_1^3+v_1^2)$ by \eqref{eq:wp2}, whence $(s+\wp(0,v_1,0))_1=0$ and $(s+\wp(0,v_1,0))_2=s_2+s_1^2+v_1(v_1^2+v_1)=s_2+s_1^2+s_1v_1$; and $\Phi_2(s_0,s_1,0,v_1)=s_1^2+s_1v_1$. Over $K(v_1,v_2)$, $(s+\wp(0,v_1,v_2))_2=s_2+\Phi_2(s_0,s_1,0,v_1)+\wp(v_2)=s_1^2+s_1v_1$, and over $K(v_1,v')$ the same coordinate is $s_2+s_1^2+s_1v_1+\wp(v')=0$.

\section{Genus, ramification and Weierstrass semigroup of \texorpdfstring{$\Ubar_n$}{U\_n}}
\label{app:genus}

Throughout, $\Ubar_n$ is the smooth projective model of the Artin--Schreier--Witt curve $U_n=\{y\in\W_n:\wp(y)\in[\A^1]\}$, $\infty\in\Ubar_n$ its unique $G$-fixed point, $X:=\tw{\tgen}{\Ubar_n}$ and $\infty'\in X(K)$ the twist of $\infty$ (Proposition~\ref{prop:tail}(c)).

\subsection*{Ramification at $\infty$}
$\Ubar_n\to\PP^1_t$, $y\mapsto t=y_0^p-y_0$, is a $\Z/p^n$-cover, \'etale over $\A^1$ and totally ramified at $\infty$, of Witt class $[t]$, and $v_\infty(t)=-1$. For a cyclic extension of $k((1/t))$ given by a Witt vector $x=(x_0,\dots,x_{n-1})$ whose coordinates are reduced (each $x_i$ is regular at $\infty$ or has a pole of order $m_i$ prime to $p$), the upper ramification breaks are $u_j=\max\{p^{\,j-1-i}m_i:0\le i\le j-1,\ x_i\text{ not regular}\}$, $1\le j\le n$ (the formula of Schmid, Kanesaka--Sekiguchi and Thomas \cite{Sch37,KS79,Th} for finite residue fields, and of Elder--Keating \cite[Thm.~2.3]{EK25} for any perfect residue field, applied to the degree-$p^j$ quotient, whose Witt vector $(x_0,\dots,x_{j-1})$ is again reduced; the top break is also \cite[Thm.~1]{Ga}).
For $x=[t]$ only $x_0=t$ is non-regular, with $m_0=1$, so $u_j=p^{\,j-1}$ (the cover and the class $[t]$ are defined over $\F_p$, so the finite-residue-field case suffices). Herbrand's function \cite[IV, \S3]{Se2} converts these into the lower breaks

\[
l_1=u_1=1,\qquad l_{j+1}=l_j+p^j(u_{j+1}-u_j)=l_j+(p-1)p^{2j-1},
\]
which gives \eqref{eq:ln}. For $n=2$ the value $l_2(p)=p^2-p+1$ is confirmed by the direct computation in the proof of Theorem~\ref{thm:B}; for $(p,n)=(2,3)$ it is confirmed by the following computation by hand.

The tower $U_3\to U_2=E_0\to\PP^1_{y_0}\to\PP^1_t$ has layers $y_0^2+y_0=t$, $y_1^2+y_1=y_0^3+y_0^2$ and $y_2^2+y_2=C(y_0,y_1)$, with $C(a_0,a_1)$ the polynomial in \eqref{eq:wp2}. On $E_0$, $y_0$ and $y_1$ have pole orders $2$ and $3$ at $\infty$, and using $y_1^2=y_1+y_0^3+y_0^2$,
\[
C(y_0,y_1)=y_0^7+y_0^6+(y_0^3+y_0^2)y_1,
\]
of pole order $14$. One reduction step gives
\[
C(y_0,y_1)+\wp(y_0^2y_1)=y_0^4y_1+y_0^3y_1,
\]
of odd pole order $8+3=11$. Hence the Artin--Schreier extension $U_3/U_2$ has break $11$; as lower numbering passes to subgroups, $l_3=11$.

\subsection*{Genus}
Let $l_0:=-1$. The ramification group $G_i$ at $\infty$ has order $p^{\,n-j+1}$ for $l_{j-1}<i\le l_j$, so Hilbert's different formula \cite[IV, Prop.~4]{Se2} and Riemann--Hurwitz give
\[
2g(\Ubar_n)-2=-2p^n+\sum_{j=1}^n(p^{\,n-j+1}-1)(l_j-l_{j-1}),
\]
whence
\[
g(\Ubar_2)=\tfrac12p(p-1)^2,\qquad g(\Ubar_3)=7\ (p=2),\qquad g(\Ubar_4)=35\ (p=2),
\qquad g(\Ubar_3)=78\ (p=3).
\]

\subsection*{Weierstrass semigroup at $(2,3)$}
The functions $y_0$, $y_1$ and $y_2+y_0^2y_1$ are regular on $U_3$, with pole orders $4$, $6$ and $11$ at $\infty$: $\infty$ is totally ramified in $U_3/U_2$, so pole orders on $E_0$ double, and $(y_2+y_0^2y_1)^2+(y_2+y_0^2y_1)=y_0^4y_1+y_0^3y_1$ has pole order $11$ on $E_0$. The numerical semigroup $\langle4,6,11\rangle$ has exactly $7=g(\Ubar_3)$ gaps, namely $\{1,2,3,5,7,9,13\}$; since it is contained in the Weierstrass semigroup $H(\infty)$, which also has exactly $g(\Ubar_3)$ gaps, $H(\infty)=\langle4,6,11\rangle$. In particular the smallest odd pole order at $\infty$ of a function regular on $U_3$ is $11$.

\subsection*{Degrees realised on $\Ubar_n$}
Let $H=H(\infty)$ be the Weierstrass semigroup of $\infty$ on $\Ubar_n$. Riemann--Roch dimensions are invariant under base change, and $(X,\infty')\otimes_KL\cong(\Ubar_n,\infty)\otimes_kL$ (Lemma~\ref{lem:twistL}), so $H$ is also the Weierstrass semigroup of $\infty'$ on $X$. Every closed point of $X$ other than $\infty'$ is free (Proposition~\ref{prop:tail}(c)). Lemma~\ref{lem:hilbertfree}, which combines this with Hilbert irreducibility over the Hilbertian field $K$ \cite[Thm.~13.4.2]{FJ}, gives $D(U_n)\supseteq\{d\in H\smallsetminus\{0\}:p\nmid d\}$, since a function whose pole order is prime to $p$ has nonzero differential. At $(2,3)$ this is $D(U_3)\supseteq\langle4,6,11\rangle\smallsetminus\{0\}$ (Theorem~\ref{thm:status}); a closed point of degree $e\notin H$ is not a fibre of a function with poles only at $\infty'$, so the argument does not apply to the gaps $5,7,9,13$ of $H$. The degrees $5,7,9$ are excluded by Corollary~\ref{cor:B23} (no faithful $\Z/8$-curve has a free closed point of odd degree $<11$ on its twist); whether $13\in D(U_3)$ is Question~\ref{Q4}.





\begin{thebibliography}{99}

\bibitem{BLR90} S.~Bosch, W.~L\"utkebohmert and M.~Raynaud, N\'eron models, Ergeb. Math. Grenzgeb. (3) 21, Springer, 1990.

\bibitem{DR15} A.~Duncan and Z.~Reichstein, Versality of algebraic group actions and rational points on twisted varieties, J. Algebraic Geom. 24 (2015), 499--530.

\bibitem{FJ} M.~D.~Fried and M.~Jarden, Field arithmetic, 3rd ed., Ergeb. Math. Grenzgeb. (3) 11, Springer, 2008.

\bibitem{Ga} M.~Garuti, Linear systems attached to cyclic inertia, in: Arithmetic fundamental groups and noncommutative algebra, Proc. Sympos. Pure Math. 70, Amer. Math. Soc., 2002, 377--386.

\bibitem{EGA} A.~Grothendieck and J.~Dieudonn\'e, \'El\'ements de g\'eom\'etrie alg\'ebrique IV: \'Etude locale des sch\'emas et des morphismes de sch\'emas, troisi\`eme partie, Publ. Math. IH\'ES 28 (1966).

\bibitem{Ha} D.~Harbater, Moduli of $p$-covers of curves, Comm. Algebra 8 (1980), 1095--1122.

\bibitem{Hart77} R.~Hartshorne, Algebraic geometry, Grad. Texts in Math. 52, Springer, 1977.

\bibitem{Jou83} J.-P.~Jouanolou, Th\'eor\`emes de Bertini et applications, Progr. Math. 42, Birkh\"auser, Boston, 1983.

\bibitem{Ka} N.~M.~Katz, Local-to-global extensions of representations of fundamental groups, Ann. Inst. Fourier (Grenoble) 36 (1986), 69--106.

\bibitem{Kl05} S.~L.~Kleiman, The Picard scheme, in: Fundamental algebraic geometry, Math. Surveys Monogr. 123, Amer. Math. Soc., 2005, 235--321.

\bibitem{CHKZ} H.~Chu, S.-J.~Hu, M.~Kang and J.~Zhang, Groups with essential dimension one, Asian J. Math. 12 (2008), 177--192.

\bibitem{Ledet07} A.~Ledet, Finite groups of essential dimension one, J. Algebra 311 (2007), 31--37.

\bibitem{Sch37} H.~L.~Schmid, Zur Arithmetik der zyklischen $p$-K\"orper, J. reine angew. Math. 176 (1937), 161--167.

\bibitem{TV} D.~Tossici and A.~Vistoli, On the essential dimension of infinitesimal group schemes, Amer. J. Math. 135 (2013), no.~1, 103--114.

\bibitem{To} D.~Tossici, Essential dimension of infinitesimal commutative unipotent group schemes, Boll. Unione Mat. Ital. 12 (2019), no.~4, 575--581.

\bibitem{Fl08} M.~Florence, On the essential dimension of cyclic $p$-groups, Invent. Math. 171 (2008), 175--189.

\bibitem{BRV18} P.~Brosnan, Z.~Reichstein and A.~Vistoli, Essential dimension in mixed characteristic, Doc. Math. 23 (2018), 1587--1600.

\bibitem{Bry83} J.-L.~Brylinski, Th\'eorie du corps de classes de Kato et rev\^etements ab\'eliens de surfaces, Ann. Inst. Fourier (Grenoble) 33 (1983), no.~3, 23--38.

\bibitem{KS79} K.~Kanesaka and K.~Sekiguchi, Representation of Witt vectors by formal power series and its applications, Tokyo J. Math. 2 (1979), no.~2, 349--370.

\bibitem{EK25} G.~G.~Elder and K.~Keating, Artin--Schreier--Witt extensions and ramification breaks, arXiv:2503.16830 (2025).

\bibitem{Ledet04} A.~Ledet, On the essential dimension of $p$-groups, in: Galois theory and modular forms, Dev. Math. 11, Kluwer, Boston, 2004, 159--172.

\bibitem{Lip78} J.~Lipman, Desingularization of two-dimensional schemes, Ann. of Math. (2) 107 (1978), 151--207.

\bibitem{RS22} Z.~Reichstein and F.~Scavia, The behavior of essential dimension under specialization, \'Epijournal G\'eom. Alg\'ebrique 6 (2022), Art.~21.

\bibitem{RV18} Z.~Reichstein and A.~Vistoli, Essential dimension of finite groups in prime characteristic, C. R. Math. Acad. Sci. Paris 356 (2018), 463--467.

\bibitem{Se2} J.-P.~Serre, Local fields, Grad. Texts in Math. 67, Springer, 1979.

\bibitem{Silverman09} J.~H.~Silverman, The arithmetic of elliptic curves, 2nd ed., Grad. Texts in Math. 106, Springer, 2009.

\bibitem{St} H.~Stichtenoth, Algebraic function fields and codes, 2nd ed., Grad. Texts in Math. 254, Springer, 2009.

\bibitem{Th} L.~Thomas, Ramification groups in Artin--Schreier--Witt extensions, J. Th\'eor. Nombres Bordeaux 17 (2005), 689--720.

\bibitem{Witt37} E.~Witt, Zyklische K\"orper und Algebren der Charakteristik $p$ vom Grad $p^n$. Struktur diskret bewerteter perfekter K\"orper mit vollkommenem Restklassenk\"orper der Charakteristik $p$, J. Reine Angew. Math. 176 (1937), 126--140.
\bibitem{PA} A.~Shukla, The finite-degree profile of essential dimension I: general theory and characteristic prime to $|G|$, arXiv:2610.07508 (2026).

\end{thebibliography}
\end{document}